\documentclass[reqno]{amsart}

\makeindex
\usepackage{amsmath,amssymb,amsthm,url}
\usepackage{graphicx,color,case}
\usepackage[all]{xy}
\usepackage{txmac}
\usepackage{comment}

\definecolor{dark-green}{rgb}{0.4,0.6,0.2}

\definecolor{dark-blue}{rgb}{0.2,0.2,0.8}
\makeatletter
\@namedef{subjclassname@2020}{%
     \textup{2020} Mathematics Subject Classification}

\let\old@tl\~
\def\~{\raisebox{-0.8ex}{\tt\old@tl{}}}
\makeatother

\newtheorem{theorem}{Theorem}[section]  

\newtheorem{proposition-definition}[theorem]{Proposition-Definition} 
\newtheorem{lemma}[theorem]{Lemma} 
\newtheorem{claim}{Claim} 
\newtheorem{proposition}[theorem]{Proposition} 
\newtheorem{conjecture}[theorem]{Conjecture} 
\newtheorem{question}[theorem]{Question} 
\newtheorem{corollary}[theorem]{Corollary}
\theoremstyle{definition}
\newtheorem{definition}[theorem]{Definition}
\newtheorem{example}[theorem]{Example}
\newtheorem{remark}[theorem]{Remark}

\newtheorem*{remark*}{Remark}

\newcommand{\Z}{\mathbb{Z}}
\newcommand{\R}{\mathbb{R}}
\newcommand{\C}{\mathbb{C}}

 \makeatletter
    
   \@addtoreset{equation}{section}
  \makeatother

\makeatletter
\def\Mcf{\max{\operator@font cf}\,}
\def\mcf{\min{\operator@font cf}\,}
\let\Dl\Delta
\def\bt{\bar t'_2}
\def\ssim{\stackrel{\displaystyle \sim}{\vbox{\vskip-0.2em\hbox{$\scriptstyle
*$}}}}
\def\myfrac#1#2{\raisebox{0.2em}{\small$#1$}\!/\!\raisebox{-0.2em}{\small$#2$}}

\def\ffrac#1#2{\mbox{\small$\ds\frac{#1}{#2}$}}

\let\iy\infty
\let\ds\displaystyle

\let\dt\det
\def\md{\min\deg}
\def\Md{\max\deg}
\def\nb{\nabla}
\let\sg\sigma

\let\dl\delta
\let\Lm\Lambda
\def\mb{\discretionary {}{}{}}
\let\eps\varepsilon

\let\Gm\Gamma

\let\q\quad

\begin{document}

\title{Invariants of Weakly successively almost positive links}

\newcommand{\Crossing}{
\raisebox{-3mm}{
\begin{picture}(24,28)
\put(0,0){\line(1,1){10}}
\put(0,24){\vector(1,-1){24}}
\put(14,14){\vector(1,1){10}}
\end{picture} } 
}

\newcommand{\Smooth}{
\raisebox{-3mm}{
\begin{picture}(24,28)
\qbezier(0,0)(12,14)(24,0)
\qbezier(0,24)(12,10)(24,24)
\end{picture}} 
}

\newcommand{\LCross}{
\raisebox{-3mm}{
\begin{picture}(24,28)
\put(0,0){\line(1,1){24}}
\put(0,24){\line(1,-1){24}}
\end{picture}}
}

\makeatletter
\let\@wraptoccontribs\wraptoccontribs
\makeatother
\author{Tetsuya Ito}
 
\address{Department of Mathematics, Kyoto University, Kyoto 606-8502,
JAPAN}
\email{tetitoh@math.kyoto-u.ac.jp}

\author{Alexander Stoimenow}
\address{School of Computing, Complexity and Real
Computation Lab, KAIST, Daejeon 34141, Korea}
\email{stoimeno@stoimenov.net}


\begin{abstract}
As an extension of positive and almost positive diagrams and links, we
study two classes of links we call successively almost positive and weakly successively almost positive links.

We prove various properties of polynomial invariants and signatures of such links, extending previous results or answering open questions about positive or almost positive links.
We discuss their minimal genus and fibering property and for the latter prove a fibering extension of Scharlemann-Thompson's theorem (valid for general links).

\end{abstract}
\maketitle

\tableofcontents

\section{Introduction\label{sec:intro}}

\subsection{Background and motivation}

A link diagram is \emph{positive}\index{diagram! positive} if all the
crossings are positive, and a \emph{positive link} is a link that can
be represented by a positive diagram. Positive links have various nice
properties and form an important class of links.

A suggestive generalization of a positive diagram is a \emph{$k$-almost
positive diagram}\index{diagram! $k$-almost positive}, a diagram such
that all but $k$ crossings are positive. A $1$-almost positive diagram
is usually called an \emph{almost positive diagram}\index{diagram! almost
positive} and has been studied in various places.

It is known that almost positive links share various properties with positive
links. However, although there are some special properties of $2$- or
$3$-almost positive links as discussed in \cite{Przytycki-Taniyama}, when
$k$ is large, $k$-almost positive links fail to have nice properties similar
to positive links because every knot $K$ is $k$-almost positive for sufficiently
large $k$. Indeed, as we will frequently see (Example \ref{exam:2-ap-Conway})
even for $2$-almost positive knots, almost all properties of positive
links fail.

The aim of this paper is to propose and investigate natural generalizations
of a(n almost) positive diagram.

We introduce various classes of diagrams which are generalizations of (almost)
positive diagrams (see Section \ref{sec:summary-positivity} for a concise
summary of the definitions of these generalizations).
Unlike almost positive diagrams, the class we introduce can have arbitrary
many negative crossings. Nevertheless, these diagrams share various properties
with (almost) positive links. 

Among them, \emph{weakly successively almost positive links}, the central
object which we study in this paper, provide a unified and satisfactory
framework to treat the positivity of diagrams. 
One crucial feature of this class of diagrams is that it is closed under the
skein resolution in a suitable way (Theorem \ref{thm:skein-wsap})
that allows to use an induction argument.

We will prove various properties of (weakly) successively almost positive
links 
 which extend known results of (almost) positive links. Frequently our
results are new even for the positive links, or answer questions about
(almost) positive links that appeared in the literature. (For instance, see
Corollary \ref{cor:conway-positive-knot} or Corollary \ref{cor:HOMFLY-ap}.)
Thus weakly successively almost positive diagrams/links provide a better
framework to study positive diagrams.

\subsection{Summaries of various positivities\label{sec:summary-positivity}}

To begin with, for the reader's convenience, we list the definitions of various
positivities discussed or studied in the paper. 
\begin{definition}
Let $D$ be a link diagram $D$ and $k \in\Z_{\geq 0}$ be a non-negative
integer. 
\begin{itemize}
\item $D$ is \emph{positive}\index{diagram! positive} if all the crossings
of $D$ are positive.

\item $D$ is \emph{almost positive} if all but $1$ crossing of $D$ are
positive. Almost positive diagrams are divided into the following two
types.
\begin{itemize}
\item We say that $D$ is \emph{of type I} if the negative crossing $c$
of $D$ connects two Seifert circles $s$ and $s'$, so that there is no
other (positive) crossing connecting $s$ and $s'$.
\item Otherwise, we say that $D$ is \emph{of type II}.
\end{itemize}
\item $D$ is \emph{successively $k$-almost positive} if all but $k$ crossings
of $D$ are positive, and the $k$ negative crossings appear successively
along a single overarc (see Definition \ref{def:sap}).
\item $D$ is \emph{good successively $k$-almost positive} if $D$ is successively
$k$-almost positive and for a negative crossing $c$ of $D$ that connects
two Seifert circles $s$ and $s'$, there are no other crossing connecting
$s$ and $s'$ (see Definition \ref{def:good-sap}).

\item $D$ is \emph{weakly successively $k$-almost positive}
if all but $k$ crossings of $D$ are positive, and the $k$ negative crossings
appear (but not necessarily consecutively) along a single overarc which
we call \emph{negative overarc} (see Definition \ref{def:wsap} and Figure
\ref{fig:weak-succ-almost-positive}).

\item $D$ is \emph{weakly positive} if $D$ can be viewed as a descending
diagram except at positive crossings; when we walk along $D$, we pass
every negative crossing first (see Definition \ref{def:weak-positive}),
through its overpass.
\end{itemize}
\end{definition}

In the following we use abbreviations in the paper to save space.
\begin{itemize}
\item s.a.p.: successively almost positive.
\item $k$-s.a.p.: successively $k$-almost positive.
\item w.s.a.p.: weakly successively almost positive.
\item $k$-w.s.a.p.: weakly successively $k$-almost positive.
\end{itemize}

A link $L$ is \emph{positive} if $L$ is represented by a positive diagram.
An \emph{almost positive link, successively almost positive link,\dots}
is defined by the same manner. 

There are other types of positivities of links which are defined by 
braid representatives, not diagrams.
\begin{definition}
A link $L$ is 
\begin{itemize}
\item \emph{braid positive}\index{braid positive} (or, a \emph{positive
braid link}) if $L$ is represented by a closure of a a positive braid.
\item \emph{strongly quasipositive}\index{link! strongly quasipositive}\index{strongly
quasipositive! link} if $L$ is represented by a closure of a strongly
quasipositive braid\index{strongly quasipositive! braid}, a product of
braids $\{(\sigma_i\sigma_{i+1} \cdots \sigma_{j-2})\sigma_{j-1}(\sigma_i\sigma_{i+1}
\cdots \sigma_{j-2})^{-1} \: | \: 1\leq i<j \leq n\}$ (see Definition
\ref{def:SQP}).
\item \emph{quasipositive}\index{link! quasipositive link} if $L$ is represented
by a closure of a quasiposisitve braid\index{quasipositive braid}, a product
of braids of the form $\{\alpha \sigma_1 \alpha^{-1} \: | \: \alpha \in
B_n\}$.
\end{itemize}
\end{definition}

In our investigations, the following property also plays a fundamental
role (see Section \ref{sec:self-linking-number} for details and background).
\begin{definition}
A link $L$ is \emph{Bennequin-sharp} if the Bennequin inequality is an equality.
\end{definition}

The positive and almost positive links have appeared and have been studied in past
publications and enjoyed extensive treatment (for example, \cite{Cromwell,st-positive-polynomial}).
 
The successively almost positive links, and the good successively almost
positive links were introduced and studied in \cite{Ito-sap}. The current
work grew out of the investigation of natural problems raised in \cite{Ito-sap}.

The weakly successively almost positive links are the main objects which
we introduce and study in the paper. This notion has occurred (without
extra prominence lent) in Cromwell's construction of positive skein trees
(\cite[Theorem 2]{Cromwell}), and is a natural generalization of successively
almost positive links. 

The weakly positive diagrams/links is the largest class of links that
enjoys certain positivity phenomena. 

Positive braid links are the natural positive object in the braid group
point of view. Strongly quasipositive links and quasipositive links have
their origins in Rulodph's works \cite{Rudolph-clink},\cite{Rudolph-braided}
and although it is not clear from the above definition, they are more
related to the positivity of Seifert surfaces. Recently these positivity
notions gather much attention due to close connection to the contact topology,
like the Bennequin-sharp property.

The relations among these positivities are summarized in the following
diagram\footnote{\label{fn:convention-almost-positive}
In some places, like \cite{st-positive-polynomial}, `almost positive'
is
assumed, unlike here, to exclude `positive'. Also, in older works
like \cite{Buskirk}, `positive link' is used for what is called
`positive braid link' here; \em{cf.} Example \ref{exam:vB}.}.

\begin{figure}[htbp]
\let\snq\subsetneq
\begin{picture}(260,220)
\put(80,200){\{Positive braid\}}
\put(120,185){\rotatebox[origin=c]{-90}{\LARGE $\snq$}}
\put(100,170){\{Positive\}}
\put(120,155){\rotatebox[origin=c]{-90}{\LARGE $\snq$}}
\put(60,140){\{Almost positive of type I\}}
\put(150,125){\rotatebox[origin=c]{-90}{\LARGE $\snq$}}
\put(80,125){\rotatebox[origin=c]{-90}{\LARGE $\snq$}}
\put(20,110){\{Almost positive\}}
\put(120,110){\{Good successively almost positive\}}
\put(30,95){\rotatebox[origin=c]{-90}{\LARGE $\snq$}}
\put(180,95){\rotatebox[origin=c]{-90}{\LARGE $\snq$}}
\put(90,95){\rotatebox[origin=c]{-45}{\LARGE $\snq$}}
\put(110,95){\rotatebox[origin=c]{-135}{\LARGE $\snq$}}
\put(-10,80){\{Strongly quasipositive\}}
\put(120,80){\{Successively almost positive\}}
\put(0,65){\rotatebox[origin=c]{-90}{\LARGE $\snq$}}
\put(60,65){\rotatebox[origin=c]{-90}{\LARGE $\subset$}}
\put(180,65){\rotatebox[origin=c]{-90}{\LARGE $\subset$}}
\put(-50,50){\{Quasipositive\}}
\put(30,50){\{Bennequin-sharp\}}
\put(130,50){\{Weakly successively almost positive\}}
\put(180,35){\rotatebox[origin=c]{-90}{\LARGE $\snq$}}
\put(150,20){\{Weakly positive\}}
\end{picture}
\caption{Summary of relations of various positivity notions} 
\end{figure}
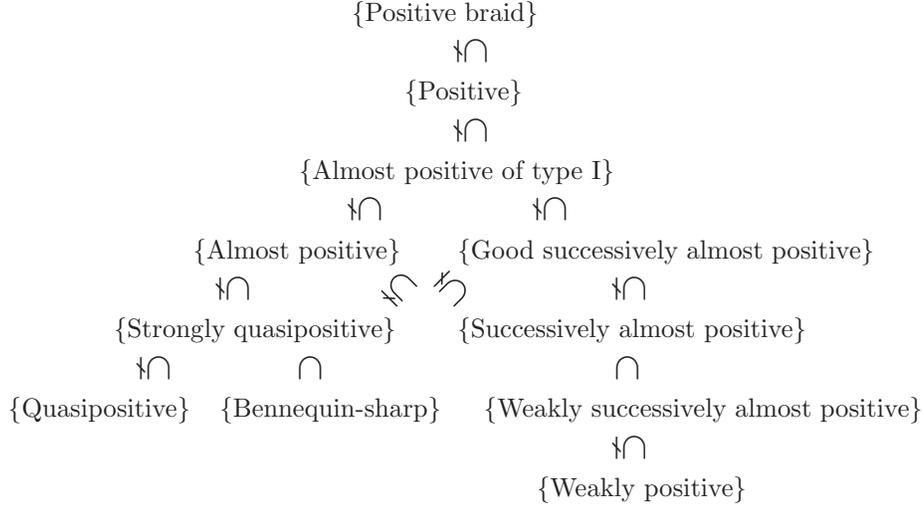

For the less obvious inclusions, compare the explanation below Definition
\ref{def:good-sap}.

We will show in Section \ref{sec:comparisons} that all of the inclusions in the above chart which are indicated
so are strict, so that (weakly) successively almost positive links form
a strictly new class of links. 
(See also Question \ref{ques:SQP}, and Question \ref{ques:strictness},
but also Remark \ref{rem:inclusion-arbitrary-k}.)

\subsection{Organization of the paper and summary of results}

In Section \ref{sec:preparation} we set up notation and terminology, and
review various standard facts in knot theory which will be used in the
paper.

Section \ref{sec:sap} provides a quick review of a successively positive
link, the content of \cite{Ito-sap}. The content of Section \ref{sec:sap}
gives a motivation and prototype of various results proven in the paper.

In Section \ref{sec:wp} we introduce a weakly positive diagram.
It is crucial that a weakly positive diagram `obviously' tells us whether
it represents a split link or not (Theorem \ref{thm:height-split}). This
special feature leads to several properties of the Conway polynomial, which
will be frequently used in the rest of the paper,

Our main object, a weakly successively almost positive diagram, is introduced
in Section \ref{sec:wsap}.
From its definition, it is clear that it admits a \emph{standard skein
triple}, a skein triple that consists of weakly successively almost positive diagrams
reducing its complexity (Theorem \ref{thm:skein-wsap}).
Using the fact that a weakly successively almost positive diagram is weakly
positive, we can pin down when a split link appears in the standard skein
triple (Lemma \ref{lem:skein-split}). Also, we have a standard unknotting/unlinking
sequence, a sequence of crossing changes that converts a w.s.a.p.\ link
into the unlink. These sequences
play a fundamental role in our study of weakly successively almost positive
links.

Section \ref{sec:Conway} establishes various illuminating properties of
the Conway polynomial of weakly successively almost positive links (Theorem
\ref{thm:Conway-polynomial}, Proposition \ref{prop:u-Conway}). They are
far-reaching generalizations of properties of the Conway polynomials of positive
links and contain a new result even for positive links.

Section \ref{sec:HOMFLY} is devoted to the HOMFLY polynomial of weakly
successively almost positive links (Theorem \ref{thm:HOMFLY-polynomial}).
We will also discuss the properties of Jones polynomials (Theorem \ref{thm:Jones-polynomial}).
As we will frequently mention in various examples, the property of the HOMFLY
polynomial (Theorem \ref{thm:HOMFLY-polynomial}) is quite strong to show
that a given knot is not weakly successively almost positive.

Section \ref{sec:Bennequin-sharp} provides an enhancement of results in
the previous two sections (Theorem \ref{thm:polynomial-B-sharp}), under the
additional assumption that $L$ is \emph{Bennequin-sharp}. We demonstrate,
under this additional assumption, the Conway and HOMFLY polynomial of weakly
successively almost positive links have more special properties.

They motivate the question how to exhibit w.s.a.p.\ links as Bennequin-sharp.
This is a separate long topic that has to be moved out to a follow-up
paper \cite{Part2}. There we will give extensive treatment of
Seifert surfaces and also explain the relation to strong quasipositivity.

Section \ref{sec:signature-positive} is devoted to the positivity of the signature;
we prove that a non-trivial weakly successively almost positive link has
\emph{strictly} positive signature (Theorem \ref{thm:signature>0}), generalizing
the (almost) positive link case. We discuss various applications of the positivity
of the signature. Among them, we present several delicate examples of knots
which are not weakly successively almost positive, but its knot polynomials
share the same properties as we have proved in Section \ref{sec:Conway},
Section \ref{sec:HOMFLY}.

Section \ref{sec:signature-general} establishes a more general signature
inequality.
We prove a signature estimate from a general link diagram (Theorem \ref{thm:signature-improved}),
which improves the signature estimate given by Baader-Dehornoy-Liechti
\cite{Baader-Dehornoy-Liechti}. As an application, we prove that every
algebraic knot concordance class contains only finitely many weakly successively
$k$-almost positive links (Theorem \ref{cor:finite-concordance}).

Section \ref{sec:comparisons} establishes the strictness of inclusions
for various classes of positivities (Theorem \ref{thm:class-inclusions}).
Since our results says that weakly successively almost positive links
and (almost) positive links share many properties, it is not surprising
that showing a given weakly successively almost positive link is not (almost)
positive is subtle.

Section \ref{sec:question} gathers various questions for (weakly) successively
almost positive links.
 
The paper contains two appendices. In \ref{sec:ST-theorem} we prove Theorem
\ref{thm:st-fibered}, a slight enhancement of Scharlemann-Thompson's theorem
of Euler characteristics of skein triples that tells us the fiberedness property,
which was used in the proof of Theorem \ref{thm:Conway-polynomial}.
We review some background on sutured manifold theory and the proof of Scharlemann-Thompson's
theorem (Theorem \ref{thm:ST}), then we explain how to get the fibration
information.

In \ref{sec:Knot-table} we determine which knots are (weakly) successively
almost positive, for all prime knots to 12 crossings, with four exceptions.

\section{Preparation\label{sec:preparation}}

In this section we prepare various notions which will be used in the paper.
We summarize our notation, conventions
and terminologies.

In the following, we usually assume that a link diagram $D$ is always
oriented.
We will usually regard a diagram $D$ is contained in $\R^{2}$, not $S^{2}$,
although we will sometimes utilize an isotopy of diagram in $S^{2}$ to
make the situation simple.

We denote by $c(D)$\index{$c(D)$} the number of crossings of $D$, and
we denote by $c_{+}(D)$\index{$c_+(D)$} (resp.\ $c_-(D)$\index{$c_-(D)$})
the number of positive (resp.\ negative) crossings.
We write $w(D)=c_+(D)-c_-(D)$ for the \em{writhe} of $D$.

An inequality $a\ge
b$ is \em{exact} if $a=b$, and \em{strict} otherwise.

\subsection{Primeness and splitness\label{sec:split}}

For a link $L$, we denote by $\#L$\index{$\#L$} the number of its components.

We use $L_1\# L_2$, as a binary operation, for the \em{connected sum}
of the links $L_1$ and $L_2$ (although when $\# L \geq 2$, it is not uniquely
determined).
We write $\#^n L$ for the ($n-1$-fold iterated) connected sum of $n$ copies
of a link $L$. Thus $\#^1 L=L$ (which is not to be confused with the integer
$\# L$). By convention, we always regard $\#^0 L$ as the unknot, regardless
the number of components.

We say that a link $L$ is \em{prime}\index{prime link} if, whenever $L=K_1\#
K_2$, then
(exactly) one of $K_1$ and $K_2$ is an unknot. Another way of saying this
is that
a link is prime if every $S^2\subset S^3$ which intersects the link in
two points (transversely), leaves an unknotted arc on (exactly) one side
of the $S^2$.

We say a link is \em{split}\index{split link} if there is some $S^2\subset
S^3$
which intersects the link nowhere, but leaves parts of the link on either
side of the $S^2$. We call such $S^2$ a \em{splitting sphere}.

We say two components $K_1,K_2$ of a link $L$ are \em{inseparable}, if
there is no splitting sphere of $L$ containing $K_1$, $K_2$ on opposite
sides.
The equivalence classes of components of $L$ with respect to the inseparability
relation are the \em{split components}\index{split components}
of $L$. We write 
\[
L=L_1\sqcup L_2\sqcup\cdots\sqcup L_m
\]
as the \em{split union} of its split components $L_i$.

Such split components may of course in general contain several components
of $L$.
We say that $L$ is a \em{totally split} link, if each split component
of $L$ contains exactly one component of $L$.

The totally split link whose $n$ (split) components are unknots is the
($n$-component) \em{unlink}, and will be denoted $U_n$\index{$U_n$}. For
simplicity we write $U_1=\bigcirc$ for the unknot.

\subsection{Seifert's algorithm and Seifert graph\label{sec:Seifert}}

In this section we review Seifert's algorithm and related notions.

For an oriented link diagram $D$ and a crossing $c$ of $D$, the \emph{smoothing}\index{smoothing}
at $c$ is a diagram $D$ obtained by replacing the crossing $c$ as 
\[
\raisebox{-2.5mm}{
\begin{picture}(24,24)
\put(10,14){\vector(-1,1){10}}
\put(24,0){\line(-1,1){10}}
\put(0,0){\vector(1,1){24}}
\end{picture} } 
\longrightarrow
\raisebox{-2.5mm}{
\begin{picture}(24,24)
\qbezier(0,0)(14,12)(0,24)
\qbezier(24,0)(10,12)(24,24)
\put(24,24){\vector(1,1){0}}
\put(0,24){\vector(-1,1){0}}
\end{picture} } 
\longleftarrow
\raisebox{-2.5mm}{
\begin{picture}(24,24)
\put(0,0){\line(1,1){10}}
\put(14,14){\vector(1,1){10}}
\put(24,0){\vector(-1,1){24}}
\end{picture} } 
\]
according to its sign. 

A triple of diagrams $(D_+,D_0,D_-)$ equal except as designated above,
or the triple of their corresponding links $(L_+,L_0,L_-)$, is called
a \emph{skein triple}\index{skein triple}.

By smoothing all the crossings of $D$, we get a disjoint union of circles.
A \emph{Seifert circle}\index{Seifert circle} of $D$ is its connected
component.  We denote the number of Seifert circles by $s(D)$\index{$s(D)$}.

Each Seifert circle $s$ separates the projection plane $\R^{2}$ into two
components of $\R^{2} \setminus s$. One is compact, which we call \emph{interior}
of $s$ , and the other is non-compact, which we call its \em{exterior}.
For Seifert circles $s$ and $s'$, we say that $s$ is \emph{contained}\index{Seifert
circle! contained} in $s'$ if $s$ belongs to the \em{interior} of $s$.

A Seifert circle $s$ is \emph{separating}\index{Seifert circle! separating}
if both its interior and exterior contain other Seifert circles. 
  
For each Seifert circle $s$, we take a disk $d_s$ in $\R^{3}$ bounded by Seifert circles having constant $z$-coordinate (height) $h_s$. We choose $h_s$ so that they are pairwise distinct and that 
\begin{equation}
h_s > h_{s'} \mbox{ whenever } s \mbox{ is contained in } s'.
\end{equation}
Then we connect these disks by attaching twisted bands at the  crossing to get Seifert surface $S_D$ of $L$.

\begin{definition}[Canonical Seifert surface, canonical genus]
The Seifert surface $S_D$ is called the \emph{canonical Seifert surface}\index
{canonical Seifert surface} of $D$.
The \emph{canonical Euler characteristic}\index{canonical Euler characteristic}
$\chi(D)$\index{$\chi(D)$} of $D$ is defined by
\[ \chi(D)=\chi(S_D)=s(D)-c(D).\]
The  \emph{canonical genus}\index{canonical genus}\index{canonical genus!
of diagram} of a diagram $D$ is defined by
\[ g(D)= \frac{-\chi(D)-\#L+2}{2}\]

\end{definition}

In the following, we will often identify the crossing $c$ with a twisted
band in the canonical Seifert surface. For example, we say that a crossing
$c$ \emph{connects Seifert circles $s$ and $s'$} if the twisted band of
$S_D$ that corresponds to $c$ connects the disks bounded by $s$ and $s'$.

To encode the structure of a diagram and its canonical Seifert surface, we
will use the following graph.

\begin{definition}\label{def:Seifert-graph}
The \emph{Seifert graph}\index{Seifert graph} $\Gamma(D)$\index{$\Gamma(D)$}
of a diagram $D$ is bipartite graph without self-loop,  whose set of vertices
of $\Gamma(D)$ is the set of Seifert circles, and whose set of edges is
the set of crossings. A crossing $c$ connecting Seifert circles $s$ and
$s'$ corresponds to an edge of $\Gamma(D)$ connecting the corresponding
vertices. 
We often view $\Gamma(D)$ as a signed graph. Each edge carries a sign
$\pm$ according to the sign of the
crossing.

\end{definition}

To analyse the properties of (Seifert) graphs, we will use the following
terminology.

\def\I{\item[$\diamond$]}
\begin{itemize}
\I An edge $e$ is \emph{singular}\index{singular edge} if there is no
other edge connecting the same vertices as $e$.
\I An edge $e$ is called an \emph{isthmus}\index{isthmus} if removing
the edge $e$ makes the graph disconnected. 
 
\I
Let $v_1$ (resp.\ $v_2$) be a vertex of a graph $\Gamma_1$ (resp.\ $\Gamma_2$),
The \emph{one-point join}\index{one-point join} $\Gamma_1 \ast_{v_1=v_2}\Gamma_2$
is a graph obtained by identifying $v_1$ and $v_2$. 
 
Although the choice of $v_1,v_2$ is obviously necessary, we often write
$\Gm_1*\Gm_2$ for the result of this operation under \em{some} choice
of $v_1$ and $v_2$. (In this loose sense, $*$ becomes both a commutative
and associative operation.)

\end{itemize}

\subsection{Diagrams and crossings\label{sec:diagrams}}

In this section we recall terminology related to diagrams and crossings.

A \emph{region}\index{region} of a link diagram $D$ is a connected component
of the complement of $\R^{2} \setminus D$ (or, $S^{2} \setminus D$). 
Every crossing $c$ of $D$ is adjacent to four regions. To illustrate their
relations we use the following terms.

\begin{definition}
\label{def:region}
We say that a region $R$ around the crossing $c$ is a \emph{Seifert circle
region}\index{region! Seifert circle region} near $c$ if, near the crossing
$c$ the regions contains pieces of Seifert circles near $c$.
We say that regions $R$ and $R'$ around the crossing $c$ are
\begin{itemize}
\item[--] \em{opposite}\index{region! opposite} (at $c$) if they do not
share an edge.  
\item[--] \em{neighbored}\index{region! neighbored} (at $c$) if they share
an edge near $c$.
\end{itemize}
\end{definition}

Similarly, to describe how the crossing connects two Seifert circles,
we use the following notions.

\begin{definition}\label{def:crossing-equivalence}
Let $c,c'$ be crossings of a diagram $D$.
\begin{itemize}
\item $c$ and $c'$ are \em{Seifert equivalent}\index{crossing! Seifert
equivalent} if they connect the same pair of (distinct) Seifert circles.
\item $c$ and $c'$ are \em{twist equivalent}\index{crossing! twist equivalent}
if one of their pairs of opposite regions coincides.
\item $c$ and $c'$ are \em{$\sim$-equivalent}\index{crossing! $\sim$-equivalent}
if 
their pairs of opposite non-Seifert circle regions coincide.
\item $c$ and $c'$ are \em{$\ssim $-equivalent}\index{crossing! $\ssim$-equivalent}
if their the pairs of opposite Seifert circle regions coincide.
\end{itemize} 
\end{definition}

By definition, $\ssim$-equivalent implies Seifert equivalent. The converse
is not true in general, because Seifert circles do not border unique regions.
However, if the diagram is special (see below definition), the converse is true.

Finally we summarize various notions of diagrams which we be used or discussed.

\begin{definition} 
We say that a link diagram $D$ is
\begin{itemize}
\item[-] \emph{reduced}\index{diagram! reduced} if it contains no nugatory
(reducib;e) crossings. Here a crossing $c$ of $D$ is \emph{nugatory (reducible)}\index{crossing!
nugatory} if there exists a circle $c$ in $\R^{2}$ that transversely intersects
with $D$ only at $c$.
\item[-] \emph{special}\index{diagram! special} if it contains no separating
Seifert circles.
\item[-] \emph{split}\index{diagram! split} if it is disconnected, as
a subspace in $\R^{2}$.
\item[-] \emph{non-prime (composite)} if there is a simple closed curve
$c$ that transversely intersects with $c$, so that for the connected components
$U, V$ of $\R^{2} \setminus c$, both $U\cap D$ and $V \cap D$ are not
embedded arcs. Otherwise, we say that $D$ is \emph{prime}.
\end{itemize}
\end{definition}

\subsection{Genera of links\label{sec:genus}}

The \emph{genus} and \emph{the slice genus (smooth 4-ball genus)} of a
link $L$ is defined by 
\begin{equation}\label{eqn:def-genus}
g(L)=\frac{-\chi(L)-\#L+2}{2}, \quad g_4(L)=\frac{-\chi_4(L)-\#L+2}{2}
\end{equation}
respectively.
Here $\chi(L)$ is the maximum Euler characteristic of Seifert surface
of $L$, and $\chi_4(L)$ is the maximum Euler characteristic of 4-ball
surface, a smoothly embedded surface in $B^{4}$ whose boundary is $L$.
 
\begin{remark}
The definition of genus assumes that a maximal Euler characteristic Seifert surface or 4-ball
surface is connected. As we will see later, this implicit assumption is
satisfied whenever $L$ is a non-split weakly positive link, because of
linking number reasons (see Corollary \ref{cor:splitness-wp}).

But regardless of the connectivity issue, for now we use $g_4$ and $g$
just as quantities defined as above. Certainly when $L$ is a knot, a surface
is connected, so that the (4-ball) genus is correctly reflected.
\end{remark}

As a quantity related to diagrams, we use the following.
\begin{definition}
\label{def:canonical-genus}
The \emph{canonical genus}\index{canonical genus! of link} of a link $L$
is defined by
\[ g_c(L)=\min \{g(D) \: | \: D \mbox{ is a diagram of }L\}\]
\end{definition}
By convention, for a diagram $D$ of $L$ we put $g_4(D)=g_4(L)$ and $\chi_4(D)=\chi_4(L)$.

Obviously for every link $L$, the inequality
\[
g_4(L)\le g(L)\le g_c(L)\,.
\]
holds.

\subsection{Murasugi sum and diagram Murasugi sum\label{sec:Murasugi-sum}}

The Murasugi sum is an operation on surfaces that generalizes the connected
sum, and has very nice features, as indicated in the title `the Murasugi
sum is a natural geometric operation' of the papers
by Gabai \cite{Gabai-MurasugisumI,Gabai-MurasugisumII}.

\begin{definition}[Murasugi sum (of surfaces in $S^{3}$)]
Let $R_1$ and $R_2$ be oriented surfaces with boundary. An oriented surface
$R$ in $S^{3}$ is a \emph{Murasugi sum}\index{Murasugi sum! of surfaces}
of $R_1$ and $R_2$ if there is a 2-sphere $S\subset S^{3}$ that separates
$S^{3}$ into two 3-balls $B_1$ and $B_2$, such that 
\begin{enumerate}
\item[(i)] $R_1 \subset B_1$, $R_2 \subset B_2$.
\item[(ii)] $D:=R_1 \cap S = R_2 \cap S$ is a $2n$-gon. 
\end{enumerate} 
We often say that $R$ is a Murasugi sum of $R_1$ and $R_2$ along the $2n$-gon
$D$.
\end{definition}

We use the following diagram version of Murasugi sum which is often called
the $\ast$-product\index{$\ast$-product}.

\begin{definition}[Diagram Murasugi sum]
Let $D_1$ be a (planar, in $\R^2$) link diagram and $s_1$ be an innermost
Seifert circle of $D_1$; this means $s_1$ bounds a disk whose interior
contains no other Seifert circles of $D_1$. Similarly, let $D_2$ be a
link diagram and $s_2$ be an outermost Seifert circle of $D_2$ (i.e.,
whose exterior does not contain any other Seifert circle\footnote{Note
that for some planar diagrams $D_2$, such a Seifert circle $s_2$ may not
exist.}). 

The \emph{diagram Murasugi sum}\index{Murasugi sum! of diagrams} (or,
$\ast$-product) $D_1\ast D_2$ is a link diagram obtained by identifying
$s_1$ and $s_2$. See Figure \ref{fig:DMurasugi-sum}. We often write $D_1
\ast_{s_1 = s_2} D_2$ to indicate the identified Seifert circles. The
resulting Seifert circle is
separating.
\end{definition}

The diagram Murasugi sum is not uniquely determined by $(D_1,s_1)$ and
$(D_2,s_2)$ since there is a freedom to choose how the crossings connected
to $s_1$ and $s_2$ align along the identified Seifert circle $s_1=s_2$.
For the canonical Seifert surface $S_{D_1 \ast D_2}$ of $D_1 \ast D_2$
is the Murasugi sum of the canonical Seifert surfaces $S_{D_1}$ and $S_{D_2}$
along the disks bounded by $s_1$ and $s_2$.

\begin{figure}[htbp]
\includegraphics*[width=90mm]{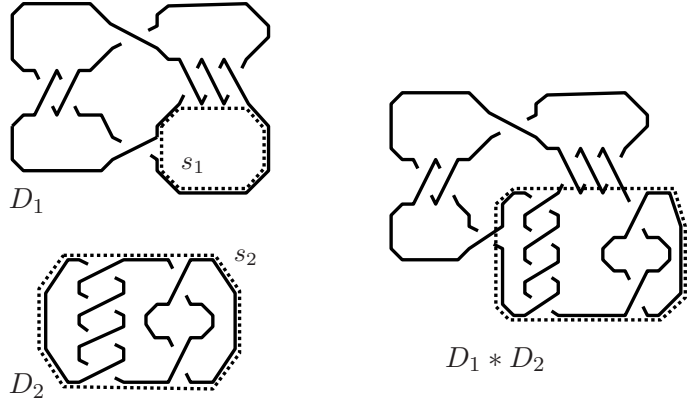}
\begin{picture}(0,0)
\put(-260,70) {\large $D_1$}
\put(-260,0) {\large $D_2$}
\put(-195,85) {$s_1$}
\put(-175,50) {$s_2$}
\put(-95,10) {\large $D_1 \ast D_2$}
\end{picture}
\smallskip
\caption{Diagram Murasugi sum of $D_1$ and $D_2$ along their Seifert circles
$s_1$, $s_2$.} 
\label{fig:DMurasugi-sum}
\end{figure}

On the other hand, on the level of Seifert graphs, the diagram Murasugi
sum is always well-defined, and corresponds to the one-point join:
\[ \Gamma(D_1 \ast_{s_1 = s_2} D_2) = \Gamma(D_1) \ast_{s_1=s_2}\Gamma(D_2).
\]

The canonical Seifert surface $S_{D_1 \ast D_2}$ of $D_1 \ast D_2$ is
a Murasugi sum of the canonical Seifert surfaces $S_{D_1}$ and $S_{D_2}$.
The (diagram) Murasugi sum has the following properties. 

\begin{theorem}\label{thm:Murasugi-sum}
Let $K_1, K_2, K_1\ast K_2$ be the links represented by $D_1,D_2,D_1 \ast
D_2$, respectively.
\begin{itemize}
\item[(i)] $S_{D_1\ast D_2}$ attains the maximum Euler characteristic,
if and only if both $S_{D_1}$ and $S_{D_2}$ attain the maximum Euler characteristic
\cite{Gabai-MurasugisumI,Gabai-MurasugisumII}.
\item[(ii)] $S_{D_1\ast D_2}$ is a fiber surface if and only if both $S_{D_1}$
and $S_{D_2}$ are fiber surfaces \cite{Gabai-MurasugisumI,Gabai-MurasugisumII}.
\item[(iii)] Let $p_{n}(K) \in \Z[v,v^{-1}]$ be the coefficient of $z^{n}$
of the HOMFLY polynomial $P_K(v,z)$. Then 
\[ p_{1-\chi(D_1\ast D_2)}(K_1\ast K_2) = p_{1-\chi(D_1)}(K_1) \cdot p_{1-\chi(D_2)}(K_2).\]
In particular, 
\[
a_{1-\chi(D_1\ast D_2)}(K_1\ast K_2) = a_{1-\chi(D_1)}(K_1) \cdot a_{1-\chi(D_2)}(K_2)\,,
\]
where $a_{n}(K)$ is the coefficient of $z^{n}$ in $\nb$ \cite{Murasugi-Przytycki-skein}.
(Compare Section \ref{sec:polynomial-invariants}.)
\end{itemize}
\end{theorem}

Thus it is often useful to decompose $D$ into diagram Murasugi sums.

A \emph{homogeneous diagram}\index{diagram! homogeneous} is a diagram
which is a Murasugi sum of special alternating diagrams. A \emph{homogeneous
link}\index{homogeneous link} introduced in \cite{Cromwell} is a link
that is represented by a homogeneous diagram. Homogeneous links are a common
generalization of alternating and positive links.

\subsection{Gauss diagram}

Let $I$ be a finite totally ordered set. When $\#I = \ell$, we usually
identify $I$ with the set $\{1,\ldots, \ell\}$ with the standard ordering
$1<2<\cdots <\ell$.

Let $\bigsqcup_{i \in I} (S^{1}_{i},\ast_i)$ be the disjoint union of
based oriented circles $(S^1_i,\ast_i)$ indexed by $I$.

\begin{itemize}
\item A \emph{chord} is an unordered pair $\{u,v\}$ of distinct points
of $\bigsqcup_{i \in I} (S^{1}_{i},\ast_i)$.
\item An \emph{(unsigned) arrow} is an ordered pair $\vec{a}=(o,u)$ of
distinct points of $\bigsqcup_{i \in I} (S^{1}_{i},\ast_i)$. We call $o$
and $u$ the \emph{arrow tail} and the \emph{arrow head}, respectively.
\item A \emph{signed arrow} $(\vec{a},\varepsilon)$ is an arrow $\vec{a}=(o,u)$
equipped with a sign $\varepsilon \in \{\pm\}$.  
\end{itemize}

\begin{definition}
\label{def:Gauss-diagram}
A \emph{Gauss diagram}\index{Gauss diagram} $G$ is a set of signed arrows
$\{( (o_i,u_i),\varepsilon_i)\}$ on $\bigsqcup_{i \in I} (S^{1}_{i},\ast_i)$
whose endpoints are pairwise distinct. The \emph{degree} $\deg(G)$ of
a Gauss diagram $G$ is the number of arrows of $G$.
A \emph{chord diagram}\index{chord diagram} is defined similarly. 
\end{definition}

By abuse of notation we sometimes view a signed arrow simply as a chord
or unsigned arrow by forgetting additional information, and in a Gauss diagram
$G$ we will sometimes allow some arrows of $G$ to be unsigned, or, just to be chords.
We will call such a diagram \emph{weak Gauss diagram}, when we need to
distinguish them from the honest Gauss diagrams as defined in Definition
\ref{def:Gauss-diagram}. A \emph{sub Gauss diagram} $G'$ of $G$ is a
(weak) Gauss diagram whose set of (signed) arrows is a subset of the set
of signed arrows of $G$.

As usual, we express a (weak) Gauss diagram $G$ as a diagram consisting
of circle, arrows (or chords) and its sign by drawing an arrow from $o$
to $u$, for each arrow $\vec{a}=(o,u)$ of $G$

An \emph{ordered based link}\index{link! ordered based} is an oriented
$\ell$-component link $L=L_1 \cup \cdots \cup L_{\ell}$ such that the
set of the components $I=\{1,\ldots,\ell\}$ is ordered, and that for each
component $L_i$, a base point $\ast_i \in L_{i}$ is assigned.

Recall that a link diagram $D$ of a link $L \subset \R^{3}$ is an image
$p(L)$ of the projection to the plane $p:\R^{3} \rightarrow \R^{2}$ having
only the double point singularities, equipped with over-and-under information
at each double point. We call the image $D_i:= p(L_i)$ a of component
$L_i$ of $L$ a \emph{component}\index{diagram! component} of the diagram
$D$ (do not confuse with the \emph{connected components} of $D$
as a subset of $\R^2$). By a \emph{sub
diagram}\index{diagram! sub diagram} of $D$ we mean an image $p(L')$ of
a sublink $L'$ of $L$.

An \emph{ordered based link diagram}\index{diagram! ordered based} $D$
is a link diagram of an ordered based link $L$, taken so that none of the base
points $\ast_i$ is a crossing point. That is, an ordered based
link diagram is a link diagram such that 
\begin{itemize}
\item the ordering of components of $D$ is given, and that,
\item for each component $D_i$ of $D$, a base point $\ast_i$ is given.
\end{itemize}

For an ordered based link diagram $D$, one can assign the Gauss diagram
$G_D$\index{$G_D$}\index{Gauss diagram! of ordered based link diagram}
as follows. We view the diagram $D$ as an immersion $\gamma: \sqcup_{i
\in I}(S_i^{1},\ast_i) \rightarrow \R^{2}$ sending the base point $\ast_i$
of $S^{1}_i$ to the base point $\ast_i$ of $D_i$. For each double point
(i.e., a crossing point) $c$ of $D$, we assign the signed arrow $((o(c),u(c)),\varepsilon(c))$,
where $o(c)$ and $u(c)$ are the preimage of the over/under arcs at the
crossing $c$, and $\varepsilon(c)$ is the sign of the crossing $c$. 

We walk along $D_1$ counterclockwise starting from the basepoint $\ast_1$ of $D_1$. When we get
back to the base point
$\ast_1$, then we turn to the second component $D_2$, and walk along $D_2$
from $\ast_2$, then walk along $D_3$ from $\ast_3$, and so on. This procedure
defines the total ordering on the set of arrows (the set of the crossings);
for arrows $\vec{a}=\vec{a}(c)$ and $\vec{a'}=\vec{a'}(c')$ that correspond
to crossings $c$ and $c'$, we define $\vec{a}<\vec{a'}$ if the crossing
$c$ appears before the crossing $c'$.
Similarly, we define the total ordering on the set of the endpoints of
arrows (i.e., the set of the preimage of the double points of $D$). For
the points $x,y \in \sqcup_{i \in I}$ of the endpoints of arrows, we define
$x<y$ if $x$ appears before $y$. We call these orderings \emph{walk-along
ordering}\index{walk-along ordering} of $D$.

For a weak Gauss diagram $A$ and a Gauss diagram $G$, we define
\[ \langle A, G \rangle = \sum_{G' \subset G} \varepsilon(G') \]
where the summation runs the set of all sub Gauss diagrams $G'$ of $G$
such that $G'$ is equal to $A$, and $\varepsilon(G') \in \{\pm 1\}$ is
the product of the signs of all arrows of $G'$. 
Here `$G'$ is equal to $A$' means that after forgetting some additional
information of $G'$ if necessary (such as, by forgetting the sign) $G'$
becomes the same weak Gauss diagram $A$.

Then the linking number is expressed by 
\begin{equation}
\label{eqn:linking-number-formula}
lk(L_i,L_j) = \left\langle \raisebox{-5mm}{
\begin{picture}(72,32)
\put(12,18){\circle{24}}
\put(9,4){$\ast$}
\put(12,0){\scriptsize $i$}
\put(60,18){\circle{24}}
\put(57,4){$\ast$}
\put(60,0){\scriptsize $j$}
\put(48,18){\vector(-1,0){24}}
\end{picture} 
}, G_D \right\rangle 
= \left\langle \raisebox{-5mm}{
\begin{picture}(72,32)
\put(12,18){\circle{24}}
\put(9,4){$\ast$}
\put(12,0){\scriptsize $i$}
\put(60,18){\circle{24}}
\put(57,4){$\ast$}
\put(60,0){\scriptsize $j$}
\put(24,18){\vector(1,0){24}}
\end{picture} 
}, G_D \right\rangle 
\end{equation}

In general, we extend the pairing for a formal linear combination $\mathcal{C}=\sum_{i=1}a_i
A_i$ of weak Gauss diagrams as
\[ \langle \mathcal{C}, G \rangle =  \sum_{i=1} a_{i}\langle A_i, G \rangle
\]
Then every finite type invariant of links is written in the form $\langle
\mathcal{C}, G \rangle$ for a suitable $\mathcal{C}$ \cite{Goussarov-Polyak-Viro}.

\subsection{Polynomial invariants \label{sec:polynomial-invariants}}

In this paper we use the convention that the skein relation of the 
{\em HOMFLY(-PT) polynomial} is 
\[
v^{-1}P_{K_+}(v,z)-vP_{K_-}(v,z)=zP_{K_0}(v,z)\,.
\]
It is known that in $(zv)^{-\#L+1}P_L(v,z)$ only monomials of even degree
of both $v$ and $z$ occur. Also, if $!L$ is the mirror image of $L$, then
\begin{equation}\label{HOMFLY:mirror}
P_{!L}(v,z)=(-1)^{\#L-1}P_L(v^{-1},z)\,.
\end{equation}

 The {\em Jones polynomial} and the {\em Conway polynomial} are recovered
from the HOMFLY polynomial by
\begin{equation}\label{eqn:HOMFLY-Jones}
V(t)=P(t,t^{1/2}-t^{-1/2})\,.
\end{equation}
\begin{equation}\label{eqn:HOMFLY-Conway}
\nabla_L(z)=P_L(z,1)\,.
\end{equation}

It is well-known that the Conway polynomial $\nabla_L(z)$ of a link $L$
is of the form 
\begin{equation}
\label{eqn:Conway-form}
\nabla_L(z)= \sum_{i=0}^{d} a_{\#L-1+2i}(L) z^{\#L-1+2i}\,,
\end{equation}
where $d= \frac{1}{2}(\max \deg_z \nabla_L(z)-\#L+1)$.
We define the \em{leading coefficient}\index{leading coefficient} of the
Conway polynomial by
\[ \Mcf\nb=a_d \neq 0.\]

Moreover, for any link $L$,
\begin{equation}\label{eqn:Conway-bound-chi}
\Md_z\nb_L(z)\,\le 1-\chi(L)\,.
\end{equation}

The \em{Alexander polynomial} $\Dl$ is an equivalent
of the Conway polynomial, given by 
\begin{equation}\label{eqn:Conway-Alexander}
\Dl(t)=\nb(t^{1/2}-t^{-1/2})\,.
\end{equation}
For reasons that will become clear shortly, though, it will be more convenient
for us to almost exclusively use the Conway form of the Alexander polynomial;
$\Dl$ will occur (in this meaning) only in some computational arguments
of Section \ref{sec:comparisons}.

A version of the \em{Kauffman polynomial} will also be discussed,
but at a very late stage, and it will be used less extensively, so it is deferred to Section \ref{sec:ap-I-vs-II}.

\subsection{Signature and signature
type invariants\label{sec:signature}}

In this section we review a definition and basic properties of (Levine-Tristram)
signatures which we use later. For details, we refer to a concise survey
\cite{Conway}. We also review other invariants which share similar properties
with signatures.

We use the convention that the Seifert matrix $A=A_S$ of a Seifert surface
$S$ of $L$ is an $m\times m$ matrix whose $i,j$ entry $A_{ij}$ is 
\[A_{ij} = - lk(\alpha_i,\alpha_j^{+}). \]
Here $m$ is the rank of $H_1(S;\Z)$, $(\alpha_1,\ldots,\alpha_{m})$ is
a set of simple closed curves on $S$ that forms a basis of $H_1(S;\Z)
\cong \Z^{m}$, and $\alpha^{+}_j$ denotes the curve $\alpha_j$ pushed
off $S$ along the positive normal direction of $S$.

The \emph{Levine-Tristram signature}\index{Levine-Tristram signature}
$\sigma_{\omega}(L)$ for $\omega \in \{z \in \C \: | \: |z|=1\}$, is defined
as usual by the signature of $(1-\omega)A + (1-\overline{\omega})A^{T}$.
The \emph{signature}\index{signature} of a link $L$ is $\sigma(L)=\sigma_{-1}(L)$.

Under this convention, the positive (right-hand) trefoil has non-negative
Levine-Tristram signatures and positive signature $\sg=\sg_{-1}=2$. 

We summarize basic properties of the signature which will be used later. Similar
properties hold for the more general Levine-Tristram signatures $\sigma_{\omega}(L)$.

\begin{theorem}
\label{thm:signature-property}
The signature $\sigma$ has the following properties.
\begin{itemize}
\item[(i)] If a link $L_-$ is obtained from $L_+$ by a positive-to-negative
crossing change,
\[ \sg(L_+)-\sg(L_-)\in\{0,1,2\}.\]
In particular, $\sigma(L_+)\leq \sigma(L_-)$.
\item[(ii)] If a link $L_0$ is obtained from a link $L$ by smoothing a
crossing,
\[ |\sg(L)-\sg(L_0)|\le 1\]
\item[(iii)] The signature is additive under connected sum
\[ \sg(L\# L')=\sg(L)+\sg(L').\]
\item[(iv)] $\sigma(L) \leq 1-\chi_4(L) \le \min \{2g(L)+\#L-1,2u(L)\}$.
(For the unlinking number $u(L)$ see Section \ref{sec:unknotting}.)
\item[(v)] $\sg(K)$ is always even for a knot $K$. In general, $\sg(L)+\#
L$ is odd whenever $\Dl_L(-1)=\nb_L(2\sqrt{-1})\ne 0$.
\item[(vi)] If $\nabla_L(2\sqrt{-1}) \neq 0$, then
\[
\label{eqn:nabla-sigma}
\sigma(L) - \#L \equiv\begin{cases}   -1 \pmod 4& \mbox{ if } (2\sqrt{-1})^{1-m}\nb_L(2\sqrt{-1})
>0 \\ 1 \pmod 4& \mbox{ if }  (2\sqrt{-1})^{1-m}\nb_L(2\sqrt{-1})
<0\end{cases}
\]
\end{itemize}
\end{theorem}

The (twice of) the Heegaard Floer tau-invariant $2\tau$  \cite{Ozsvath-Szabo-tau}
and the Rasmussen invariant $s$ \cite{Rasmussen} share many properties
with the signature. They satisfy the properties (i)--(iv). Since many arguments
or discussions, we use only the properties (i)--(iv), we can often extend
the results on signatures to these invariants by the same argument.
Among them, we will occasionally use the tau-invariant, for which also
a computing package exists \cite{Ozsvath-Szabo-HFcalculator}. Similar
arguments for $s$ are valid but will not be explicitly stated or used.

\begin{remark}
Originally, although $\tau$ and $s$ invariants are only defined for knots,
but they are extended for the link case (see \cite{Cavallo} for $\tau$,
and see \cite{Beliakova-Wehrli,Pardon} for $s$) so that they satisfy the
same properties (i),(ii). Although for the sake of simplicity whenever the
$\tau$ (or $s$) invariant is concerned, we will only state the results for knots,
a similar conclusion holds for the link case.
\end{remark}

\subsection{Self-linking number, Bennequin's inequality, 
and strong quasipositivity\label{sec:self-linking-number}}

The \emph{self-linking number}\index{$sl(D)$}\index{self-linking number!
diagram} of a diagram $D$ is defined by 
\begin{equation}
\label{eqn:self-linking-number}
sl(D) = -s(D)+w(D)=-\chi(D)-2c_-(D)\,. 
\end{equation}
The \emph{maximum self-linking number} of a link $L$ is set as
\[
 \overline{sl}(L) = \max \{\,sl(D) \: | \: D
   \mbox{ is a diagram representing } L\}\,.
\]
The terminology `self-linking number' comes from transverse link theory;
$sl(D)$ is the self-linking number of a closed braid obtained from $D$
by Vogel-Yamada's method \cite{Vogel,Yamada}. In particular, $\overline{sl}(L)$
is the maximum of the self-linking number of a transverse link topologically
isotopic to $L$.

Bennequin \cite{Bennequin} showed that 
\begin{equation}\label{eqn:Bennequin}
-\chi(L)\ge \overline{sl}(L)
\end{equation}
for every link $L$, the so-called \em{Bennequin inequality}\index{Bennequin
inequality}. This inequality later had $\chi(L)$ replaced by $\chi_4(L)$
giving the \em{slice Bennequin inequality} (see for example \cite{Rudolph-slice-Bennequin}):

\begin{equation}\label{eqn:slice-Bennequin}
-\chi_4(L)\ge \overline{sl}(L).
\end{equation}

This type of inequality is ubiquitous, in the sense that a knot concordance
homomorphism $v$ gives a similar (slice-) Bennequin type inequality
\[ v(K) \ge \overline{sl}(K) \]
as long as it satisfies the properties that $v(K) \leq 2g_4(K)$ for all
knots $K$, and that $v(T_{p,q}) = 2g_4(T_{p,q})$ for the $(p,q)$-torus
knots $T_{p,q}$. (Such an invariant is called a \emph{slice-torus invariant}
\cite{Lewark}. See \cite{Cavallo-Collari} for an extension to the link case.)

For example, the invariants $\tau$ and $s$ satisfy the Bennequin-type
inequality
\begin{equation}\label{eqn:tau-Bennequin}
\overline{sl}(K)+1\le 2\tau(K)\le 2g_4(K)\,.
\end{equation}

\begin{definition}
We call a link $L$ {\em Bennequin-sharp} if $-\chi(L)=\overline{sl}(L)$.
\end{definition}

It is known that Bennequin-sharpness implies that various $4$-dimensional
invariants coincide with the usual three-genus.

\begin{proposition}\label{prop:Bennequin-sharp} 
For a Bennequin-sharp knot $K$, we have 
\[
1-\chi(K)=2g(K)= 2g_4(K)=\overline{sl}(K)+1=s(K)=2\tau(K).
\]
(The same equality holds for a slice-torus invariant $v$.)
\end{proposition}
\begin{proof}
These identities follow from the slice-Bennequin type inequalities 
\[ \overline{sl}(K) \leq s(K)-1 \leq 2g_4(K)-1, \ \overline{sl}(K) \leq
2\tau(K)-1 \leq 2g_4(K)-1\,. \]
\end{proof}

The Bennequin-sharpness property is related to the following positivity.

\begin{definition}
\label{def:SQP}
An $n$-braid is \emph{strongly quasipositive}\index{strongly quasipositive!
braid} if it is a product of positive band generators $a_{i,j} = (\sigma_i\sigma_{i+1}
\cdots \sigma_{j-2})\sigma_{j-1}(\sigma_i\sigma_{i+1} \cdots \sigma_{j-2})^{-1}$
$(1\leq i<j \leq n)$, where $\sigma_{i}$ is the standard generator of
the braid group $B_n$. 
A link is \emph{strongly quasipositive}\index{strongly quasipositive!
link} if it is represented as the closure of a strongly quasipositive
braid. 
\end{definition}

There is an equivalent definition using Seifert surfaces: A Seifert surface
is \emph{quasipositive}\index{Seifert surface! quasipositive} if $S$ is
realized as an incompressible subsurface of the fiber surface of a positive
torus link. A link is strongly quasipositive if and only if it bounds
a quasipositive Seifert surface \cite{Rudolph-construction-QPIII}. The
notion of quasipositive Seifert surface behaves nicely under the Murasugi
sum.

\begin{theorem}\cite{Rudolph-construction-QPV}
\label{thm:M-sum-quasipositive}
The Murasugi sum of $S_1$ and $S_2$ is quasipositive if and only if both
$S_1$ and $S_2$ are quasipositive.
\end{theorem} 

For the closed braid diagram $D$ from a strongly quasipositive braid representative
of $L$, let $S$ be the Seifert surface given by attaching a disk for every
braid strand and a band for each $a_{i,j}$. The Seifert surface $S$ is
quasipositive, and $sl(D)=-\chi(S)$. Therefore we conclude

\begin{proposition}\label{pro:sqpBs}
A strongly quasipositive link is Bennequin-sharp.
\end{proposition}

It is conjectured that the converse is also true.

\begin{conjecture}\label{conj:SQP=Bennequin-sharp}
A link $L$ is strongly quasipositive if and only if it is Bennequin-sharp.
\end{conjecture}

We mention the following fact for later reference.

\begin{proposition}\label{pro:BS}
Positive and almost positive links (both for type I and type
II) are Bennequin-sharp.
\end{proposition}

For positive links, one can see \cite{Rudolph-positive-QP}.
For almost positive links, this follows from
\cite{Feller-Lewark-Lobb} and Proposition \ref{pro:sqpBs}.
But a far simpler argument can be retrieved from \cite{st-realizing-SQP}
and will be given in a much more general form in \cite{Part2}.

Note also that a (not necessarily strongly) quasipositive
braid gives in a similar way a \em{Seifert ribbon}, which makes
\eqref{eqn:slice-Bennequin} exact for a quasipositive link.
(But a quasipositive link is not necessarily Bennequin-sharp.)

\begin{remark}\label{rem:Hedden}
By the slice Bennequin inequality, when $L$ is strongly quasipositive
(or, in fact, Bennequin-sharp), then
\begin{equation}\label{eqn:g3=g4}
g_4(L)=g(L)\,.
\end{equation}
The figure-8-knot already shows that this property is far from implying strong quasipositivity. Indeed, various stronger properties which lead to \eqref{eqn:g3=g4} fail
to show strong quasipositivity, as one can see in the following examples.

\begin{itemize}
\item If we assume $2g(L)=\sigma(L)$,
then it implies \eqref{eqn:g3=g4}. However,
the knot $K=13_{8541}$ (see Section \ref{sec:Knot-numbering})
satisfies $\sigma(K)=2g(K)=2$,
but $K$ is not strongly quasipositive. In fact, since
$\min \deg_v P_K(v,z) < 0$, one can conclude
that $K$ is not even quasipositive, by
combining exactness of \eqref{eqn:slice-Bennequin}
with Morton's inequality
$1+\overline{sl}(K) \leq \min \deg_v P_K(v,z)$.
\item If we assume $2g(L)=2\tau(L)$,
then again it implies \eqref{eqn:g3=g4}. However,
the knot $14_{45575}$ (a trefoil Whitehead double) is
non-(strongly)-quasipositive knot such that $2\tau(K)=2g(K)=2$.
(On the other hand, if $K$ is fibered, then $\tau(K)=g(K)$ implies $K$
is strongly quasipositive \cite{Hedden}, which also proves Conjecture
\ref{conj:SQP=Bennequin-sharp} for fibered \em{knots}.)
\end{itemize}
We may emphasize that \eqref{eqn:g3=g4} will become very
relevant below and will have to be paid attention in several
discussions.
\end{remark}

\subsection{Unknotting and unlinking\label{sec:unknotting}}

We define the \em{Gordian distance}\index{Gordian distance} $d(K_1,K_2)$
between two
knots $K_1$, $K_2$ (or more generally two links $K_1,K_2$ with
$\#K_1=\#K_2$), as the minimal number of crossing changes needed to pass
between $K_1$ and $K_2$.

For a knot $K$ let $u(K)$ be the \em{unknotting number}\index{unknotting
number} of $K$, which is $d(K,\bigcirc)$. Analogously, for a link $L$
with $n=\#L>1$
components, we define $u(L)$, the \em{unlinking number}\index{unlinking
number} of $L$, as $u(L)=d(L,U_{n})$, the Gordian distance to the $n$-component
unlink.

Similarly, let for a diagram $D$ of a knot $K$, the unknotting number
$u(D)$ be the minimal number of crossing changes in $D$ to make
an unknot diagram out of $D$. By a well-known standard argument,
\[
u(K)\,=\,\min\,\{\,u(D)\,|\,\mbox{$D$ is a diagram of $K$}\,\}
\]

For an $\#L$-component link $L$, 
let $u^{comp}(L)$\index{$u^{comp}(L)$} be the total componentwise unknotting
number, the sum of the unknotting number of each component.
Let $sp(L)$\index{$sp(L)$} be the \emph{splitting number}\index{splitting
number}, the minimum number of crossing changes between different components
to make $L$ as $\#L$ component totally split link (see Section \ref{sec:split}).

When taking sign of changed crossings into account, we say that
$d_+(K_1,K_2)$, the \emph{positive-to-negative Gordian distance}\index{Gordian
distance! positive-to-negative} between $K_1$ and $K_2$, is the minimal
number of positive-to-negative crossing changes needed to make $K_2$ out
of $K_1$.
In the case of positive-to-negative crossing changes, such a sequence
may not exist, in particular when $\sg(K_1)<\sg(K_2)$ or $\tau(K_1)<\tau(K_2)$.
In such case we set $d_+(K_1,K_2)=\iy$. Note that $d_+(K_1,K_2)<\iy$ is
designated as $K_1\ge K_2$ in the partial order of \cite{Przytycki-Taniyama}.

Also, while $d$ is obviously symmetric, $d(K_1,K_2)=d(K_2,K_1)$, we have
that $d_+$ has an antisymmetry: $d_+(K_1,K_2)=d_-(K_2,K_1)$,
where $d_-$ is defined in the suggestive way.

We also set $u_+(K)=d_+(K,\bigcirc)$ to be the \emph{positive-to-negative
unknotting number} of $K$. (Again $u_+(K)=\iy$ is well possible.)

There are various relations of unknotting numbers to the previously
discussed invariants, like
\[
u_+(K)\ge u(K)\ge g_4(K)\ge \max\{|\sg(K)|/2,|\tau(K)|\}\,.
\]

For another way, there is a well-known method using the (homology of)
the double branched covering $\Sigma_2(K)$ to estimate the unknotting
number and the Gordian distances \cite{Lickorish}. As a quite special
case of Lickorish's argument, we will use the following criterion that
uses the \em{determinant}\index{determinant}
\[ \dt(K)=|\nb_K(2\sqrt{-1})|=|\Delta_K(-1)| = \# H_1(\Sigma_2(K);\Z).
\]
By $3_1$ we will denote the (positive, or right-handed) trefoil; see Section
\ref{sec:Knot-numbering}.

\begin{proposition}
\label{prop:torsion-test}
Let $K_1,\ldots,K_n$ be a knot, and let
\[ \dl:=\gcd(\dt(K_1),\dots,\dt(K_n)), \mbox{ and } K:=K_1\#K_2\#\cdots
\#K_n. \]
If $\delta>1$ then $u(K)\ge n$. Similarly if $\dl>3(=\det(3_1))$, then
$d(K,3_1)\ge n$.
\end{proposition}

\subsection{Knot numbering\label{sec:Knot-numbering}}

It follows Rolfsen's tables \cite[Appendix]{Rolfsen} up to 10 crossings,
taking into account the Perko duplication by shifting the index down by
1 for the last 4 knots. (We thus let Perko's knot $10_{161}$ give away
its superfluous alias
$10_{162}$, and the knot finishing the table, which Rolfsen
labels $10_{166}$, is $10_{165}$ for us.)

Fixing a mirroring convention of the knots will not be very relevant.
(For instance, the HOMFLY polynomial can always resolve which
of the mirror images will be meant.)

For 11-16 crossings, the tables of \cite{KnotScape} are used. However,
non-alternating knots are appended after alternating ones of the same
crossing number.
Thus $13_{4878}=13a_{4878}$ is the last alternating 13 crossing knot (the
$(2,13)$-torus knot), and $13_{4879}=13n_1$ is the first non-alternating
one.

For knots of more than 16 crossings (where tables are at least
not available), we adopt the following nomenclature.
We use a similar name of the form `$K=X_{*y}$'. Hereby $X$ is an
integer ($>16$), which is the \em{presumable} crossing number of $K$.
This means, we found (and will display) a diagram of $X$ crossings
of $K$, but we do not generally claim that there is no smaller-crossing
diagram. We verified this crossing minimality for a few knots
$K$ using separate tools (and will occasionally mention it),
but in general this does not seem reasonably feasible until tables are
released, and anyway it is a topic too far off our discussion here.
The index $y$ is used for enumerative purposes only, in order to
facilitate reference and distinguish between the different examples 
below. It has no intended relation to existing knot tables;
to emphasize this fact, and avoid any false association,
we add the asterisk.

We will not appeal to table nomenclature for links, except that let us
fix that we use the letter $H=2_1^2$ for the (positive) Hopf link (which
will be needed at many places). By abuse of notation, we will often use
the same symbol $H$ to represent the standard 2-crossing diagram of the
positive Hopf link.

\section{Successively and good successively
$k$-almost positive diagrams\label{sec:sap}}

In this section we review some of the content of \cite{Ito-sap}, the definition
and properties of (good) successively almost positive diagrams. 

Although we no longer need any results of \cite{Ito-sap} because we will
develop and prove various results in more general settings, they explain
why we introduce and study (weakly) successively almost positive links.

The notion of a successively $k$-almost positive diagram appeared (without
name) in  \cite[Theorem 5.3]{Ito-Motegi-Teragaito}, a joint work of the
first author.
This diagram is designed so that the technique of constructing generalized
torsion elements developed therein can be applied.

\begin{definition}[Successively $k$-almost positive diagram/link]
\label{def:sap}
A knot or link diagram $D$ is \emph{successively $k$-almost positive}\index{diagram!
successively $k$-almost positive}\index{successively almost positive}
if all but $k$ crossings of $D$ are positive, and the $k$ negative crossings
appear successively along a single overarc which we call the \emph{negative
overarc}\index{negative overarc} (see Figure \ref{fig:succ-almost-positive}).
\end{definition}

It is useful to remark that an overarc of negative crossings can be made
\emph{under}arc by rotating the projecting plane. Thus using an underarc
gives an equivalent definition. (We will use this equivalence in some
examples below.)

\begin{figure}[htbp]
\includegraphics*[width=45mm]{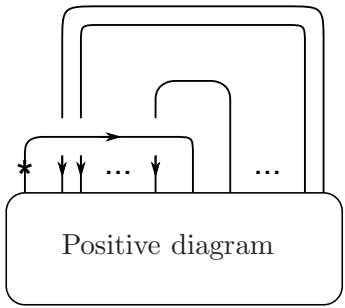}
\begin{picture}(0,0)
\put(-110,20) {\large Positive diagram}
\end{picture}
\caption{Successively $k$-almost positive diagram. $\ast$ represents the
standard base point, which will be used in Section \ref{sec:wsap}.} 
\label{fig:succ-almost-positive}
\end{figure}

The definition says that a successively $0$-almost positive (resp.\ successively
$1$-almost positive) diagram is nothing but a positive (resp.\ almost
positive) diagram. 
On the other hand, a $2$-almost positive diagram is not necessarily a
successively $2$-almost positive diagram; the two-crossing diagram of
the negative Hopf link is $2$-almost positive but not successively $2$-almost
positive.

An investigation of additional properties of diagrams (as we will explain
shortly, motivated from the type I/type II dichotomy of almost positive
diagrams)  leads us to take a closer look at the negative crossings.
It turns out the following distinction is critical.

\begin{definition}[Good crossing and diagram]
\label{def:good-crossing}
A negative crossing of a link diagram $D$ is \emph{good}\index{crossing!
good} if
no other crossing connects the same two Seifert circles. Otherwise it
is \emph{bad}\index{crossing! good}. 
If $D$ has no bad crossing,
we call $D$ \emph{good}\index{diagram! good}.
\end{definition}

In terms of the Seifert graph, a negative crossing $c$ is good if and
only if $c$ is a singular negative edge of $\Gamma(D)$. In terms of the terminologies
in Definition \ref{def:crossing-equivalence} a negative crossing $c$ is
good if and only if there are other crossings which are Seifert equivalent
to $c$.

Using the notion of good crossing, we consider the following restricted
class of successively $k$-almost positive diagrams.

\begin{definition}[Good successively $k$-almost positive diagram/link]
\label{def:good-sap}
A successively $k$-almost positive diagram $D$ is \em{good}\index{successively
almost positive! good successively almost positive} if
all its negative crossings are good.
\end{definition}

This can be paraphrased by saying that, when two distinct Seifert circles
$s,s'$ of $D$ are connected by a negative crossing, then there are no
other crossings connecting $s$ and $s'$. See Figure \ref{fig:goodsap}
for a schematic illustration.

\begin{figure}[htbp]
\includegraphics*[width=55mm]{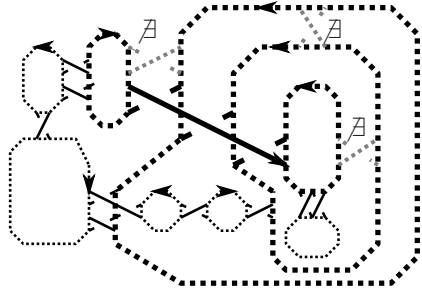}
\begin{picture}(0,0)
\put(-112,95) {$\not \exists$}
\put(-42,95) {$\not \exists$}
\put(-33,58) {$\not \exists$}
\end{picture}
\caption{A schematic illustration of good successively almost positive
diagram; dotted circles represent the Seifert circles, and the bold line
is the negative overarc. Thick dotted Seifert circles represent the Seifert
circles connected by negative crossings. The goodness condition says that
dotted positive crossings, connecting two thick dotted Seifert circles,
cannot exist.}
\label{fig:goodsap}
\end{figure}

A \emph{successively almost positive link}\index{successively almost positive
link} is a link represented by a successively $k$-almost positive diagram
for some $k$. Similarly, a \emph{good successively almost positive link}\index{good
successively almost positive link} is a link represented by a good successively
$k$-almost positive diagram for some $k$.

To avoid confusing notations like `$L$ is not good successively almost
positive', we refer by \emph{loosely successively almost positive}\index{successively
almost positive! loosely successively almost positive} to imply a successively
almost positive diagram/link which fails to have the condition of the
good successively almost positive diagram/link, whenever we would like
to emphasize it is not good.

The distinction of good/loosely successively almost positive diagrams 
is a generalization
of the type I/ type II classification of almost positive diagrams.
Following the terminology of \cite{Feller-Lewark-Lobb}, we say an almost
positive diagram $D$ is  
\begin{itemize}
\item[-] \emph{of type I} if for the two Seifert circles $s$ and $s'$ connected
by the unique negative crossing $c_-$, there are no other crossings connecting
$s$ and $s'$.
\item[-] \emph{of type II} otherwise, i.e., if there are positive crossings
connecting $s$ and $s'$. 
\end{itemize}

An almost positive diagram of type I (resp.\ type II) is nothing but a
good (resp.\ loosely) successively almost positive diagram. 
According to the types of almost positive diagrams, the behavior of their
canonical genus diverges as follows.

\begin{theorem}
\label{thm:ap}\cite{st-positive-polynomial}
Let $L$ be a link represented by an almost positive diagram $D$.
\begin{itemize}
\item[(i)] If $D$ is of type I, then $\chi(D)=\chi(L)=-\overline{sl}(L)$.
\item[(ii)] If $D$ is of type II then $\chi(L)-2=\chi(D)<\chi(L)$.
\end{itemize}
\end{theorem}

This dichotomy plays a fundamental role in the study of almost positive
diagrams and links -- the proof of various properties of almost positive
links often splits into the analysis of the two cases \cite{st-positive-polynomial,Feller-Lewark-Lobb}.

The good/loose distinction can be seen as a generalization of type I/II
dichotomy of almost positive diagrams, as the following result shows.

\begin{theorem}\cite{Ito-sap}
\label{thm:canonical-surface-sap}
Assume that $D$ is a successively $k$-almost positive diagram of a 
link $L$. 
\begin{itemize}
\item[(i)] If $D$ is good, then $\chi(D)=\chi(L)=-\overline{sl}(L)$. Moreover,
$S_D$ is quasipositive.
\item[(ii)] If $D$ is loose (not good), then $\chi(D)<\chi(L)$.
\end{itemize} 
\end{theorem} 

Recall a split diagram represents a split link, but the converse is not
true.
However, a positive diagram represents a split link if and only if the
diagram is split -- this is easily seen by the linking number.
On the other hand, a non-split successively almost positive diagram (say,
the closure of the braid $\sigma_1\sigma_1^{-1}$) may represent a split
link. 
A good successively $k$-almost positive diagram also shares nice properties
with a positive diagram. 

\begin{theorem}\cite{Ito-sap}
\label{thm:split-visible-sap}
A good successively $k$-almost positive diagram $D$ represents a split
link if and only if $D$ is split. 
\end{theorem}

Using the tight connection between $\chi(L)$ and good s.a.p.\ diagram and
the easiness of detecting splitness, we proved the following properties
of the link polynomials which are well-known for positive links.

\begin{theorem}\cite{Ito-sap}
\label{thm:polynomial-sap}
Let $L$ be a good successively almost positive link.
Then $\max \deg_v P_L(v,z)=1-\chi(L)$.
Moreover, if $L$ is non-split, $\max \deg_z \nabla_L(z)=1-\chi(L)$. 
\end{theorem}

These results justify the assertion 
\begin{center}
\emph{`Good successively almost positive links are good generalizations
of (almost) positive links'},
\end{center}
and pose a question whether one can extend these properties for loosely
successively almost positive links.

As we have mentioned, we will (re)prove these results in a more general
form under more general assumptions (partly in \cite{Part2}).

\section{Weakly positive diagram\label{sec:wp}}

\subsection{Definition and simple properties}

\begin{definition}
\label{def:weak-positive}
An ordered based link diagram $D$ is \emph{weakly positive}\index{weakly
positive}\index{diagram! weakly positive}
if $o(c) < u(c)$ holds for every negative crossing $c$ of $D$. That is,
every negative crossing $c$ first appears along an overarc when we walk
along $D$. In a Gauss diagram language, it is equivalent to saying that
$G_D$ has no sub-Gauss diagram of the form
\[\raisebox{-5mm}{
\begin{picture}(72,32)
\put(12,18){\circle{24}}
\put(9,4){$\ast$}
\put(12,0){\scriptsize $i$}
\put(60,18){\circle{24}}
\put(57,4){$\ast$}
\put(60,0){\scriptsize $j$}
\put(48,18){\vector(-1,0){24}}
\put(33,22){$-$}
\end{picture} 
} \ (i<j), \quad
\raisebox{-5mm}{
\begin{picture}(24,32)
\put(12,18){\circle{24}}
\put(10,4){$\ast$}
\put(12,0){\scriptsize $i$}
\put(0,18){\vector(1,0){24}}
\put(10,22){$-$}
\end{picture} 
}  
\]
\end{definition}

In the following, by abuse of notation we say that a link diagram $D$
is \emph{weakly positive} if $D$ is a weakly positive diagram with a suitable
choice of ordering and base points.
We say that a link $L$ is \emph{weakly positive}\index{weakly positive
link} if $L$ is represented by a weakly positive diagram.

\begin{example}\label{exam:2-ap-is-wp}
A 2-almost positive diagram $D$ of a \emph{knot} (i.e., a diagram of a
knot whose crossings are positive except two) is weakly positive, by taking
a base point $\ast$ so that the two negative crossings form a sub-Gauss diagram
$\raisebox{-4mm}{
\begin{picture}(24,28)
\put(12,10){\circle{24}}
\put(9,-4){$\ast$}
\put(21,18){\vector(-1,-1){17}}
\put(20,1){\vector(-1,1){17}}
\put(12,15){\small $-$}
\put(14,4){\small $-$}
\end{picture}.}
$
Contrarily, a $2$-almost positive diagram of a \emph{link} is not always
weakly positive; the standard diagram of the negative Hopf link gives such
an example.
Similarly, a $3$-almost positive diagram of a knot is not always weakly
positive; the standard diagram of the negative trefoil gives such an example.
\end{example}

We observe the following positivity property. 

\begin{proposition}
\label{prop:positivity-wp}
If $L$ is weakly positive, then $L$ can be made into an unlink by positive-to-negative
crossing changes.
Consequently,
\begin{itemize}
\item $\sigma_{\omega}(L)\geq 0$ for every $\omega \in \{z \in \C \: |
\: |z|=1\}$. 
\item If $L$ is a knot, then $s(L) \geq 0$ and $\tau(L) \geq 0$. 
\end{itemize}
\end{proposition}

\begin{proof}
The latter assertion follows from the former, since when $L'$ is obtained
from $L$ by the positive-to-negative crossing changes, then $v(L)\geq
v(L')$ holds for $v=\sigma_{\omega},s,\tau$.

We prove the theorem by induction on $(c(D),c_{+}(D))$.
Let $c$ be the positive crossing of $D$ such that $u(c)$ is minimum (with
respect to the walk-along ordering $<$). That is, we take the first positive
crossing $c$ which we pass along an underarc. Let $D_-$ be the diagram
obtained by changing $c$ to a negative crossing. Since $D_-$ is weakly
positive, by induction the link $L_-$ represented by $D_-$ can be made
to unlink by positive-to-negative crossing changes.
\end{proof}

\subsection{Splitness criterion for weakly positive links\label{sec:splitness-wp}}

 As we have seen and discussed in Section \ref{sec:sap}, the splitness
of a link is not equivalent to the splitness of a diagram, even for almost
positive diagrams. 
To extend visibility of splitness, we introduce the following notion.

\begin{definition}[Height-split diagram]
A link diagram $D$ is \emph{height-split}\index{diagram! height-split}
if $D$ is decomposed as a union of sub diagrams $D=D_1 \cup D_2$ such
that the subdiagram $D_1$ lies above $D_2$; at each crossing of $D$ formed by $D_1$
and $D_2$, the component in $D_1$ always appears as an overarc. 
\end{definition}

This is a generalization of split diagram in the sense that a height-split
diagram obviously represents split link; when $D_1=p(L_1)$ and $D_2=p(L_2)$,
then $L$ is the split union of $L_1$ and $L_2$.

\begin{theorem}
\label{thm:height-split}
Let $D$ be a weakly positive diagram. Then $D$ represents a split link
if and only if $D$ is height-split. 
\end{theorem}

The theorem essentially comes from the following non-triviality of the
linking numbers. 
\begin{lemma}
\label{lem:height-spli}
Let $L=L_1 \cup \cdots \cup L_{\ell}$ be an $\ell$-component link, and
let $D$ be a weakly positive diagram of $L$. For $i<j$, $lk(L_i,L_j)\geq
0$, and $lk(L_i,L_j)= 0$ happens if and only if the component $D_i$ lies
above of $D_j$.  
\end{lemma}
\begin{proof}
Since the Gauss diagram $G_D$ of $D$ contains no arrows of the form 
$\raisebox{-5mm}{
\begin{picture}(72,32)
\put(12,18){\circle{24}}
\put(9,4){$\ast$}
\put(12,0){\scriptsize $i$}
\put(60,18){\circle{24}}
\put(57,4){$\ast$}
\put(60,0){\scriptsize $j$}
\put(48,18){\vector(-1,0){24}}
\put(33,22){$-$}
\end{picture} 
}
$, by \eqref{eqn:linking-number-formula} 
\begin{align*}
lk(L_i,L_j) &= \left\langle G_D, 
\raisebox{-5mm}{
\begin{picture}(72,32)
\put(12,18){\circle{24}}
\put(9,4){$\ast$}
\put(12,0){\scriptsize $i$}
\put(60,18){\circle{24}}
\put(57,4){$\ast$}
\put(60,0){\scriptsize $j$}
\put(48,18){\vector(-1,0){24}}
\end{picture} 
}
\right\rangle 
= 
\left\langle G_D,  \raisebox{-5mm}{
\begin{picture}(72,28)
\put(12,18){\circle{24}}
\put(9,4){$\ast$}
\put(12,-1){\scriptsize $i$}
\put(60,18){\circle{24}}
\put(57,4){$\ast$}
\put(60,-1){\scriptsize $j$}
\put(48,18){\vector(-1,0){24}}
\put(33,22){$+$}
\end{picture} 
}\right\rangle \geq 0.
\end{align*}
Thus when $lk(L_i,L_j)=0$, then $G_D$ does not contain an arrow of the
form 
$\raisebox{-4mm}{
\begin{picture}(72,28)
\put(12,14){\circle{24}}
\put(9,0){$\ast$}
\put(12,-5){\scriptsize $i$}
\put(60,14){\circle{24}}
\put(57,0){$\ast$}
\put(60,-5){\scriptsize $j$}
\put(48,14){\vector(-1,0){24}}
\put(33,18){$+$}
\end{picture} 
}
$, either.
Therefore all the arrows connecting $D_i$ and $D_j$ are of the form $\raisebox{-5mm}{
\begin{picture}(72,32)
\put(12,18){\circle{24}}
\put(9,4){$\ast$}
\put(12,0){\scriptsize $i$}
\put(60,18){\circle{24}}
\put(57,4){$\ast$}
\put(60,0){\scriptsize $j$}
\put(24,18){\vector(1,0){24}}
\end{picture} 
}
$ which means that at every crossing of $D_i$ and $D_j$, $D_i$ always
appear as an overarc. 
\end{proof}

To utilize the information of linking numbers we use the following graph.

\begin{definition}
\label{def:linking-graph}
The \emph{linking graph}\index{linking graph} $\Lambda(L)$\index{$\Lambda(L)$}
of a link $L$ is a (labelled) graph, such that the set of vertices $V(\Lambda(L))$
is the set of the components of $L$, and two vertices $L'$ and $L''$ are
connected by an edge if and only if their linking number $lk(L',L'')$
is non-zero, and we assign a labelling $lk(L',L'')$ to that edge.
\end{definition}

If $L$ is split, then $\Lambda(L)$ is disconnected. Of course, the converse
is not true as a famous brunnian link shows.

\begin{proof}[Proof of Theorem \ref{thm:height-split}]
Assume that $L$ is split, so the linking graph $\Lambda(L)$ is disconnected.
Let $\Lambda_o$ be the connected component of the linking graph that contains
$L_1$, and let $\Lambda_u$ be the union of the rest of the components.
Let $L_o = \bigsqcup_{i \in V(\Lambda_{o})} L_i$ and $L_u=\bigsqcup_{i
\in V(\Lambda_{u})} L_i$, and, let $D_o$ and $D_u$ be the subdiagrams
of $D$ that correspond to $L_o$ and $L_u$. 

Since for any component $L_i$ of $L_o$ and $L_j$ of $L_u$, we have $lk(L_i,L_j)=0$,
by Lemma \ref{lem:height-spli} the diagram $D_i$ lies above of $D_j$ if
$i<j$. 
Thus, to see that $D_o$ lies above of $D_u$, it is sufficient to show
that $i<j$ holds for all $i \in V(\Lambda_{o})$ and $j \in V(\Lambda_{u})$.

Assume, to the contrary that $i>j$ happens. Since $1<j$, this implies
that 
$D_1$ lies above of $D_j$, and that $D_j$ lies above of $D_i$. This means
that $D_1$ lies above of $D_i$, so $lk(D_i,D_1)=0$.
However, since $L_i$ and $L_1$ belong to the same connected component
of the linking graph, $lk(D_i,D_1)\neq 0$. This is a contradiction. 
\end{proof}

Since $lk(L_i,L_j)\geq 0$ for a weakly positive link $L$, it is useful
to replace the labelled graph $\Lambda(L)$ by a usual (unlabelled) graph
$\Lambda'(L)$ defined as follows.

\begin{definition}
\label{def:non-weighter-linking-graph}
The \emph{non-weighted linking graph}\index{linking graph! non-weighted
linking graph} $\Lm'(L)$\index{$\Lambda'(L)$} of a weakly positive link
$L$ is a graph such that the set of vertices $V(\Lambda(L))$ is the set
of the components of $L$, and two vertices $L'$ and $L''$ are connected
by $lk(L',L'')$ parallel edges.
\end{definition}

As a corollary we get the following splitness criterion for weakly positive
links.
\begin{corollary}
\label{cor:splitness-wp}
For a weakly positive link $L$, the following are equivalent.
\begin{itemize}
\item[(i)] $L$ is split.
\item[(ii)] The linking graph $\Lambda(L)$ (equivalently, the non-weighted
linking graph $\Lambda'(L)$) is disconnected.
\item[(iii)] The coefficient of $z^{\#L-1}$ of $\nabla_L(z)$ is zero.
\end{itemize}
\end{corollary}
\begin{proof}
The equivalence of (i) and (ii) follows from Theorem \ref{thm:height-split},
and the implication (i) $\Rightarrow$ (iii) is obvious. To see (iii) $\Rightarrow$
(ii), we use Hoste-Hosokawa's formula \cite{Hoste,Hosokawa} which says
that the lowest coefficient $a_{\#L-1}(L)$ of $z^{\#L-1}$ of the Conway
polynomial $\nabla_L(z)$ of an $n$-component link is given by
\begin{equation}
\label{eqn:Hoste-Hosokawa}
a_{\#L-1}(L)=\sum_{g \in T}\overline{g}.
\end{equation}
Here $T$ denotes the set of all subtrees $g$ of the linking graph $\Lambda(L)$
having exactly $\#L-1$ distinct edges, and $\overline{g}$ denotes the
product of linking numbers that appear as an edge of $g$. Thus in terms
of the non-weighted linking graph $\Lambda'(L)$, \eqref{eqn:Hoste-Hosokawa}
is written as
\begin{equation}
\label{eqn:Hoste-Hosokawa2}
a_{\#L-1}(L)= \# \mbox{spanning trees of } \Lm'(L).
\end{equation}
Thus, if $a_{\#L-1}(L)=0$, then $\Lm'(L)$ has no spanning trees, which
means that the linking graphs $\Lambda(L)$ and $\Lm'(L)$ are disconnected.
\end{proof}

Corollary \ref{cor:splitness-wp} gives a close connection between the (non-weighted)
linking graph $\Lambda'(L)$ and splitness. As for the primeness, we have
the following.

\begin{lemma}
\label{lemma:wp-isthmus}
Let $L$ be a weakly positive, non-split link. If $\Lambda'(L)$ has an isthmus,
then  $L$ is a connected sum of the positive Hopf link and some other
weakly positive links.
\end{lemma}
\begin{proof}
Assume that an isthmus $e$ is an edge connecting $L_i$ and $L_j$ ($i \geq
j$).
That $e$ is an isthmus means that there are no other edges  connecting $L_i$
and $L_j$, so 
$lk(L_i,L_j)=1$. Therefore there exists a unique positive crossing $c$
that corresponds to $\raisebox{-5mm}{
\begin{picture}(72,32)
\put(12,18){\circle{24}}
\put(9,4){$\ast$}
\put(12,0){\scriptsize $i$}
\put(60,18){\circle{24}}
\put(57,4){$\ast$}
\put(60,0){\scriptsize $j$}
\put(48,18){\vector(-1,0){24}}
\put(33,22){$+$}
\end{picture} 
}$
in its Gauss diagram. 
Let $D_-$ be the diagram obtained by changing $c$ into a negative crossing.
Then $D_-$ still remains to be weakly positive. Since $e$ is an isthmus,
$\Lambda'(D_-)$ is disconnected so $D_-$ is height-split. Thus we may
write have $D_-=D_o \cup D_u$, where $D_o$ lies above of $D_u$. Since
$D$ and $D_-$ are the same except the crossing $c$, by separating $D_o$
and $D_u$ preserving the crossing $c$, we get a diagram $D'$ which is
 `almost' split, in the sense that the $D_o$ and $D_u$ are disjoint except
at the crossing $c$ and the other positive crossing $c'$, as shown in
Figure \ref{fig:weak-isthmus}. In particular, the diagram $D'$ is the
connected sum of the Hopf link diagram $H$ and some other link diagrams.

\begin{figure}[htbp]
\includegraphics*[width=100mm]{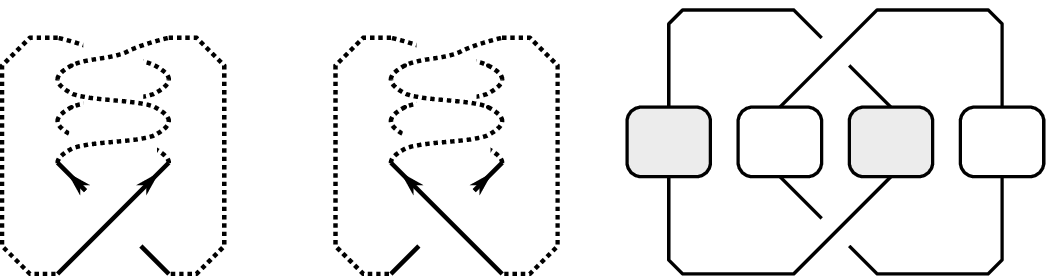}
\begin{picture}(0,0)
\put(-82,34) {\large $D'_o$}
\put(-22,34) {\large $D''_o$}
\put(-52,34) {\large $D'_u$}
\put(-114,34) {\large $D''_u$}
\put(-260,-4) {$c$}
\put(-230,-4) {$L_i$}
\put(-295,-4) {$L_j$}
\put(-300,70) {(i)}
\put(-210,70) {(ii)}
\put(-130,70) {(iii)}
\end{picture}
\caption{(i) The diagram $D$; the component $L_i$ lies above of $L_j$, except at $c$. (ii)
The height-split $D_-$ by crossing change at $c$. (iii) Diagram $D'$ which
is `almost' split.} 
\label{fig:weak-isthmus}
\end{figure} 

\end{proof}

This observation leads to the following estimate of the lowest coefficient
of the Conway polynomial.

\begin{proposition}
\label{prop:Conway-wp}
When $L$ is a non-split prime weakly positive link, then 
\[ a_{\#L-1}(L)\ge lk(L)\ge \#L\]
unless $L$ is the Hopf link. (Here $lk(L)=\sum_{i<j}lk(L_i,L_j)$ is the
total linking number.)
\end{proposition}
\begin{proof}
By Lemma \ref{lemma:wp-isthmus}, the non-weighted linking graph $\Lambda(L')$
has no isthmus. Then the assertion follows from \eqref{eqn:Hoste-Hosokawa2}
that $a_{\#L-1}$ is the number of spanning trees of $\Lambda'(L)$, together
with the standard fact of graph theory that a (loop-free multi-)graph
without an isthmus has at least as many spanning trees as edges (and at
least as many edges as vertices). 

\end{proof}

As a complementary result, we mention the following.
\begin{proposition}
\label{prop:Conway-wp-1}
If a non-split weakly positive link $L$ satisfies $a_{\#L -1}(L)=1$, then
$L$ is the connected sum of $\#L-1$ Hopf links and some other (weakly
positive) knots.
\end{proposition}
\begin{proof}
We prove the assertion by induction on $\#L$. The case $\#L=2$ is proven
in 
Proposition \ref{prop:Conway-wp}. 
By Hoste-Hosokawa's formula \eqref{eqn:Hoste-Hosokawa}, the linking graph
$\Lambda(L)$ is the path graph with $\#L$ vertices and all the $(\#L-1)$
edges have weight one. In particular, by Lemma \ref{lemma:wp-isthmus}
$L$ is a connected sum of the Hopf link $H$ and some other weakly positive
links $L',L''$. Since $1= a_{\#L-1}(L)=a_{\#L'-1}(L')a_{1}(H)a_{\#L''-1}(L'')$,
we have $a_{\#L'-1}(L')= a_{\#L''-1}(L'')=1$. Thus by induction both $L'$ and
$L''$ are a connected sum of Hopf links and some other (weakly positive)
knots.
\end{proof} 

\begin{remark}
In the proof of Theorem \ref{thm:height-split}, Lemma \ref{lem:height-spli},
Lemma \ref{lemma:wp-isthmus} and Proposition \ref{prop:Conway-wp}, we only
used the property that $G_D$ contains no sub-Gauss diagram of the form
$\raisebox{-5mm}{
\begin{picture}(72,32)
\put(12,18){\circle{24}}
\put(9,4){$\ast$}
\put(12,0){\scriptsize $i$}
\put(60,18){\circle{24}}
\put(57,4){$\ast$}
\put(60,0){\scriptsize $j$}
\put(48,18){\vector(-1,0){24}}
\put(33,22){$-$}
\end{picture} 
}$\,.
\end{remark}

\section{Weakly successively almost positive diagram\label{sec:wsap}}

In this section we introduce our main object, a weakly successively almost
positive diagram, which is an obvious generalization of a successively
almost positive diagram.

\subsection{Weakly successively almost positive diagram and its standard
skein triple}

\begin{definition}
\label{def:wsap}
We say that a diagram $D$ is \emph{weakly successively $k$-almost positive}
\index{weakly successively almost positive}\index{diagram! weakly successively
$k$-almost positive} if all but $k$ crossings of $D$ are positive, and
the $k$ negative crossings appear (but not necessarily consecutively)
along a single overarc (see Figure \ref{fig:weak-succ-almost-positive}).
We call this the \em{negative overarc}\index{negative overarc}.
\end{definition}

\begin{figure}[htbp]
\includegraphics*[width=50mm]{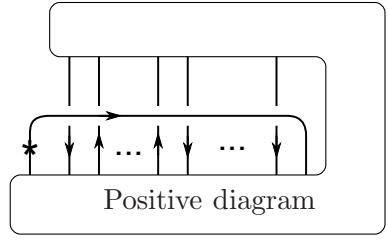}
\begin{picture}(0,0)
\put(-110,10) {\large Positive diagram}
\end{picture}
\caption{Weakly successively almost positive diagram.} 
\label{fig:weak-succ-almost-positive}
\end{figure} 

A weakly successively almost positive diagram is weakly positive. We take
the ordering of the component so that the component that contains the
negative overarc is the smallest (say, $L_1$), and take a base point $\ast_1$
near the initial point of the negative overarc. (The other choices, the
orderings and base points on other components are arbitrary.)

In the following, we will always regard a weakly successively almost positive
diagram as an ordered based link diagram, so that it is a weakly positive diagram.

The \emph{first underarc positive crossing} of a weakly successively almost
positive diagram $D$ is the positive crossing $c$ which we first pass
along underarc; that is, the positive crossing which is the endpoint
of the negative overarc (see Figure \ref{fig:diagram_standard_skein}).

\begin{definition}[Complexity]
For a weakly successively almost positive diagram $D$, we define the \emph{complexity}\index{complexity}
of $D$ by 
\[ \mathcal{C}(D) = (c(D),c(D)-\ell(D))\]
where $\ell(D)$ is the number of crossings that lie on the negative overarc,
which we call the \emph{length} of the negative overarc.

For a weakly successively almost positive link $L$,
we define the complexity
\[ \mathcal{C}(L) = \min \{\mathcal{C}(D) \: | \:  D \mbox{ is a weakly
successively almost positive diagram of } L \}\]

We say that a successively almost diagram diagram $D$ of a link $L$ is
\emph{minimum}\index{minimum weakly successively almost positive diagram}
if $\mathcal{C}(D)=\mathcal{C}(L)$.
\end{definition}

As usual, we compare the complexity according to the lexicographical ordering.

A key feature of a weakly successively almost positive diagram is that it
admits a complexity-reducing skein resolution in the realm of weakly successively
almost positive diagrams. 

\begin{definition}
Let $(D=D_+,D_0,D_-)$ be the skein triple obtained at the first positive 
crossing $c$ terminating the negative overarc (which passes
it as undercrossing). That is, $D_0$ is a diagram obtained by smoothing the crossing
$c$, and $D_-$ is a diagram obtained by changing the positive crossing
$c$ into a negative crossing. 

We call $(D=D_+,D_0,D_-)$ (or, the links $(L=L_+,L_0,L_-)$ represented
by $D$, $D_0$ and $D_-$) the \emph{standard skein triple}\index{standard
skein triple} of $D$.
\end{definition}

\begin{theorem}
\label{thm:skein-wsap}
Let $D$ be a weakly successively almost positive diagram and let $(D=D_+,D_-,D_0)$
be the standard skein triple. Then both $D_-$ and $D_0$ are  weakly successively
almost positive and $\mathcal{C}(D_0), \mathcal{C}(D_-) < \mathcal{C}(D)$.
\end{theorem}

\begin{proof}
The digram $D_-$ is naturally regarded as a ordered based link diagram.
We view $D_0$ as an ordered based link diagram as follows (see Figure
\ref{fig:diagram_standard_skein}).

When smoothing the crossing $c$ connects two distinct components $L_1$
and $L_i$ to form a component $L_{1^{*}}$ of $L_0$, then we just forget
the relevant information of the component $L_i$; the ordering is $1^{*}<2<\cdots$
and the base point of $D_{1^*}$ is just $\ast_{1}$.

When the smoothing the crossing separates the component $L_1$ into two
distinct components $L_{1'}$ and $L_{1''}$, then we put the ordering $1'<1''<2<\cdots$,
where $L_{1'}$ is the component that contains the base point $\ast_1$
of $D_1$. We take the base point $\ast_{1'}$ of $D_{1'}$ as $\ast_1$,
and take a base point of $D_{1''}$ arbitrary (however, it is often convenient
to take near the smoothed crossing $c$).

\begin{figure}[htbp]
\includegraphics*[width=100mm]{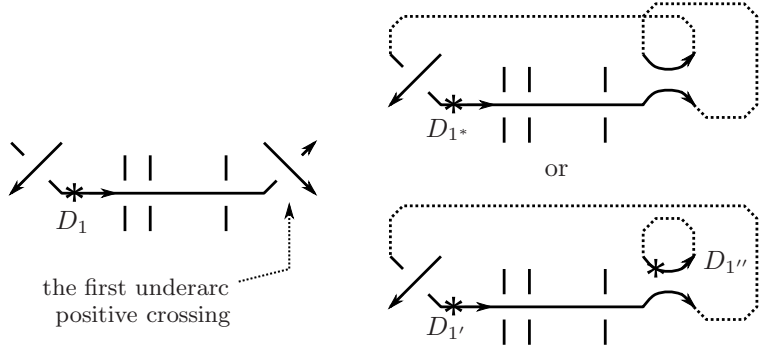}
\begin{picture}(0,0)
\put(-270,45) {$D_1$}
\put(-275,20) {\small the first underarc}
\put(-270,10) {\small positive crossing}
\put(-85,65) {or}
\put(-130,80) {$D_{1^*}$}
\put(-130,5) {$D_{1'}$}
\put(-25,30) {$D_{1''}$}
\end{picture}
\caption{The first underarc crossing (left) of $D$ and how to view the
diagram $D_0$ as a based ordered diagram (right)} 
\label{fig:diagram_standard_skein}
\end{figure} 

Then both $D_0$ and $D_-$ are weakly successively positive.
By the definition of the complexity, $\mathcal{C}(D_0),\mathcal{C}(D_-)
< \mathcal{C}(D)$.
\end{proof}

\subsection{Split links in the standard skein triple}

To use the standard skein triple effectively, we often need to take care
of split links. Thanks to Theorem \ref{thm:height-split}, we can completely
understand when $L_-$ or $L_0$ becomes a split link, when we use the standard
skein triple for a \emph{minimum} w.s.a.p.\ diagram.

\begin{lemma}
\label{lem:skein-split}
Let $D$ be a minimum weakly successively almost positive diagram representing
a non-split link $L$, and let $(L=L_+,L_0,L_-)$ be the standard skein
triple of $L$.
Then $L_0$ is always non-split. Moreover, if $L_{-}$ is split, then the
minimum diagram $D$ can be taken so that is it of the form $D= D' \# H
\# D''$ (see Figure \ref{fig:diagram_when_D-_split}), where $D'$ is a weakly
successively almost positive diagram, $D''$ is a positive diagram, and
$H$ is the standard Hopf link diagram.
\begin{figure}[htbp]
\includegraphics*[width=35mm]{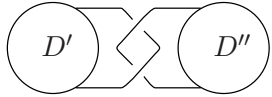}
\begin{picture}(0,0)
\put(-90,15) {\large $D'$}
\put(-25,15) {\large $D''$}
\end{picture}
\caption{Diagram $D=D'\# H \# D''$} 
\label{fig:diagram_when_D-_split}
\end{figure} 
\end{lemma}
\begin{proof}
Since $D$ is minimum, $D$ is reduced and the diagram $D_0$ is non-split.

First, assume to the contrary that $L_0$ is split. By Theorem \ref{thm:height-split},
$D_0$ is height-split so $D_0 = D_u \cup D_s$ where $D_{o}$ contains the
negative overarc, lying above of $D_{u}$. 
Since the diagram $D_0$ itself is non-split, there are crossings between
$D_o$ and $D_u$. However, as $D_u$ lies above of $D_o$, such a crossing can
be removed by suitable isotopy. This leads to a weak successively almost
positive diagram of $L$ with smaller complexity (see Figure \ref{fig:diagram_when_D0_split}).
This is a contradiction.

\begin{figure}[htbp]
\includegraphics*[width=80mm]{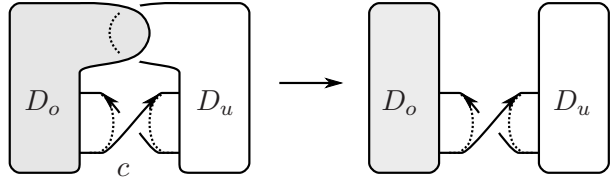}
\begin{picture}(0,0)
\put(-225,25) {\large $D_o$}
\put(-160,25) {\large $D_u$}
\put(-190,0) {\large $c$}
\put(-90,25) {\large $D_o$}
\put(-25,25) {\large $D_u$}
\end{picture}
\caption{When $D_0$ is height-split, $D_o$ and $D_u$ overlap but their
overlaps can be removed without touching the skein crossing $c$.} 
\label{fig:diagram_when_D0_split}
\end{figure} 

Next assume that $L_-$ is split. By Theorem \ref{thm:height-split}, $D_-$
is height-split so we may write $D_0 = D_u \cup D_s$ where $D_{o}$ contains
the negative overarc, lying above of $D_{u}$. Since $D_-$ is not split,
there are crossings between $D_o$ and $D_u$. The assumption that $D$ is
minimum implies that the number of crossings between $D_o$ and $D_u$ is
two. Because otherwise, as in the $D_0$ case, by removing overlaps between
$D_o$ and $D_u$ we can find a weakly successively positive diagram of
$L$ with smaller complexity. 
Then finally we modify the diagram $D$ in the form $D' \# H \# D''$, without
changing the complexity (see Figure \ref{fig:diagram_when_D-_split-proof}).

\begin{figure}[htbp]
\includegraphics*[width=100mm]{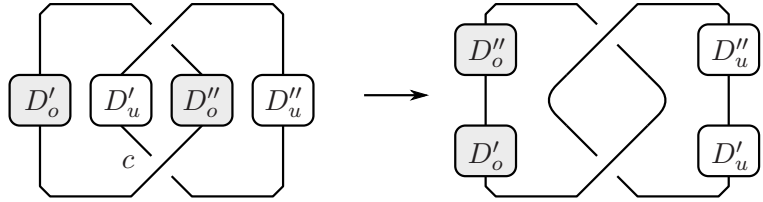}
\begin{picture}(0,0)
\put(-282,34) {\large $D'_o$}
\put(-222,34) {\large $D''_o$}
\put(-245,12) {\large $c$}
\put(-252,34) {\large $D'_u$}
\put(-191,34) {\large $D''_u$}
\put(-114,14) {\large $D'_o$}
\put(-114,53) {\large $D''_o$}
\put(-22,14) {\large $D'_u$}
\put(-22,53) {\large $D''_u$}
\end{picture}
\caption{The case $D_-$ is height-split} 
\label{fig:diagram_when_D-_split-proof}
\end{figure} 
\end{proof}

\subsection{Standard unknotting/unlinking sequence}

The standard skein triple gives rise to the following quite specific types
of unknotting/unlinking sequence of a non-split w.s.a.p.\ link $L$.

\begin{proposition-definition}[Standard unknotting/unlinking sequence]
\label{prop-def:unknotting-sequence}
For each non-split w.s.a.p.\ link $L$, there exists a positive-to-negative
crossing switch sequence 
\begin{equation}\label{eqn:stand-unlink-sequence}
L=L_0\to L_1\to\cdots\to L_m
\end{equation} having the following properties:
\begin{itemize}
\item $L_m$ is the connected sum of $(\#L-1)$ positive Hopf links (when
$L$ is a knot, $L_m$ is the unknot).
\item All $L_{i}$ are non-split weakly successively almost positive.
\item For each $i$, the crossing change $L_i \to L_{i+1}$ is realized
as the standard skein triple $(D=D_+,D_-,D_0)$ of a minimum weakly successively
almost positive diagram $D$ of $L_i$.  
\end{itemize}
We call such a crossing change sequence \eqref{eqn:stand-unlink-sequence}
a \emph{standard unknotting/unlinking sequence}\index{standard unknotting
sequence}\index{standard unlinking sequence} of $L$.
\end{proposition-definition}

\begin{proof}
To show the assertion, it is sufficient to show that for a non-split w.s.a.p.\ %
link $L$ which is not a connected sum of positive Hopf links, there is
a minimum w.s.a.p.\ diagram $D$ of $L$ such that its standard resolution
$D_-$ (the diagram $D_-$ in the standard skein triple $(D=D_+,D_0,D_-)$)
represents a non-split link.

By Lemma \ref{lem:skein-split}, $D_-$ represents a split link only
if $D$ is a diagram connected sum $D=D'\# H \# D''$, where $D'$ is a w.s.a.p.\ %
diagram, $H$ is the standard Hopf link diagram, and $D''$ is a positive
diagram. Let $L'$ and $L''$ be the link represented by $D'$ and $D''$,
respectively.
 
If $L'$ is not a connected sum of Hopf links, by induction there exists
a minimum w.s.a.p.\ diagram $\widehat{D}'$ of $L'$ such that its standard
resolution, $\widehat{D}'_-$ does not represent a split link.
Let $\widehat{D} = \widehat{D'} \# H \# D''$, where the connected sum
is taken away from the negative over arc of $\widehat{D}'$. Then $\widehat{D}$
is a minimum w.s.a.p.\ diagram of $L$ and its standard resolution $\widehat{D_-}$
represents a non-split link.

If $L'$ is a connected sum of Hopf links, $D'$ is a positive diagram, and
$L''$ is not a connected sum of Hopf links, then we replace the base point
of $D$ on $D''$ to interchange the role of $D'$ and $D''$. Then the same
argument gives the desired minimum w.s.a.p.\ diagram $\widehat{D}$ of $L$.
\end{proof}

\begin{remark}\label{rem:variant-standard-sequence}
Note that there is a way to modify the standard unknotting/unlinking
sequence by choosing the standard skein triple not at the end of
the negative overarc, but at the beginning. That is, we choose the
crossing in $D_+$ to be the underpass terminating the overarc backwards,
and in $D_-$ and $D_0$ can move the starting point along the negative 
overarc backward along its respective component. We will not need this
freedom to extend the negative overarc backward too often, but this will
be used for Theorem \ref{th:trefoil}.
\end{remark}

\begin{remark}
It is informative to note that the argument in this section can be applied
to a wider class of link diagrams. 
For the proofs of the theorems in this section to work, we actually need
a class of diagrams $\mathcal{C}$ such that:
\begin{itemize}
\item[(0)] The trivial link diagrams belong to the class $\mathcal{C}$.
\item[(1)] Theorem \ref{thm:height-split} holds; namely, a diagram $D$
in the class $\mathcal{C}$ represents a split link of and only if it is
height-split.
\item[(2)] A minimum complexity diagram $D$ admits a positive skein resolution
in the class $\mathcal{C}$; $D$ has a positive crossing $c$ such that
in the skein triple $(D=D_+,D_0,D_-)$, both $D_-$ and $D_0$ belong to the
class $\mathcal{C}$.
\end{itemize}

In this prospect, a good feature of a weakly successively almost positive
diagram is that the standard triple $(D=D_+,D_0,D_-)$ satisfies the property
(2).

Thus it an interesting problem to find a useful (and wider) class of diagrams
$\mathcal{C}$ that satisfies all the above properties. Since the class
of weakly positive diagrams satisfies properties (0) and (1), it is important
to study to what extent we can expect the property (2) for weakly positive
diagrams.
\end{remark}

\section{Application of standard skein triple I: Conway polynomial\label{sec:Conway}}

We proceed to apply the standard skein triple to establish various properties
of w.s.a.p.\ links.

For a general skein triple $(D=D_+,D_0,D_-)$ of the diagram $D$, we always
have the inequality
\begin{equation}
\label{eqn:skein-inequality-chi4}
\mbox{$-\chi_4(D_-) \ge  -\chi_4(D)-2$ and $-\chi_4(D_0)
\ge -\chi_4(D)-1$.}
\end{equation}

As for the maximum Euler characteristic, we can say much stronger.
\begin{theorem}[Scharlemann-Thompson \cite{Scharlemann-Thompson}]
\label{thm:ST}
Let $(L_+,L_-,L_0)$ be a skein triple.
Then one of the following occurs.
\begin{itemize}
\item $\chi(L_+)=\chi(L_-) \leq \chi(L_0)-1$.
\item $\chi(L_+)=\chi(L_0)-1 \leq \chi(L_-)$.
\item $\chi(L_-)=\chi(L_0)-1 \leq \chi(L_+)$.
\end{itemize}
\end{theorem}

As we will prove in Appendix, we can also say about the fiberedness properties
from the skein triple.
\begin{theorem}[Fibered link enhancement of Scharlemann-Thompson's Theorem]
\label{thm:st-fibered}
Let $(L_+,L_-,L_0)$ be a skein triple.
\begin{itemize}
\item[(i)] Assume that $\chi(L_+)=\chi(L_-) < \chi(L_0)-1$, and that $L_-$
is fibered. Then $L_+$ is fibered.
\item[(ii)] Assume that $\chi(L_{\pm})=\chi(L_0)-1 < \chi(L_{\mp})$, and
that $L_0$ is fibered. Then $L_{\pm}$ is fibered.
\end{itemize}
\end{theorem}

Armed with this knowledge, we use the standard skein triple to relate
knot polynomials and $\chi$ or $\chi_4$.

\subsection{Conway polynomial and (4-ball) genus}

The following property of the Conway polynomial nicely reflects the positivity
of diagrams.
\begin{definition}
We say that a Conway polynomial (see conventions of Section \ref{sec:polynomial-invariants},
specifically the form \eqref{eqn:Conway-form}) $\nabla_L(z)=\sum_{i=0}^{d}
a_{\#L-1+2i}(L)z^{\#L-1+2i}$ is \emph{strictly positive}\index{strictly
positive} if for all $i=0,\ldots,d$, we have $a_{\#L-1+2i}(L)>0$.
\end{definition}

\begin{theorem}
\label{thm:Conway-polynomial}
Assume $L$ is a non-split weakly successively almost positive link.
\begin{itemize}
\item[(i)] $\displaystyle a_{\#L-1+2i}(L) \geq \binom{g_4(L)}{i}$.
\item[(ii)] $\max \deg_z \nabla_L(z) = 1-\chi(L)$. 
\item[(iii)] The Conway polynomial is strictly positive.
\item[(iv)] $L$ is fibered if and only if the Conway polynomial is monic,
i.e., its leading coefficient $a_{1-\chi(L)}(L)=1$.
\end{itemize}
\end{theorem}

This is an advance even when restricting to (almost) positive
links. Property (iii) is only a slight
improvement of \cite[Corollary 2.2]{Cromwell}, but is (for
almost positive links) subsumed by
property (i), which was not known except in special cases.
(See Corollary \ref{cor:conway-positive-knot} and remarks below it
for positive links, and Proposition \ref{pro:BS}
together with \eqref{eqn:g3=g4} for almost positive.)
Further note that, while for a positive link properties (ii)
and (iv) are rather clear (as discussed in \cite{Cromwell}),
part (ii) was obtained for an almost positive link
in \cite{st-minimum-genus} only with great effort, and part (iv) has remained a question there even in this case.

\begin{proof}
All the assertions (i)--(iv) are proven by induction on the complexity
of $L$.
Let $D=D_+$ be a minimum w.s.a.p.\ diagram of $L$, and $(D_+,D_0,D_-)$ be
its standard skein triple.\\

\begin{caselist}

\case{$L_-$ is split.}\\

By Lemma \ref{lem:skein-split}, we may assume that the diagram $D$ is
of the form $D= D' \# H \# D'' $ (where $D'$ is w.s.a.p., $H$ is the standard
Hopf link diagram, and $D''$ is positive).

Let $L'$ and $L''$ be the links represented by $D'$ and $D''$, respectively.
Then 
\[ \#L=\#L'+\#L'',\q g_4(L) \leq g_4(L') + g_4(L''), \mbox{ and } \chi(L)=\chi(L')+\chi(L'')-2.\]
Moreover both $L'$ and $L''$ have strictly smaller complexity so they satisfy
the properties (i)--(iv).

Since $\nabla_L(z)=\nabla_{L'}(z)\nabla_H(z)\nabla_{L''}(z) = z\nabla_{L'}(z)\nabla_{L''}(z)$,
the property (ii) of $L'$ and $L''$ shows that
\begin{align*}
a_{\#L-1+2i}(L) &= \sum_{k+j = i} a_{\#L'-1+2j}(L') a_{\#L''-1+2k}(L'')\\
&\geq \sum_{k+j=i} \binom{g_4(L')}{j}\binom{g_4(L'')}{k} = \binom{g_4(L')+g_4(L'')}{i}
\\
& \geq  \binom{g_4(L)}{i}.
\end{align*}

Similarly,
\begin{align*}
\max \deg_z \nabla_L(z) & = \max \deg_z \nabla_{L'}(z)  + \max \deg_z
\nabla_{L''} (z) +1  \\
& = (1-\chi(L'))+(1-\chi(L''))+1\\
& = 1-(\chi(L')+\chi(L'')-2)\\
& = 1-\chi(L)
\end{align*}
In particular, since
\[ a_{\#L-1+2i}(L) = \sum_{k+j = i} a_{\#L'-1+2j}(L') a_{\#L''-1+2k}(L'')\,,
\]
this means that $a_{\#L-1+2i}(L) >0$ for all $i=0,\ldots,2g(L) = 1-\chi(L)+1-\#L$.

Finally, if $\nabla_L(z)$ is monic, then $\nabla_{L'}(z)$ and $ \nabla_{L''}(z)$
are monic. By induction, $L'$ and $L''$ are fibered, so $L=L'\#H \# L''$
is fibered.\\

\case{$L_-$ is non-split.}\\

To prove the assertion (i), we look at $\# L_0$. 
If $\#L_0 = \#L -1$, then 
\[ g_4(D_0)=(-\chi_4(D_0)- \#L_0 +2)/2 \ge (-\chi_4(D)-1 -(\#L-1) +2)/2
= g_4(D).\]
Thus by induction and the skein relation,
\[ a_{\#L-1+2i}(L) \geq a_{\#L_0 -1 +2i}(L_0) \geq \binom{g_4(D_0)}{i}
= \binom{g_4(D)}{i} \]

Similarly, if $\#L_0 = \#L +1$, then
\[ g_4(D_0)= (-\chi_4(D_0)- \#L_0 +2)/2 \ge (-\chi_4(D)-1 -(\#L+1) +2)/2
=g_4(D)-1\]
and
\[ g_4(D_-)= (-\chi_4(D_-)- \#L_0 +2)/2 \ge (-\chi_4(D)-2 -\#L +2)/2 =
g_4(D)-1.\]
By induction and the skein relation, we get
\begin{align*}
a_{\#L-1+2i}(L) & = a_{\#L_0-1+2(i-1)}(L_0)+ a_{\#L_{-} -1 +2i}(L_-)\\
& \geq \binom{g_4(D)-1}{i-1} + \binom{g_4(D)-1}{i} = \binom{g_4(D)}{i}.\\
\end{align*}

Similarly, to prove the assertion (ii)--(iv), we look at the maximal degree
of the Conway polynomial. 
Since both $\nabla_{L_0}(z),\nabla_{L_-}(z)$ are strictly positive by
induction, we consider the following three cases. 
\\ 

\textit{Case A:} $\max \deg_z \nabla_{L}(z) = \max \deg_z \nabla_{L_-}(z)
= \max \deg_z \nabla_{L_0}(z)+1$\\

In this case $\chi(L_-) = \chi(L_0)-1$, hence by Theorem \ref{thm:ST},
$\chi(L_-) = \chi(L_0)-1 \leq \chi(L)$.
Therefore
\[ 1-\chi(L) \le 1- \chi(L_-) = \max \deg_z \nabla_{L_-}(z)  = \max \deg_z
\nabla_{L}(z)  \le 1-\chi(L)\,, \]
so we get the desired equality $1-\chi(L) = \max \deg_z \nabla_{L}(z)$.

The strict positivity of $\nabla_L(z)$ follows from the strict positivity
of $\nabla_{L_-}(z)$. In this case $\nabla_L(z)$ cannot be monic.\\

\textit{Case B:} $\max \deg_z \nabla_{L}(z) = \max \deg_z \nabla_{L_-}(z)
 >  \max \deg_z \nabla_{L_0}(z)+1$\\

Since $\chi(L_-) = 1-\max \deg_z \nabla_{L_-}(z)  < 1- \max \deg_z \nabla_{L_0}(z)
= \chi(L_0)-1$, we have by Theorem \ref{thm:ST} 
$\chi(L_+) = \chi(L_-) < \chi(L_0)-1 $.
Therefore, $1-\chi(L)=1-\chi(L_-) =\max \deg_z \nabla_{L_-}(z)=\max \deg_z
\nabla_{L}(z)$. 

The strict positivity of $\nabla_L(z)$ follows from the strict positivity
of $\nabla_{L_-}(z)$.

If $\nabla_{L}(z)$ is monic, then $\nabla_{L_-}(z)$ is monic, so by induction
$L_-$ is fibered. By Theorem \ref{thm:st-fibered}, we conclude $L$ is
fibered.\\

\textit{Case C:} $\max \deg_z \nabla_{L}(z) = \max \deg_z \nabla_{L_0}(z)+1
> \max \deg_z \nabla_{L_-}(z)$\\

Since $\chi(L_0)-1 = - \max \deg_z \nabla_{L_0}(z) = 1- \max \deg_z \nabla_{L}(z)<1-
\max \deg_z \nabla_{L_-}(z)$, by Theorem \ref{thm:ST}
$\chi(L_0)-1 = \chi(L) < \chi(L_-)$.
Therefore, $1-\chi(L)= 2-\chi(L_0)= 1+\max \deg_z \nabla_{L_0}(z) =\max
\deg_z \nabla_{L}(z)$. 

To see $\nabla_{L}(z)$ is strictly positive, we note that by Corollary
\ref{cor:splitness-wp} 
\[ \min \deg_{z} \nabla_{L_0}(z) =\#L_{0} -1 = \#L-2 \mbox{ or, } \#L.
\]
Thus for $i>0$, $\#L-1+2i \geq \#L_0-1$.
Therefore for $i=0,\ldots, \frac{1}{2}(\max \deg_z \nabla_{L}(z) -\#L+1)$,
\[ 
a_{\#L-1+2i}(L) = a_{\#L_0 +2i}(L_0) + a_{\#L_{-}-1+2i}(L_-) \geq
\begin{cases} a_{\#L-1}(L_-) >0 & (i=0)\\a_{\#L-1+2i}(L_0)>0 & (i>0)
\end{cases}\]

If $\nabla_{L}(z)$ is monic, then $\nabla_{L_0}(z)$ is monic, so by induction,
$L_0$ is fibered. From Theorem \ref{thm:st-fibered}, we conclude $L$ is fibered.\\

\end{caselist}

\end{proof}

\begin{corollary}
\label{cor:conway-positive-knot}
If $K$ is a positive knot, 
\begin{equation}
\label{eqn:conway-positive-knot} a_{2i}(K) \geq \binom{g(K)}{i}
\end{equation}
\end{corollary}
This is an (ostensible) improvement of \cite[Proposition 4.1]{st-braiding-IV}.

\begin{example}\label{exam:vB}
Using a similar (but simpler) induction argument, Van
Buskirk \cite{Buskirk} showed
\begin{equation}
\label{eqn:VB} \binom{g(K)}{i} \leq a_{2i}(K) \leq  \binom{g(K)+i}{g(K)-i}
\end{equation}
for a \emph{positive braid\/%
\footnote{Recall the footnote on p.\ \pageref{fn:convention-almost-positive}.
Van Buskirk's designation of `positive' is
different from what we called so here, in accordance with
the vast consensus in the more recent literature.} knot} $K$.

While Corollary \ref{cor:conway-positive-knot} extends the left estimate
of \eqref{eqn:VB} for positive knots, notice that no upper bound based
on $g(K)$ and $i$ alone could apply for a general positive knot $K$, as
the example of twist knots shows.
Even for a fibered positive knot $K$ (of which there are finitely
many for given $g(K)$), the right estimate in \eqref{eqn:VB}
is false in this form, as show the simple examples
$K=10_{154},10_{161}$.
\end{example}

By Proposition \ref{pro:BS},
we have \eqref{eqn:conway-positive-knot}
for almost positive knots as well. In light of these results, we
expect that in Theorem \ref{thm:Conway-polynomial}
(i) we can use $g(L)$ instead of $g_4(L)$. This is obviously true if 
\eqref{eqn:g3=g4} holds. We will address this issue at some length in \cite{Part2}.

It is then interesting and natural to discuss to what extent Theorem
\ref{thm:Conway-polynomial} (i) or Corollary \ref{cor:conway-positive-knot}
is optimal.

The following examples and observations on other related results show
the (limited) scope of (possible) further improvements.
This pertains to some kind of optimality of
Theorem \ref{thm:Conway-polynomial}
(i), at least as far as positive links are concerned.

\begin{example}\label{exam:Conway-positive}
\def\theenumi{\arabic{enumi}}
\def\labelenumi{\arabic{enumi})}
\begin{enumerate}
\item The connected sums of (positive) Hopf links and trefoils make
\eqref{eqn:conway-positive-knot} exact for arbitrary $\# L,i$. 
Thus without adding further assumptions such as primeness, the bound \eqref{eqn:conway-positive-knot}
is optimal.

\item For the leading coefficient case, $i=g$, the presence of fibered
links makes
\eqref{eqn:conway-positive-knot} exact in a trivial way, even though we
add an assumption that $L$ is \emph{prime}.
Positive prime fibered links occur for all $g>0$ and $\# L$,
as exemplify the (positively oriented) Montesinos links
\[
\let\ds\displaystyle
\def\f#1#2{\mbox{\small$\ds\frac{#1}{#2}$}}
L=N_{m,k}=M\Bigl( -1,-\f{1}{2k-1},
\underbrace{\f{1}{2},\dots,\f{1}{2}}_m\Bigr)
\]
(for $k,m>0$, with $\#N_{m,k}=m$ and $g(N_{m,k})=k$ for \eqref{eqn:def-genus};
see \cite[Section 2.7]{st-adequate} for explanation and convention).
The most economic positive {\em braid} prime links we found occur as
closures of extensions of the braids $\bigl((\sg_1\sg_3\dots)
(\sg_2^2\sg_4^2\dots)\bigr)^2$, which would apply approximately
for all $g\ge \# L/2$.

\item
For the second leading coefficient $i=g-1$, we mention that in \cite{Ito-HOMFLY}
we made some improvements of \eqref{eqn:VB} for the prime positive \emph{braid}
link cases.
For example, \eqref{eqn:VB} tells that $g(K) \leq a_{2g-2}(K) (\leq 2g(K)-1)$
for a positive braid knot $K$, but it turns out that $a_{2g-2}(K)=2g(K)-1$
whenever $K$ is a prime positive braid knot. We have a similar improvement
for $a_{2g(K)-4}$  whenever $K$ is prime.

\item 
\label{exam:4}
If one considers $i<\!\!< g$, a simple skein (module) calculation for
the (positively oriented) Montesinos links
\[
M_{m,n}=M\bigl(
\underbrace{\myfrac{2}{3},\dots, \myfrac{2}{3}}_n,
\underbrace{\myfrac{1}{2},\dots,\myfrac{1}{2}}_m\bigr)
\]
(for $m>0$, with $\#M_{m,n}=m$ and $g(M_{m,n})=n$ for \eqref{eqn:def-genus})
shows
\begin{equation}\label{mi}
a_{m-1+2i}(M_{m,n})=\binom{n}{i}\cdot (m+i)\,.
\end{equation}
(For example, $n=2$ and $m=1$ gives $M_{1,2}=8_{15}$
with $\nb(8_{15})=1+4z^2+3z^4$.)
Thus, even if we consider prime links $L$,
at least when $\# L$ and $i$ are fixed,
\eqref{eqn:conway-positive-knot} cannot be improved by more than the
factor $\#L+i$ independent of $g$.
\item For $i=1$, and $K$ a positive knot, a substantial study was conducted
in \cite{st-positive-knots}. Along with $a_2(K)\ge g(K)$ (\cite[Theorem
6.2]{st-positive-knots}), we knew some modifications and improvements.

\item For $i=0$, the impact of $\# L$ under primeness, suggested by part
\ref{exam:4}), is indeed widely present in a much weaker setting.
As we have seen in Proposition \ref{prop:Conway-wp}, for a weakly positive
prime link, which is not the Hopf link, $a_{\#L-1}\geq \#L$. However,
Theorem \ref{thm:Conway-polynomial} only shows that $a_{\#L-1}\geq 1$.
\end{enumerate}
\end{example}

On the other hand, the next example shows that Theorem \ref{thm:Conway-polynomial}
cannot be extended for 2-almost positive links.

\begin{example}
\label{exam:2-ap-Conway}
The 2-almost positive knot $12_{1581}$ has $\Md\nb<2g=4$.
The 2-almost positive knot $13_{6407}$ has $\Md\nb=2g=4$ and $\nb$ is
monic, but the knot is not fibered. This can be inspected from \cite{st-knot-table}.

The knots also fail various conditions on the HOMFLY polynomial which
we prove in the next section. There is thus further evidence that even
for $k=2$ most
properties are lost, as soon as the ``coordination'' of the
negative crossings is abandoned.
\end{example}

Finally, since the notion of successively almost positive link comes from
a study of generalized torsion element, which is motivated from the bi-orderability
of the link group, it deserves to mention the following corollary.

\begin{corollary}
If $L$ is weakly successively almost positive, then the link group $\pi_1(S^{3}
\setminus L)$ is not bi-orderable.
\end{corollary}
\begin{proof}
Since $\nabla_L(z)$ is strictly positive, the Alexander
polynomial $\Delta_L(t)=\nb_L(t^{\frac{1}{2}}-t^{-\frac{1}{2}})$ cannot
have a positive real root. Since for a link $L$ whose link group is bi-orderable,
if $\Md_z \nabla_L(z)=1-\chi(L)$ (such a link is called \emph{rationally
homologically fibered}), then its Alexander polynomial $\Delta_L(t)$ has
at least one positive real root \cite{Ito-BO}, this shows that $\pi_1(S^{3}
\setminus L)$ is not bi-orderable.

\end{proof}

\subsection{Unknotting number and the Conway polynomial}

{}From the standard unknotting/unlinking sequence, we can also relate the unknotting/unlinking
numbers and the Conway polynomial. We refer to Section \ref{sec:unknotting}.

\begin{proposition}\label{prop:u-Conway}
{$ $}
\begin{itemize}
\item[(i)] For each w.s.a.p.\ knot $K$, 
 \begin{equation}\label{eqn:u-u+-a2}
u(K)\le u_+(K) \le a_2(K)\,.
\end{equation}
Moreover, if $u(K)=a_2(K)$ or $u_+(K)=a_2(L)$, then the length of the standard
unknotting
sequence \eqref{eqn:stand-unlink-sequence} is equal to $u(K)$, and $a_2(K_i)=u(K)-i$.
\item[(ii)] For each non-split link $L$,
\begin{align*}
sp(L) &\leq \#L- 2 + a_{\#L-1}(L) \\
u^{comp}(L) & \leq a_{\#L +1}(L) \\
u(L) &\leq  \#L- 2 + a_{\#L-1}(L)+a_{\#L+1}(L)
\end{align*}
\end{itemize}
\end{proposition}

\begin{remark}
The inequality \eqref{eqn:u-u+-a2} was known for positive knots $K$ from
\cite[Theorem 6.2]{st-positive-knots}. More precisely, it can be seen
from the
proof (compare also with \cite[Theorem 6.4]{st-positive-knots}) that
$a_2(K)\ge u(D)$ for the unknotting number of a positive diagram $D$ of
$K$. 
Since obviously such a diagram, we have $u_+(K)\le u(D)$, this gives,
for
positive knots, a (slight) improvement of \eqref{eqn:u-u+-a2}, not recovered
here.
However, in Corollary \ref{cor:G-distance-trefoil} we will have another
(slight, but also independent) improvement of \eqref{eqn:u-u+-a2}, valid
for
all w.s.a.p.\ knots $K$.

The inequality \eqref{eqn:u-u+-a2} for $a_2(K)=1$ would imply
$u(K)=u_+(K)=1$, but in fact the case turns out uninteresting,
because from Proposition \ref{prop:char-trefoil} we will know exactly
what knots occur: only the trefoil. (This is also the reason we chose
in Example \ref{exam:u-Conway} a knot with $a_2=2$.)

\end{remark}

\begin{proof}
(i) Let $K=K_0  \to K_1\to \cdots \to K_m =\bigcirc\,$ be a standard unknotting
sequence. 
Since each crossing change $K_i \to K_{i+1}$ is a part of the standard
skein triple $(D_+,D_0,D_-)$ of a minimum w.s.a.p.\ diagram $D=D_+$ of $K_i$,
by Lemma \ref{lem:skein-split} the 2-component link $K'_{i}$ represented
by $D_0$ is not split. Thus $a_1(K'_i)\geq 1$.
Therefore 
\begin{align*}
a_2(K) &=a_2(K_1)+a_1(K'_0) = a_2(K_2)+a_1(K'_1)+a_1(K'_0)\\
&= \cdots \\
&= a_2(K_m) + \sum_{i=0}^{m-1} a_1(K'_m) =  \sum_{i=0}^{m-1} a_1(K'_m)
 \geq m  
\end{align*}
In particular, when $a_2(K)=u(K)$ or $a_2(K)=u_+(L)$, this implies that
 the length $m=a_2(K)$ and $a_2(K_i)=u(K)-i$.\\

(ii)
For a standard unlinking sequence $L=L_0  \to L_1 \to \cdots \to L_m $
each crossing change $L_i \to L_{i+1}$ is a part of the standard skein
triple $(D_+,D_-,D_0)$ of a minimum w.s.a.p.\ diagram $D=D_+$ of $L_i$.
By Lemma \ref{lem:skein-split} the link $L'_{i}$ represented by $D_0$
is not split.

When the crossing change $L_i \to L_{i+1}$ is a self-crossing change (the two
strands at the crossing belong to the same component of $L_i$), then $\#
L'_{i} = \#L +1$. Therefore $a_{\#L}(L_i') = a_{\#L_i'-1}(L'_i) \geq 1$,
and thus
\begin{align*}
a_{\#L-1}(L_{i}) &= a_{\#L -1}(L_{i+1}), \\
a_{\#L + 1}(L_{i}) &=  a_{\#L + 1}(L_{i+1}) +  a_{\#L }(L'_i) \geq a_{\#L
+1}(L_{i+1}) + 1
\end{align*}

Similarly, when the crossing change $L_i \to L_{i+1}$ is a non-self-crossing
change (the two strands at the crossing belong to different components
of $L_i$), then $\# L'_i = \#L -1$. Thus $a_{\#L -2}(L'_i) = a_{\#L_i' -1}(L'_i)
\geq 1$ and
therefore
\begin{align*}
a_{\#L - 1}(L_{i}) &= a_{\#L -1}(L_{i+1}) + a_{\#L-2}(L'_i) \geq a_{\#L
-1}(L_{i+1}) +  1. \\
a_{\#L + 1}(L_i) & =  a_{\#L + 1}(L_{i+1}) +  a_{\#L }(L'_i) \geq a_{\#L
+ 1}(L_{i+1})
\end{align*}

Thus if in the standard unknotting sequence, there are $m'$ self-crossing
changes and $m'' $ non-self-crossing changes,
we have
\[ a_{\#L -1 }(L) \geq m'' + 1, \ a_{\#L + 1}(L) \geq m' \]

Since the link $L_m$, the connected sum of $(\#L-1)$ Hopf links, is made
into the $\# L$-component unlink by $\#L -1$ non-self crossing changes,
we have $m'+m'' +\#L-1 \geq u(L)$.
Moreover, since each component of $L_m$ is already the unknot,
\[ m'' + \#L -1 \geq sp(L), \quad m' \geq u^{comp}(L) \]
Therefore, we conclude
\begin{eqnarray*}
sp(L) & \leq & \#L-2+a_{\#L-1}(L), \\
u^{comp}(L) & \leq & a_{\#L +1}(L), \\
u(L) & \leq & \#L-2+a_{\#L-1}(L)+a_{\#L +1}(L)\,.
\end{eqnarray*}
\end{proof}

\begin{example}\label{exam:u-Conway}
The knot $K=11_{500}$ is a fibered knot with positive monic Conway polynomial
$\nabla_K(z)=1+2z^2+2z^4+4z^6+z^8$, $g_4(K)=2$ and $u(K)=3$ (see \cite{LMo}).
Thus the knot satisfies all the properties in Theorem \ref{thm:Conway-polynomial},
but is not w.s.a.p.\ by Proposition \ref{prop:u-Conway}.
\end{example}

Similarly, a standard unlinking sequence provides various lower inequalities.

\begin{proposition}
\label{prop:signature-deg-Conway}
If $L$ is a non-split w.s.a.p.\ link, then
\[ \min \deg_z\nb_L(z)=\#L-1 \le \sg(L) \le \Md_z\nb_L(z) \]
\end{proposition}

\begin{proof}
From a standard unlinking sequence \eqref{eqn:stand-unlink-sequence},
by positive-to-negative crossing changes, we may convert $L$ into a connected
sum of $\# L - 1$ (positive) Hopf
links $L_m$, so $\sg(L)\ge \sg(L_m)=\#L-1$.

For every link $L$, by \cite[Corollary 2.2]{Gilmer-Livingston}
\[ |\sigma(L)| + 1 -\# L +m_1/2 \leq m_2\]
where $m_1$ is the multiplicity of $-1$ as a root of the Alexander polynomial
$\Delta_L(t)=\nabla(t^{\frac{1}{2}}-t^{-\frac{1}{2}})$, and $m_2$ is the
number of roots of $\Delta_L(t)$ (counted with multiplicity) away from $1$.
We have
\begin{align*}
&m_1 \ge 0 \q\mbox{and} \\
&m_2 \le \max \deg_z\nb_L(z) - \min \deg_z\nabla_L(z)\,.
\end{align*}

Since $\sigma(L)\geq 0$ and $\#L -1 =\min \deg_z\nabla_L(z)$ for a w.s.a.p.\ %
link $L$, this implies 
\[ \sigma(L) \leq \#L -1 + (\max \deg_z\nb_L(z) - \min \deg_z\nabla_L(z))
= \max \deg_z\nb_L(z)\]
\end{proof}

\section{Application of standard skein triple II: HOMFLY polynomial}\label{sec:HOMFLY}
A similar argument proves various special properties of the HOMFLY polynomial.
For the HOMFLY polynomials, we can often drop the assumption that $L$
is non-split, although we can often say more if we assume the non-splitness.

\begin{theorem}
\label{thm:HOMFLY-polynomial}
Assume $L$ is a weakly successively almost positive link. Then the HOMFLY
polynomial 
\[ P_K(v,z) =\sum_{i,j}c_{i,j}v^{i}z^{j}\] 
has the following properties.
\begin{itemize}
\item[(i)] The coefficients $c_{i,j}$ satisfy the following.
\begin{itemize}
\item[(i-a)] $c_{i,j}=0$ if $i<j$.
\item[(i-b)] $c_{j,j} \geq 0$ for all $j$.
\item[(i-c)] $\sum_{i} c_{i,i}=1$. Thus $c_{j,j}=0$ for all $j$ except
exactly one $j$, where $c_{j,j}=1$.
\item[(i-d)] If $L$ is non-split, $c_{i,1-\chi(L)} \geq 0$ for all $i$.
\item[(i-e)] $L$ is fibered if and only if 
$c_{i,1-\chi(L)}=0$ for all $i$ except exactly one $i$, where $c_{i,1-\chi(L)}=1$.
\end{itemize}
\item[(ii)] $\max \deg_{z} P_{L}(v,z) = 1-\chi(L)$.
\item[(iii)] If $L$ is non-split, then $\min \deg_v P(L) \geq \#L-1$.
Moreover, if $L$ is a non-trivial knot, then $\md_vP(L)\ge 2$.
\end{itemize}
\end{theorem}
\begin{proof}
The proof is similar to the Conway polynomial case and is done by induction
on the complexity of $L$. 
Let $D$ be a minimum weakly successively positive diagram of $L$, and
let $(D=D_+,D_0,D_-)$ be the the standard skein triple.

(i-a), (i-b): They follow from the skein relation
\[ P_L(v,z) = v^{2}P_{L_-}(v,z)+vzP_{L_0}(v,z). \] 

(i-c): It  follows from the identity\footnote{Note that the coefficient
polynomial in $v$ of $z^{-m}$ in $P$ is always divisible by $(v^{-1}-v)^m$,
for the substitution to yield a genuine polynomial.
A similar remark applies to the 
the HOMFLY-Jones substitution \eqref{eqn:HOMFLY-Jones}.}
\begin{equation}\label{eqn:HOMFLY-vv}
P(v,v^{-1}-v)=1\,.
\end{equation}
Indeed, this implies 
\begin{align*}
1 & = P(v,v^{-1}-v) = \sum_{i \geq j}c_{i,j}v^{i}(v^{-1}-v)^{j} = \sum_{j}
c_{j,j} + (\mbox{terms of degree} >0 \mbox{ in }v).
\end{align*}

(i-d): The assertion is easy to see if $L_-$ is non-split. If $L_-$ is
split, then by Lemma \ref{lem:skein-split} $L$ is a connected sum $L'
\# H \# L''$ where $L',L''$ are non-split weakly successively almost positive
links of smaller complexity and $H$ is the positive Hopf link.
The HOMFLY polynomial of the positive Hopf link $H$ is $P_{H}(v,z) =z^{-1}(v-v^{3})+vz$
hence $H$ satisfies (i-d). By induction, both $L'$ and $L''$ satisfy
(i-d) and $P_L(v,z) = P_{L'}(v,z)P_H(v,z)P_{L''}(v,z)$, thus we conclude that
$L$ also satisfies (i-d). \\

(i-e): It follows from (i-d) and the observation that $a_{1-\chi(L)}(L)=\sum_{i}c_{i,1-\chi(L)}$.\\

(ii): The proof for the non-split link case is the same as for the Conway polynomial
case. 
Assume thus that $L$ is a split link, hence $L$ is a split union of non-split
links $L=L_1 \sqcup \cdots \sqcup L_m$. (That is, $L_i$ are the split
components of $L$.) Thus $1-\chi(L_i) = \max \deg_{z} P_{L_i}(v,z)$.
Since $P_{L}(v,z) = \left(\frac{v^{-1}-v}{z}\right)^{m-1}P_{L_1}(v,z)
\cdots P_{L_m}(v,z)$ we conclude
\begin{align*}
1-\chi(L) & =1-(\chi(L_1)+\cdots+\chi(L_{m}))\\
&= (1-m)+ (1-\chi(L_1)) +\cdots +(1-\chi(L_m))\\
&=  (1-m)+ \max\deg_z P_{L_1}(v,z) + \cdots + \max\deg_z P_{L_1}(v,z)
 \\
&= \max \deg_z P_{L}(v,z).
\end{align*}

(iii): The first assertion is obviously true for connected sums of (positive)
Hopf links (meant to include the unknot if $\#L=1$), so the standard unknotting/unlinking
sequence and the induction over the complexity prove the first assertion.

For a non-trivial knot $K$ let $(K=K_+,K_0,K_-)$ be the standard skein
triple. Since $K_0$ is a non-split 2-component link, $\min \deg_v P_{K_0}(v,z)
\geq 1$. 
Thus from the skein relation 
\[ P_K(v,z)=vzP_{K_0}(v,z) + v^{2}P_{K_-}(v,z) \]
we conclude $\min \deg_v P_K(v,z) \geq 2$. 

\end{proof}

As an application, we show that an alternating (or, more generally homogeneous)
w.s.a.p.\ link is always a positive link.
 
\begin{corollary}
\label{cor:homogeneous-w.s.a.p}
A homogeneous link $L$ is weakly successively almost positive if and only
if it is positive.
\end{corollary}

\begin{proof} 
It is sufficient to show the case $L$ is non-split.
We show that if a diagram $D$ of $L$ is a Murasugi sum of a positive diagram
$D_+$ and a non-trivial negative diagram $D_-$, then the link represented by $D$ is
never weakly successively almost positive. 

Let $L_+$ and $L_-$ be the links represented by $D_+$ and $D_-$, respectively.
(Note that $\chi(L_{\pm})=\chi(D_{\pm})$.)
By \cite{Cromwell} (or use Theorem \ref{thm:polynomial-B-sharp} (iii)\ ),
\[
c_{1-\chi(L_+),1-\chi(L_+)}(L_+)\ne 0\,.
\]
Similarly, since the mirror
image of a negative diagram is a positive diagram,
there exists a $j<1-\chi(L_-)$ with
$c_{j,1-\chi(L_-)}(L_-)\ne 0$, for example, $j=\chi(L_-)-1$;
see \eqref{HOMFLY:mirror}.

By Theorem \ref{thm:Murasugi-sum} (i),
(iii), this implies that there is a $j'<1-\chi(L)$ with
$c_{j', 1-\chi(L)}(L) \ne 0$. Therefore, by Theorem
\ref{thm:HOMFLY-polynomial} (i-a), $L$ is not w.s.a.p.
\end{proof}

As for the assertion (i-e) in Theorem
\ref{thm:HOMFLY-polynomial}, for the positive link case this single non-zero
$c_{i,1-\chi(L)}$ is $c_{1-\chi(L),1-\chi(L)}$ \cite{Cromwell}, so we
expect the following.

\begin{conjecture}
\label{conj:fibered-HOFMLY-top-coefficients}
For a fibered weakly successively almost positive link $L$, the unique
non-zero coefficient $c_{i,1-\chi(L)}$ in (i-e) is $c_{1-\chi(L),1-\chi(L)}=1$.
\end{conjecture}

The HOMFLY polynomial remains very powerful in ruling out the w.s.a.p.\ %
property. The difficulties to circumvent it in finding essential applications
of other tests lead to rather complicated knots in, and outside, the tables.

This will become evident at many places, like Examples \ref{exam:a2=1-criterion},
\ref{exam:a2=u+sig=2-criterion1}, \ref{exam:u-Conway}, and \ref{exam:Conway-a2-a4}.
We chose not to do a search for bizarre instances throughout,
and neither do we have space to discuss how such examples were gathered.
But they clearly emphasize both the strength of the HOMFLY polynomial
and the value of
practical computation.

Since the Jones polynomial is recovered from the HOMFLY polynomial, as
a consequence we get the following properties of the Jones polynomial.

\begin{theorem}
\label{thm:Jones-polynomial}
Let $L$ be a weakly successively almost positive link, and let $V_L(t)$
be its Jones polynomial.
\begin{itemize}
\item[(i)] The sign of the coefficient $t^{\min \deg_t V_L(t)}$ is $(-1)^{\#L-1}$.
\item[(ii)] $\min \deg_v P_K(v,z) \leq 2 \min \deg_t V_K(t) \leq \max
\deg_z P_K(v,z) = 1-\chi(K)$
\item[(iii)] If $L$ is non-trivial and non-split, $\min \deg_t V_K(t)
\geq \begin{cases} 1 & (\#L=1) \\ \frac{\#L-1}{2} & (\#L>1). \end{cases}$
\end{itemize}
\end{theorem}

\begin{proof}
Let $P_L(v,z)=\sum_{i,j}c_{i,j}v^iz^j$ be the HOMFLY polynomial of $L$,
and let 
\begin{equation}\label{m:def}
m=\min \Bigl\{\frac{1}{2}(2i-j) \: | \: c_{i,j}\neq 0\Bigr\}\,.
\end{equation}
Since 
\[ V_L(t)=P_L(t,t^{1/2}-t^{-1/2})= \sum_{i \geq j}c_{i,j}t^{i}(t^{1/2}-t^{-1/2})^{j}\,,
\] 
\begin{align}
\label{eqn:deg-Jones}
\min \deg_t V_L(t)  &= \min \deg_t P_L(t,t^{1/2}-t^{-1/2}) \geq m
\end{align}

Since $P_L(v,z) \in (vz)^{1-\#L}\Z[v^{\pm 2},z^2]$, the equality happens
when the coefficients $c_{i,j}$ satisfy the property
\begin{equation} 
\label{eqn:c_ij}
c_{i,j} > 0 \mbox{ whenever } \frac{1}{2}(2i-j) = m 
\end{equation}

Now thanks to the skein relation $P_L(v,z) = v^2P_{L_-}(v,z) + vz P_{L_0}(v,z)$,
the property \eqref{eqn:c_ij} holds for $L$ when it holds for both $L_-$
and $L_0$. 
Since the HOMFLY polynomial of the unlink satisfies the property \eqref{eqn:c_ij},
we conclude that 
\[ \min \deg_t V_L(t) =m\]
holds for all weakly successively positive links.

Since $c_{i,j} > 0$ in \eqref{eqn:c_ij}, their sum is also positive, meaning
that the coefficient of $t^m$ in $V_L(t)$ is positive (resp.\ negative)
if even (resp.\ odd) powers of $v$ occur in $P_L(v,z)$, which is equivalent
to $\#L$ being odd (resp.\ even). Note that the sign comes from the $-$ for
$t^{-1/2}$ in \eqref{eqn:deg-Jones}. This shows (i).

Consider (ii). By Theorem \ref{thm:HOMFLY-polynomial} (i-c),
there is an $i_0=j_0$ with $c_{i_0,j_0}\ne 0$.
In \eqref{m:def}, this gives $2i-j=j_0$, and
therefore we conclude
\[ m \leq \frac{1}{2}j_0 \leq \frac{1}{2} \max \deg_{z} P_L(v,z)\,.\]

Then, by Theorem \ref{thm:HOMFLY-polynomial} (i-a), we have in
\eqref{m:def} that $i\ge j$, so that $2i-j\ge i$, and we conclude
\[  m \geq \frac{1}{2} \min \deg_{v} P_L(v,z)
\,. \]
This proves (ii). Finally, (iii) similarly easily follows from (ii)
and Theorem \ref{thm:HOMFLY-polynomial} (iii).
\end{proof}

\section{Link polynomials for Bennequin-sharp weakly successively almost
positive links\label{sec:Bennequin-sharp}}

The results on link polynomials of weakly successively almost positive
links can be improved if we further assume that $L$ is Bennequin-sharp.
\index{Bennequin-sharp}

\begin{theorem}\label{thm:polynomial-B-sharp}
Let $L$ be a weakly successively almost positive link. 
If $L$ is Bennequin-sharp, then the following holds.
\begin{itemize}
\item[(i)] $\displaystyle a_{\#L-1+2i}(L) \geq \binom{g(L)}{i}$.
\item[(ii)] $\min\deg_v P_L(v,z) = 2\min \deg_t V_L(t)= \max\deg_z P_L(v,z)
= 1-\chi(L)$.
\item[(iii)] For $j \neq 1-\chi(L)$, $c_{j,j}=0$, unless $j=1-\chi(L)$.
Furthermore, $c_{1-\chi(L),1-\chi(L)}=1$.
\item[(iv)] If $L$ is non-split, then $c_{1-\chi(L),-1-\chi(L)} \leq 1-\chi(L)$.
Equality holds when $L$ is fibered.
\item[(v)] If $L$ is non-split, then
\begin{align*}
V_L(t) &= (-1)^{\#L-1}\Bigl(t^{\frac{1}{2}(1-\chi(L))} +
\bigl(\chi(L)-1+c_{1-\chi(L),-1-\chi(L)}\bigr)
t^{\frac{1}{2}(3-\chi(L))} \\
& \quad + (\mbox{higher degree terms})\Bigr). 
\end{align*}
In particular, if $L$ is fibered then 
\[
V_L(t)= (-1)^{\#L-1}t^{\frac{1}{2}(1-\chi(L))} + \bigl(\mbox{terms of
$t$-degree at least $\frac{1}{2}(5-\chi(L))$}\bigr)
\]
\end{itemize}
\end{theorem}

\begin{proof}
(i): It follows from Proposition \ref{prop:Bennequin-sharp} and Theorem
\ref{thm:Conway-polynomial} (i).\\

\noindent
(ii): Since $L$ is Bennequin-sharp, by Proposition \ref{prop:Bennequin-sharp}
\[ 1- \chi(L) = \overline{sl}(L)+1 \leq \min \deg_{v} P_{L}(v,z)\]
where the second inequality is Morton's inequality \cite{Morton}.
Thus by Theorem \ref{thm:Jones-polynomial} (ii),
\begin{align*}
1- \chi(L) \leq \min \deg_{v} P_{L}(v,z) \leq  2 \min\deg_t V_L(t) \leq \max
\deg_z P_L(v,z) = 1-\chi(L). \\
\end{align*}

\noindent
(iii): 
By (ii) and Theorem  \ref{thm:HOMFLY-polynomial} (i-a) $c_{i,i} = 0$ unless
$i \neq 1-\chi(L)$, hence By Theorem \ref{thm:HOMFLY-polynomial} (i-c),
$c_{1-\chi(L),1-\chi(L)}=1$.\\

\noindent
(iv): By looking at the coefficient of $v^{2}$ in $P_K(v,v^{-1}-v)=1$
we get 
\[ -(1-\chi(L)) c_{1-\chi(L),1-\chi(L)} + c_{3-\chi(L),1-\chi(L)} + c_{1-\chi(L),-1-\chi(L)}
=0 \]
Thus if $L$ is non-split, by Theorem \ref{thm:HOMFLY-polynomial} (i-d)
\begin{equation}
\label{eqn:rem1}
 (1-\chi(L)) - c_{1-\chi(L),-1-\chi(L)} = c_{3-\chi(L),1-\chi(L)} \geq
0
\end{equation}
as desired. 
Moreover, when $L$ is fibered (in this case $L$ is automatically non-split),
the Conway polynomial $\nabla_L(z)=P_K(1,z)$ is monic. By (iii)  
\[ 1 = a_{1-\chi(L)}(L)= \sum_{j\geq 1-\chi(L)}c_{j,1-\chi(L)} = 1 + 
\sum_{j> 1-\chi(L)}c_{j,1-\chi(L)}.\]
Since $c_{j,1-\chi(L)}\geq 0$ by  Theorem \ref{thm:HOMFLY-polynomial}
(i-d), this implies that $c_{j,1-\chi(L)}=0$ for all $j>1-\chi(L)$. Therefore
\[ (1-\chi(L))- c_{1-\chi(L),-1-\chi(L)} = c_{3-\chi(L),1-\chi(L)}= 0.\]

\noindent
(v): Since $V_L(t)= P_L(t,t^{1/2}-t^{-1/2}) = \sum_{i\geq j} c_{i,j} t^i(t^{1/2}-t^{-1/2})^{j}$,
we immediately get
\begin{align*}
V_L(t) &= (-1)^{1-\chi(L)}\left( t^{\frac{1}{2}(1-\chi(L))} + (-1+\chi(L)+c_{1-\chi(L),-1-\chi(L)})
t^{\frac{3}{2}(3-\chi(L))} + \cdots \right)
\end{align*}

\end{proof}

\begin{remark}\label{rem:gap-free}
It is known that when $L$ is a positive link, then the Theorem \ref{thm:polynomial-B-sharp}
(iv),(v) is an equivalence \cite{st-non-triviality}; $c_{1-\chi(L),-1-\chi(L)}=1-\chi(L)$
(equivalently, the coefficient of $t^{\frac{1}{2}(3-\chi(L))}$ is zero)
if and only if $L$ is fibered. 

This equivalence can be extended to Bennequin-sharp weakly successively
almost positive links, if one can show the following \em{gap-free property}\index{gap-free
property} of the maximal $z$-degree term of $P_L(v,z)$;
\begin{equation}
\label{eqn:gap-free-property} c_{i,1-\chi(L)}\neq 0, c_{j,1-\chi(L)} \neq
0 \Rightarrow c_{k,1-\chi(L)}\neq 0 \mbox{ for all } k= i,i+2,\ldots,
j 
\end{equation}
i.e., the coefficient of $z^{1-\chi(L)}$ is of the form
\[ c_{m,1-\chi(L)}v^{m} + c_{m+2,1-\chi(L)}v^{m+2} + \cdots + c_{M,1-\chi(L)}v^{M}
\]
for some $m,M$, where all the coefficients $c_{k,1-\chi(L)}$ are non-zero.
Compare with Theorem \ref{thm:Conway-polynomial} (iii), which can be understood
as the assertion that the Conway polynomial of weakly successively almost
positive link is gap-free.

Indeed, the inequality \eqref{eqn:rem1} in the proof of (iv) shows that
 $c_{1-\chi(L),-1-\chi(L)}=1-\chi(L)$ if and only if $c_{3-\chi(L),1-\chi(L)}=0$.

When we assume that $P_L(v,z)$ satisfies the gap-free property \eqref{eqn:gap-free-property},
$c_{3-\chi(L),1-\chi(L)}=0$ means that $c_{m,1-\chi(L)}=0$ for all $m\geq
3-\chi(L)$. Therefore
\[ a_{1-\chi(L)}(L) = \sum_{i}c_{i,1-\chi(L)} = c_{1-\chi(L),1-\chi(L)}
= 1\,. \]
By Theorem \ref{thm:Conway-polynomial} (iv), $L$ is fibered.
\end{remark}

We mention that because of Proposition \ref{pro:BS}, Theorem \ref{thm:polynomial-B-sharp} (ii) positively answers \cite[Question 2]{st-positive-polynomial}.

\begin{corollary}
\label{cor:HOMFLY-ap}
If $L$ is almost positive link, 
\begin{equation*}\label{pdg}
\min \deg_v P_L(v,z)=\max \deg_z P_L(v,z)= 1-\chi(L)
\end{equation*} 
\end{corollary}

\section{Positivity of the signature and related invariants\label{sec:signature-positive}}

It is known that the signature of a non-trivial (almost) positive link
is always\footnote{Remember our sign convention for $\sg$ fixed in Section
\ref{sec:signature}}
strictly positive: $\sigma(K) > 0$ \cite{Przytycki-Taniyama}. 
In this section we extend the positivity of the signature to w.s.a.p.\ %
links and derive some consequences.

Since we have already seen $\sigma(L) \geq \#L-1$ in Proposition \ref{prop:signature-deg-Conway}
for non-split links, it remains to consider the knot case.

\begin{theorem}\label{thm:signature>0}
If $K$ is a non-trivial weakly successively almost positive knot,
 then $\sg(K)>0$. In particular $K$ is not (algebraically) slice, and
is not amphicheiral.
\end{theorem}

\begin{proof} 
Let $\det(K) = |\nabla_K(2\sqrt{-1})|\geq 1$ be the determinant of
the knot $K$ (compare with Section \ref{sec:unknotting}).
Assume to the contrary that there is a non-trivial w.s.a.p.\ knot $K$ with
$\sigma(K)=0$. Let
\[ d = \min \{\det(K) \: | \: K \mbox{ is a non-trivial w.s.a.p.\ knot},
\sigma(K)=0\},\]
and take a non-trivial w.s.a.p.\ knot $K$ with $\sigma(K)=0$ and $\det(K)=d$
so that among such knots, its complexity $\mathcal{C}(K)$ is minimum.

Let $D$ be a minimum w.s.a.p.\ diagram of $K$ and let $(D=D_+,D_0,D_-)$
be the standard skein triple.
Since $0= \sigma(K_+)\geq \sigma(K_-)\geq 0$, $\sigma(K_-)=0$. 
By the minimum assumption of the complextiy of $K$, this implies that
$\det(K_-)> d$. 

By Theorem \ref{thm:signature-property} (v), $\sigma(K)=\sigma(K_-)=0$
implies that $\nabla_{K}(2\sqrt{-1}),\nabla_{K_-}(2\sqrt{-1})>0$. Also,
$K_0$ is w.s.a.p., so by Proposition \ref{prop:positivity-wp}, we have $\sg(K_0)\ge 0$.
Thus from Theorem \ref{thm:signature-property} (ii), we see that
$\sigma(K)=\sigma(K_-)=0$ implies $\sigma(K_0) \in \{0,1\}$. Hence
by Theorem \ref{thm:signature-property} (v), $(2\sqrt{-1})\nabla_{K_0}(2\sqrt{-1})
\geq 0$.
Therefore
\begin{align*}
\det(K_+) &= |\nabla_{K}(2\sqrt{-1})| = \nabla_{K}(2\sqrt{-1}) = \nabla_{K_-}(2\sqrt{-1})
+ (2\sqrt{-1})\nabla_{K_0}(2\sqrt{-1})\\
&= |\nabla_{K_-}(2\sqrt{-1})| + (2\sqrt{-1})\nabla_{K_0}(2\sqrt{-1})\\
&\geq \det(K_-) > d.
\end{align*}
This is a contradiction.
\end{proof}

We give a few applications of this signature positivity.

First we give a characterization of the simplest w.s.a.p.\ links and knots.

\begin{proposition}
\label{prop:char-hopf-link}
For a w.s.a.p.\ link $L$, the following conditions are equivalent:
\begin{itemize}
\item[(i)] $L$ is a connected sum of $\#L-1$ positive Hopf links.
\item[(ii)] $\nabla_L(z)=z^{\#L-1}$.
\item[(iii)] $a_{\#L-1}(L)=1$ and $\sigma(L)=\#L -1$.
\end{itemize}
\end{proposition}
\begin{proof}
(i) implies (ii) and (iii). To see (ii) implies (iii), we note that $\nabla_L(z)
= z^{\#L-1}$ implies that $L$ is non-split. Therefore $\sigma(L)\geq \#L
-1$ by Proposition \ref{prop:Conway-wp}. If $\sigma(L) > \#L-1$, then
by Corollary \ref{prop:signature-deg-Conway}, $\nabla_L(z) \neq z^{\#L-1}$,
so $\sigma(L) = \#L-1$.

We show (iii) implies (i). Again, $a_{\#L-1}(L)=1$ implies that $L$ is
non-split, hence by Proposition \ref{prop:Conway-wp}, $L$ is a connected
sum of $(\# L-1)$ positive Hopf links and some other w.s.a.p.\ knots. Since a non-trivial
w.s.a.p.\ knot has non-trivial signature, 
$\sigma(L)=\#L-1$ implies that $L$ is a connected sum of $(\#L-1)$ positive
Hopf link.
\end{proof}

\begin{proposition}\label{prop:char-trefoil}
A weakly successively almost positive knot $K$ is the positive (right-handed)
trefoil if and only if $a_2(K)=1$.
\end{proposition}

This conclusion was also well-known for positive knots from
\cite{st-positive-knots}, thus Proposition \ref{prop:char-trefoil} is
its generalization.

\begin{proof}
The one direction is obvious, so assume $a_2(K)=1$.
Let $D$ be a minimum w.s.a.p.\ diagram of $K$ and let $(D_+,D_-,D_0)$ be
the standard skein triple. Since $K_0$ is not split, $a_1(K_0)\geq 1$.

Since $a_2(K) = a_2(K_-)+a_1(K_0)=1$, we need that $a_2(K_-)=0$ and $a_1(K_0)=1$. This
implies that $K_-$ is the unknot. Since $1 \leq \sigma(K_0) \leq \sigma(K_{-})+1
= 1$, this implies $\sigma(K_0)=1$, so by Proposition \ref {prop:char-hopf-link},
$K_0$ is the Hopf link. Thus $\nabla_K(z)=\nabla_{K_-}(z)+z \nabla_{K_0}(z)
= 1 +z^{2}$. This shows that $K$ is a fibered knot of genus one, and only
the trefoil is possible. 
\end{proof}

\begin{example}\label{exam:a2=1-criterion}
The knots $16_{838974}$, $16_{954872}$ and $16_{1263307}$ are among (very)
few that can be prohibited from being w.s.a.p.\ using Proposition \ref{prop:char-trefoil}
 but not by the previous conditions on $\sg$, $\nb$ and $P$. (Composite
knots can be discarded with Proposition \ref{prop:u-Conway} and Scharlemann's
result \cite{Scharlemann}, and no smaller crossing prime examples
exist.)
\end{example}

Next we give a slight improvement of the strictly positive property of
the Conway polynomial in Theorem \ref{thm:Conway-polynomial}.
\begin{proposition}\label{prop:Conway-a_j=1}
Let $L$ be a w.s.a.p.\ non-split link. 
Assume $a_j(L)=1$ for some $j>\#L-1$. Then $\Md_z\nb(L)=j$ (hence $j=1-\chi(L)$
and $L$ is fibered).
\end{proposition}

Another way of saying this is that all ``intermediate'' coefficients of
$\nb$,
 those for $0<i<d=g(L)$ in \eqref{eqn:Conway-form}, are at least
$2$ for a non-split w.s.a.p.\ link $L$.

This will, of course, readily follow from Theorem \ref{thm:Conway-polynomial}
(i) if we are able to see $g(L)=g_4(L)$, which follows if $L$ is Bennequin-sharp
by Proposition \ref{prop:Bennequin-sharp}.
But as long as this is not fully established (see Question \ref{ques:SQP}
(a)(b)) the proposition retains an own (however modest) merit.

\begin{proof}
We prove the assertion on the complexity of $L$.
Let $D$ be a minimum w.s.a.p.\ diagram of $L$ and $(L_+,L_0,L_-)$ be the
standard skein triple. 
 Since
\[ 1 = a_{j}(L) = a_{j}(L_-) + a_{j-1}(L_0) \]
we have two cases.\\

\begin{caselist}
\case{$a_{j-1}(L_0) = 0$ and $a_{j}(L_-) = 1$}\\

By induction, $\max \deg_z \nabla_{L_-}(z) = j$, and $\max \deg_z \nabla_{L_0}(z)
< j-1$. Hence $\max \deg_z \nabla_{L}(z) = j$ as desired.\\

\case{$a_{j}(L_-)= 0$ and $a_{j-1}(L_0)=1$}\\

By induction, $\max \deg_z \nabla_{L_-}(z) < j $.
If $j-1 > \# L_0 -1$, then $\max \deg_z \nabla_{L_0}(z) = j-1$ by induction,
so  $\max \deg_z \nabla_{L}(z) = j$ as desired.

Thus we assume that $j-1 = \#L_0 -1$. Since $j > \#L -1$, this happens
only if $j= \#L +1 = \# L_0$. 

If $\max \deg_z \nabla_{L_0}(z) = j-1 = \#L_0 -1$, then $\max \deg_z \nabla_{L}(z)
= j$ as desired. So we may further assume that $\max \deg_z \nabla_{L_0}(z)
> j-1 = \#L_0-1$, namely, $\nabla_{L_0}(z) \neq z^{\#L_0 -1}$.

By Proposition \ref{prop:Conway-wp-1} and Proposition \ref{prop:char-hopf-link},
$a_{\#L_0 -1}(L_0)=1$ and $\nabla_{L_0}(z) \neq z^{\#L_0 -1}$ imply that
$L_0$ is a connected sum of $(\#L_0 -1)$ Hopf links and \emph{non-trivial}
w.s.a.p.\ knots. By Theorem \ref{thm:signature>0}, $\sigma(L_0) \geq (\#L_0-1)
+2 = \#L +2$. 

On the other hand, by Proposition \ref{prop:signature-deg-Conway}
\[ \#L_- - 1 \leq \sigma(L_-) \leq \max \deg_z \nb_{L_-}(z) < j = \#L_- +1\,, \]
which means that $\sigma(L_-) \in \{\#L -1,\#L\}$. This is a contradiction,
because $\sigma(L_-) - \sigma(L_0) \in \{-1,0,1\}$.

\end{caselist}

\end{proof}

\begin{example}\label{exam:Conway-a2-a4}
The knots $14_{26091}$ and $14_{43602}$ satisfy $a_2>a_4=1$ and $\Md_z\nb=6$,
so by  Proposition \ref{prop:Conway-a_j=1} they are not w.s.a.p.\ They
are among minimal crossing prime knots which demonstrate that this criterion
is more essential compared to Example \ref{exam:a2=1-criterion}.
\end{example}

Note that Proposition \ref{prop:Conway-a_j=1} implies that for a w.s.a.p.\ %
knot $K$,
\begin{equation}\label{eqn:conway-at-1}
\nb_K(1)= \Dl_K\left(\ffrac{3\pm \sqrt{5}}{2}\right)=\sum_{i=1}^{g(K)}a_{2i}(K)
\geq 2g(K)
\end{equation}
Equality is possible only if $K$ is fibered, and we know yet no examples
except $K=3_1$ and $3_1\# 3_1$. Indeed, \eqref{eqn:conway-at-1} looks
to be far from optimal since it follows from Theorem \ref{thm:Conway-polynomial}
(i)
\begin{equation}\label{eqn:conway-at-1-ver2}
\nb_K(1)=\Dl_K\left(\ffrac{3\pm\sqrt{5}}{2}\right)=\sum_{i=1}^{g(K)}a_{2i}(K)\,\ge\,\sum_{i=1}^{g(K)}
\binom{2g_4(K)}{i} = 2^{g_4(K)} \geq 2^{\sg(K)/2}
\end{equation}

Finally, we give a slightly different characterization of the trefoil
and the unknot.

A knot $K$ is \emph{2-trivadjacent} (2-adjacent to the trivial knot) if
$K$ admits a diagram $D$ having two distinct crossings $c_1,c_2$ such
that when we apply the crossing change at $c_1,c_2$, or both $c_1$ and
$c_2$, we get the unknot (see \cite{Askitas-Kalfagianni} for the notion
of adjacency).

\begin{corollary}
\label{cor:2-trivadjacent}
A weakly successively almost positive knot is 2-trivadjacent if and only
if it is either the unknot or the trefoil.
\end{corollary}
\begin{proof}
Assume that $K$ is non-trivial and 2-trivadjacent, so there is a diagram
$D$ with crossings $c_1,c_2$ such that crossing change at
$c_1,c_2$, or both $c_1$ and $c_2$ yields the unknot (note that $D$ may
not be w.s.a.p.).
 
By Theorem \ref{thm:signature>0} $\sigma(K)>0$. Since negative-to-positive
crossing change never decreases the signature, this means that both $c_1$
and $c_2$ are positive crossings.

For $\ast_1,\ast_2 \in \{+,-,0\}$, let $D_{\ast_1,\ast_2}$ be the diagram
obtained by replacing each $\ast_i$ as positive/negative/smoothed crossing
(so $D=D_{++}$).
By the skein formula of the Conway polynomial, for every diagram $D$ we
have
\begin{align*}
z^{2}\nabla_{D_{00}}(z) &= \left(\nabla_{D_{++}}(z) - \nabla_{D_{-+}}(z)
- \nabla_{D_{+-}}(z) +\nabla_{D_{--}}(z) \right).
\end{align*}
Since $D_{+-},D_{-+},D_{--}$ represents the unknot, we get $\nabla_{K}(z)=1+z^{2}\nabla_{D_{00}}(z)$.
This implies $a_2(K)=0,1$.

If $a_2(K)=0$, then $K$ is the unknot by Theorem \ref{thm:Conway-polynomial} (iii) and (ii). If $a_2(K)=1$, then $K$ is the trefoil by Proposition \ref{prop:char-trefoil}.
\end{proof}

This observation allows us to extend strict positivity for various other
invariants. 
The following theorem is a generalization of the result for
almost positive knots \cite[Theorem 1.4]{Przytycki-Taniyama}.

\begin{theorem}\label{th:trefoil}
Let $K$ be a weakly successively almost positive knot.
Then by positive-to-negative crossing changes, $K$ can be made to the
trefoil knot, $d_+(K,3_1)<\iy$. 
\end{theorem}

\begin{proof}
Let $D$ be a minimum complexity w.s.a.p.\ diagram of $K$.
We prove the theorem by induction on the complexity $\mathcal{C}$.
Let $c_1$ and $c_2$ be the positive crossings which are the end points
of the negative overarc. Note that $c_1\ne c_2$, since $K$ does have a Hopf link
factor. Since crossing changes at $c_1$ or $c_2$ give
rise to w.s.a.p.\ diagrams of knots with smaller complexity
(Remark \ref{rem:variant-standard-sequence}), if the crossing
change at $c_1$ or $c_2$ gives a non-trivial knot $\tilde K$, then by induction
on $\tilde K$, we
find a positive-to-negative crossing change sequence which makes $K$
to the trefoil.

Thus we assume that both crossing changes, at $c_1$ or $c_2$, give
rise to the unknot, i.e., $K_{-+}=K_{+-}=\bigcirc$.
Let $K_{--}$ be the w.s.a.p.\ knot obtained by crossing
changes at both $c_1$ and $c_2$. Since the signature cannot increase by
the positive to negative crossing change, $\sigma(K_{--})=0$ which means
that $K_{--}$ is also the unknot. Therefore $K$ is $2$-trivadjacent. By
Corollary \ref{cor:2-trivadjacent}, this means that $K$ is the trefoil.
\end{proof}

This gives more positivity properties, such as,
\begin{corollary}\label{cor:tau}
If $K$ is a non-trivial w.s.a.p.\ knot, then $\tau(K)\geq 1$.
\end{corollary}

The following is a useful remark on the proof of Theorem
\ref{th:trefoil}. It improves Proposition \ref{prop:u-Conway} (i).

\begin{corollary}\label{cor:G-distance-trefoil}
If $K$ is a (non-trivial) w.s.a.p.\ knot, then $K$ has (positive-to-negative)
Gordian distance $d_+(K,3_1)\le a_2(K)-1$ to the (positive)
trefoil. 
\end{corollary}

\begin{proof}
Obviously we constructed in the proof of Theorem
\ref{th:trefoil} a standard unknotting sequence,
so that $K_{m-1}$ is the trefoil, and $m\le a_2(K)$.
\end{proof}

To see that Corollary \ref{cor:G-distance-trefoil} is useful as a w.s.a.p.\ %
test, consider the following examples. 

\begin{figure}[htb]
\[
\begin{array}{cc}
\includegraphics*[width=55mm]{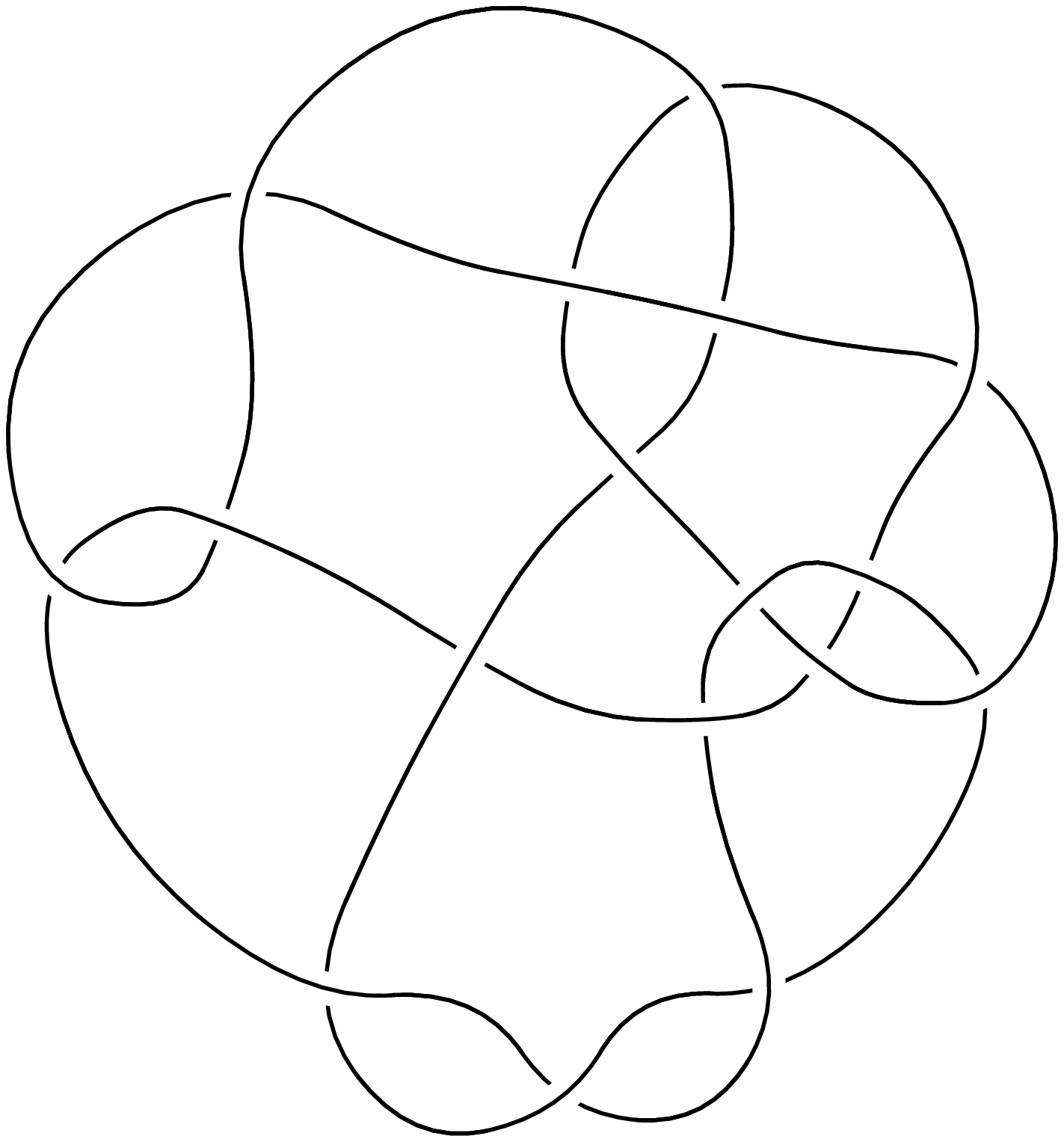} & 

\includegraphics*[width=55mm]{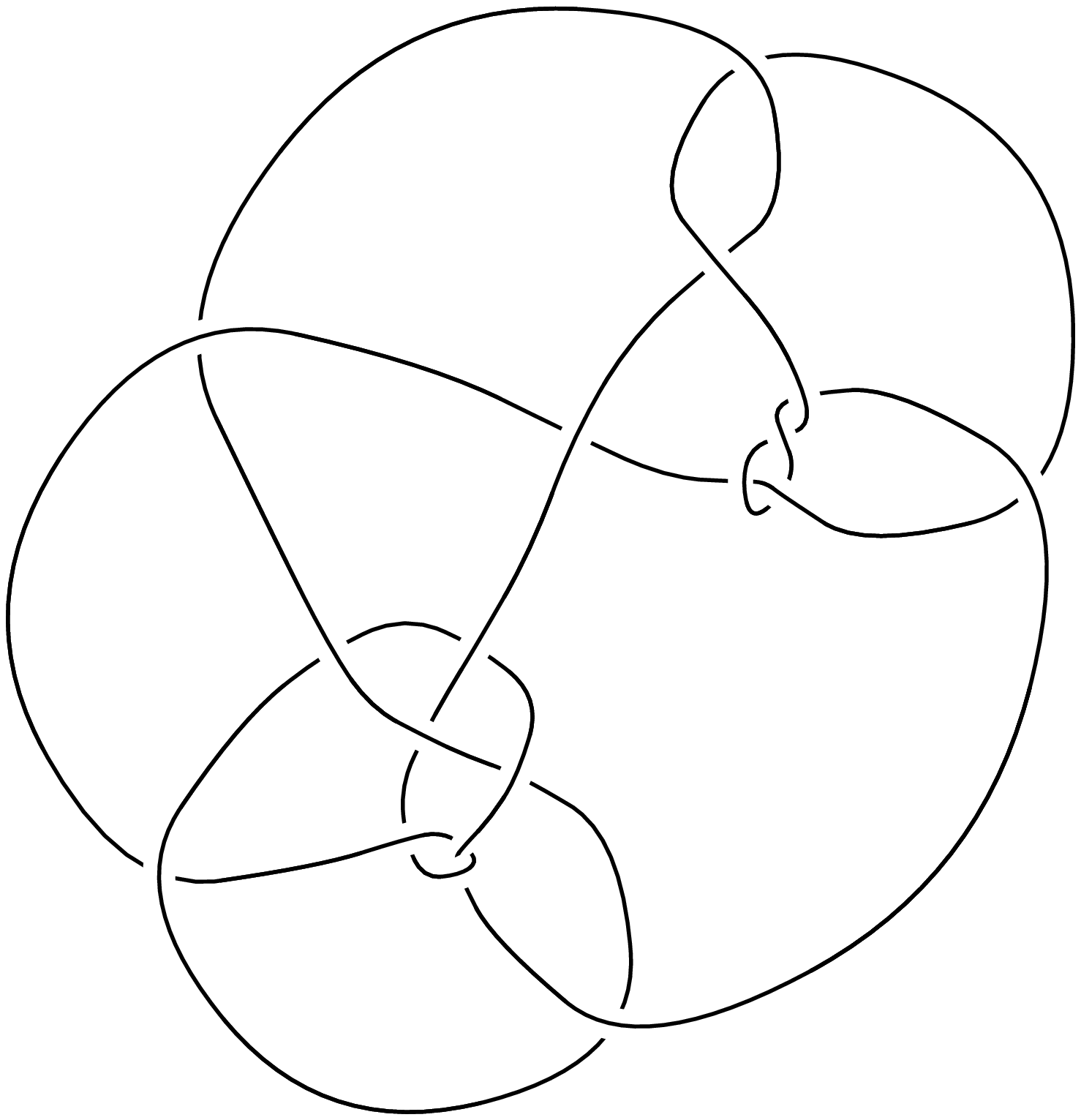} \\

17_{*115227} & 18_{*176710} \\[8mm]
\includegraphics*[width=55mm]{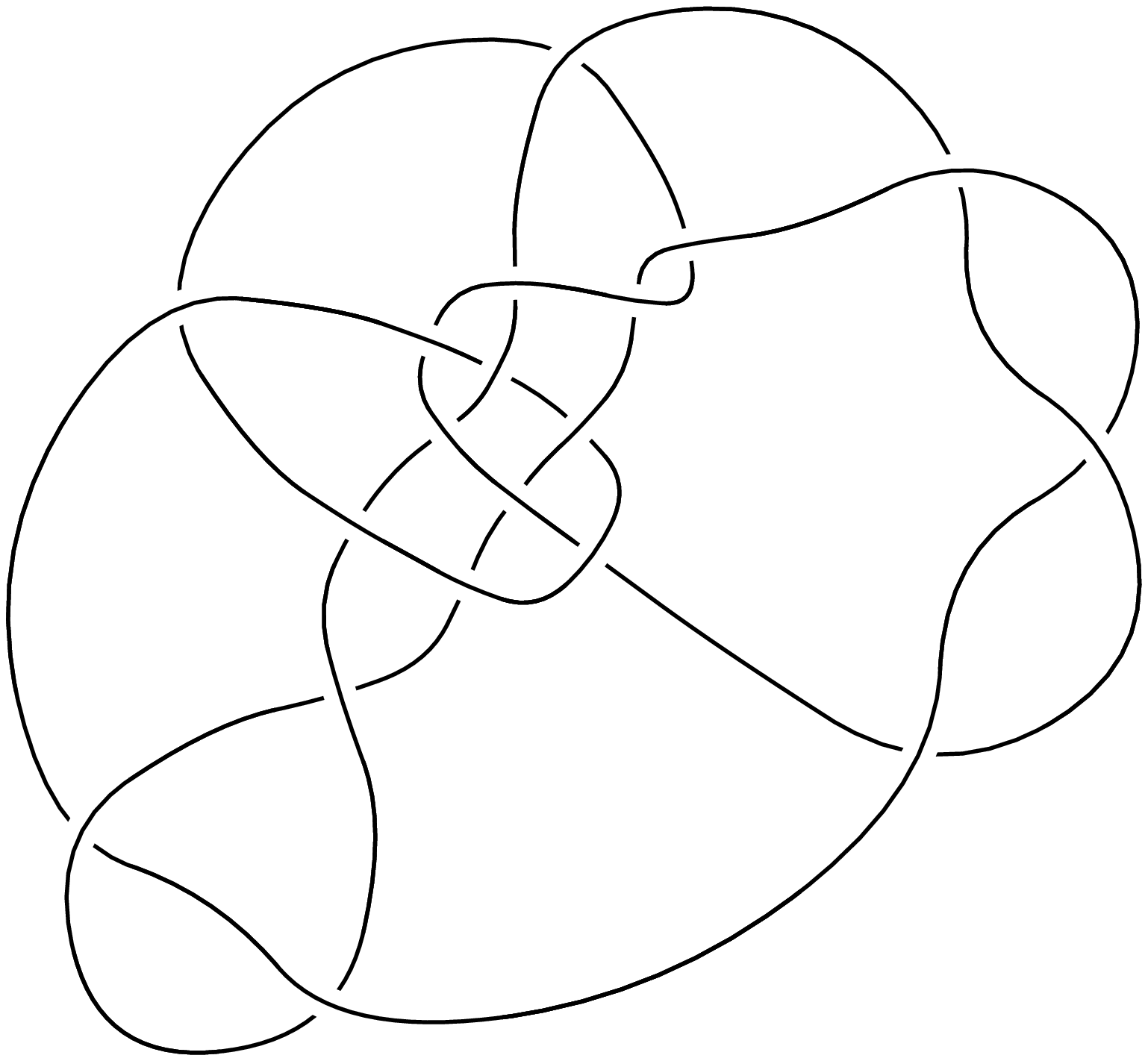} &
\includegraphics*[width=55mm]{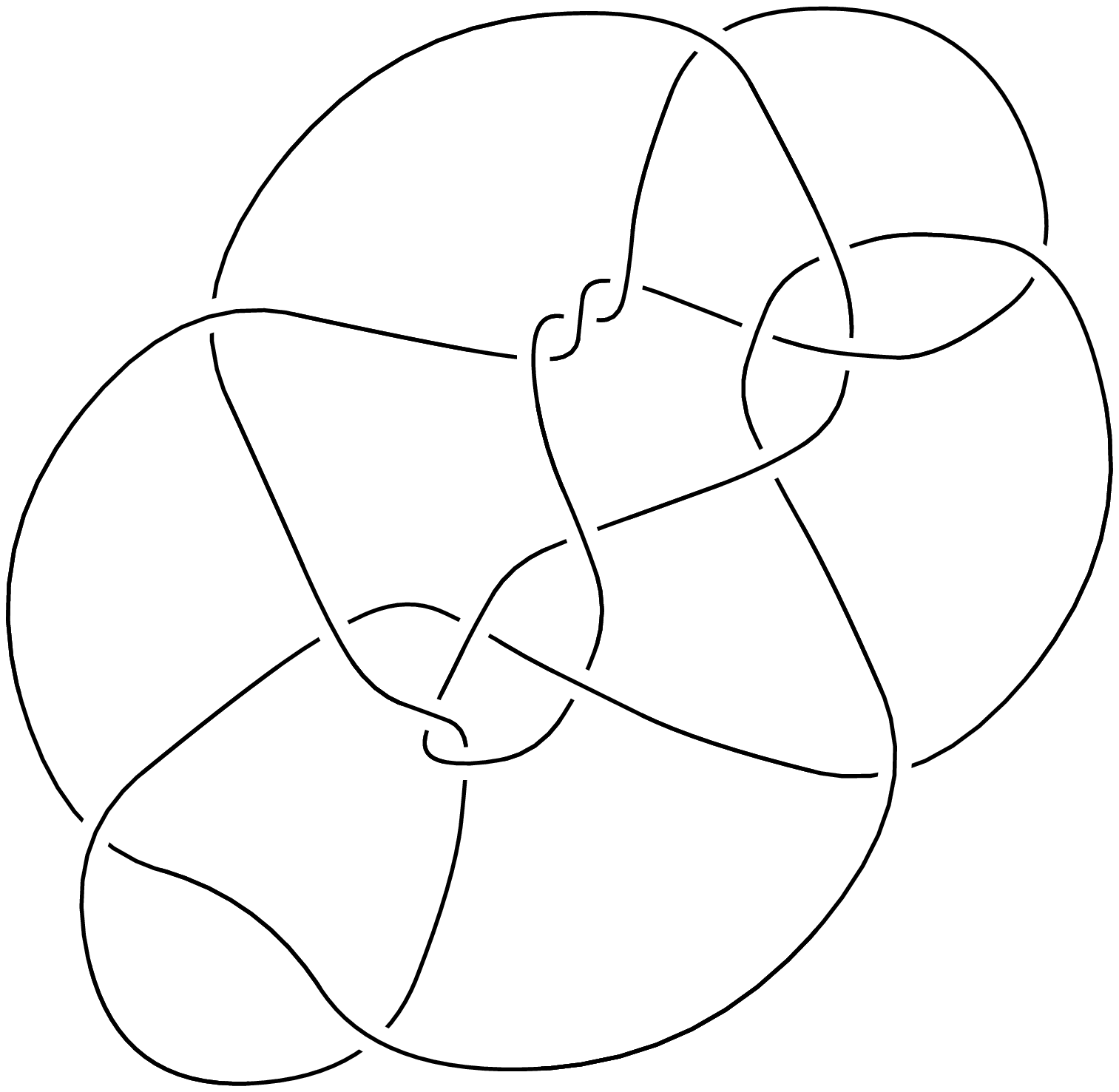} \\

19_{*405610} & 19_{*393831}
\end{array}
\]
\caption{\label{Fig9}Knots used in the construction of Example \ref{exam:criteria-Cor2} and
\ref{exam:criteria-Cor3}, showing that Corollary \ref{cor:G-distance-trefoil}
can be essential as a test for w.s.a.p.\ All these knots also satisfy the
property (a)--(d) stated in Example \ref{exam:criteria-Cor}.}
\end{figure}

\begin{example}\label{exam:criteria-Cor}
Take a knot $K$ having the following properties (there are systematic
constructions of such knots for any such $\nb$ \cite[Theorem 3.1]{st-realization}).
\begin{itemize}
\item[(a)] $u(K)=1$ and $\sigma(K)=2$.
\item[(b)] $\nabla_K(z)$ is strictly positive and $a_2(K) = 1$. 
\item[(c)] $\det(K)=|\nb(2\sqrt{-1})|\equiv 3\bmod 4$ but is not $3$.
\item[(d)] $g(K)=\max \deg_z \nabla_K(z)$, and $K$ is fibered if $\nabla_K(z)$
is monic.
\end{itemize}

Then the knot 
\[
K_n^*=\#^nK
\]
has $\sg(K_n^*)>0$ and its Conway polynomial satisfies all the properties
of a w.s.a.p.\ link which we proved so far.

On the other hand, by Proposition \ref{prop:torsion-test}, $d(K_n^{*},3_1)
\geq n$. Since $a_2(K_n^*)=u(K_n^*)=n$, by Corollary \ref{cor:G-distance-trefoil}
$K_n^*$ is not w.s.a.p.\ for all $n>1$.

The knot $K=10_{156}$ is the simplest example of such a knot. 

We have $10_{156} \not \geq 3_1$, because for $\omega = e^{2\pi\sqrt{-1} t}$ for
 $\frac{1}{3}<t< \alpha$ (here $e^{2\pi\sqrt{-1}\alpha}$ is a root of
the Alexander polynomial of $10_{156}$), $\sigma_{\omega}(3_1)=2 > 0=\sigma_{\omega}(10_{156})$.

The knot $K=14_{28430}$ is another example satisfying all the required
properties, where $d_+(K,3_1)=1$. So $d_+(K^*_n,3_1)\le n$. But we have
no strict inequality; this can be seen either with Proposition \ref{prop:torsion-test},
or by using Levine-Tristram signatures.

\end{example}

\begin{example}\label{exam:criteria-Cor2}

The construction of Example \ref{exam:criteria-Cor} does not control the
HOMFLY polynomial and (like at several other places) it turned out that
among table knots, the HOMFLY test (Theorem \ref{thm:HOMFLY-polynomial})
could not be overcome in examples with provable $u\le a_2$.
However, $K^*_n$ for $K$ being the 17 crossing knot $K=17_{*115227}$ does
the job.
(Also $\tau(K)=1$, so that Corollary \ref{cor:tau}
cannot be used either.)

For any of the (at least 9 provably distinct) examples $K$ we could not
confirm that $K\ge 3_1$ (for one example, this was ruled out using Levine-Tristram
signatures, and the other 8 are undecided).

\end{example}
 
\begin{example}\label{exam:criteria-Cor3}
To construct examples $K_n^*\ge 3_1$ making the HOMFLY obstruction fail,
we have to use connected sums with different factors.
(And again the restrictivity of the test manifests itself in rather complicated
knots, which required a lot of effort to put together.)

The knot $K=19_{*405610}$ has $\tau=1$, $\sg=2$, $u_+=1$,
$\nb=1+z^2+2z^4+z^6$ (so $\det = 35$).
Similarly, the knot $K=18_{*176710}$ has $\tau=1$, $\sg=2$, $u_+=1$, $\nb=1+z^2-z^4+2z^6$
(so $\det=147$), and $d_+(K,3_1)=1$.

Let 
\[
K_n^*=18_{*176710}\#\,\left(\#^{n-1}19_{*405610}\right).
\]
By Proposition \ref{prop:torsion-test} 
\begin{equation}\label{eqn:u+=d+=a2}
u_+(K_n^*)=d_+(K_n^*,3_1)=a_2(K_n^*)=n.
\end{equation}
(So $d_+(K,3_1)$ is the smallest possible for which Corollary \ref{cor:G-distance-trefoil}
would obstruct.)

Also $\nb(K_n^*)$ is strictly positive for $n\ge 2$.
Moreover, for $n \geq 3$ it also satisfies the property stated in Proposition
\ref{prop:Conway-a_j=1}. 

For a different polynomial,
instead of $18_{*176710}$, one can also take $19_{*393831}$
with $\nb=1+z^2-4z^4+3z^6$, where $n\ge 4$ would work.
(Only Proposition \ref{prop:torsion-test} for $\dl=7$
is applied so far to establish
\eqref{eqn:u+=d+=a2} in all examples we have.)
\end{example}

For a w.s.a.p.\ knot with $a_2(K)=u(K)$, we get the following more specific
property.
\begin{proposition}
\label{prop:a=u+sig=2}
Assume that $K$ is a w.s.a.p.\ knot of $\sigma(K)=2$.
If $a_2(K)=u(K)$, then $\det(K) \leq 4a_2(K)-1$. Moreover, if equality
occurs then $g(K)=1$.
 \end{proposition}
\begin{proof}

We prove the proposition by induction on $u(K)=a_2(K)$. 
If $u(K)=a_2(K)=1$, then $K$ is the trefoil, and the assertion follows.
Thus in the following we assume that $u(K)=a_2(K)>1$.

Let $K=K_0 \to K_1 \to \cdots \to K_{u(K)-1} \to K_{u(K)}$ be a standard
unknotting sequence. By Proposition \ref{prop:u-Conway} (i), $u(K_1)=a_2(K_1)
= u(K)-1$ so by induction $\det(K_1) \leq 4a_2(K_1) -1$.

Since $2= \sigma(K)\geq \sigma(K_{1})>0$, $\sigma(K_0)= \sigma(K_1)=2$.
Thus by Theorem \ref{thm:signature-property} (v), $\det(K_0)=-\nabla_{K_0}(2\sqrt{-1})$
and $\det(K_1)=-\nabla_{K_1}(2\sqrt{-1})$.

Let $K'$ be the (w.s.a.p.)
2-component link obtained by smoothing the crossing of
$K=K_0$ where we apply the crossing change. 
Then $K'$ is non-split, $a_1(K')=1$, and $\sigma(K') \in \{1,2,3\}$.

By Proposition \ref{prop:Conway-wp}, $a_1(K')=1$ implies that either $K'$
is the Hopf link, or, it is non-prime. If $K'$ is non-prime we put $K'=L_1\#
L_2$, where $L_1$ is the 2-component w.s.a.p.\ link and $L_2$ is a w.s.a.p.\ %
knot. Since $\sigma(L_1)\geq1$ and $\sigma(L_2)\geq 2$, $\sigma(K')\geq
3$. Thus in this case $\sigma(K')=3$. By Theorem \ref{thm:signature-property}
(vi), this implies that $(2\sqrt{-1})\nabla_{K'}(2\sqrt{-1})>0$.
Therefore
\begin{align*}
\det(K) &= -\nabla_K(2\sqrt{-1}) = -\nabla_{K_1}(2\sqrt{-1}) -(2\sqrt{-1})\nabla_{K'}(2\sqrt{-1})\\
&< \det(K_1) \leq 4a_2(K_1)-1 \leq 4a_2(K)-1
\end{align*}
In particular, the inequality is strict.

If $K'$ is the Hopf link,
\begin{align*}
\det(K) &= -\nabla_K(2\sqrt{-1}) = -\nabla_{K_1}(2\sqrt{-1}) -(2\sqrt{-1})\nabla_{K'}(2\sqrt{-1})\\
&= \det(K_1) +4 \leq 4a_2(K_1)-1 + 4 = 4a_2(K) -1
\end{align*}
In this case, the inequality is exact, i.e., 
\begin{equation}\label{xxx}
\det(K)=4a_2(K)-1
\end{equation}
holds
if and only if $\det(K_1)=4a_2(K_1)-1$. By induction $g(K_1)=1$, hence
$\nabla_{K_1}(z)=1 +a_2(K_1)z^{2}$. We conclude that if \eqref{xxx}
holds then $\nabla_K(z) = \nabla_{K_1}(z) + z \nabla_{K'}(z) = 1 +a_2(K)z^2$,
so $g(K)=1$.
\end{proof}

Proposition \ref{prop:a=u+sig=2} can be used to detect non-w.s.a.p.\ knots.

\begin{example}\label{exam:a2=u+sig=2-criterion}
Some knots whose unknotting number was previously
considered, as listed in \cite{LMo}, and to which the test applies (and
$\nb$ is strictly positive) are $10_{117}$, $11_{378}$, $11_{391}$, $11_{494}$,
$3_1\# 6_3$ (for unknotting number 2) and $12_{1435}$
(for unknotting number 3). For higher $a_2$ one can take
$3_1\# \left(\#^{m-1}10_{164}\right)$ (whose unknotting number being
equal to $a_2=m$ follows from Proposition \ref{prop:torsion-test}
with $\dl=3$).
\end{example}

The HOMFLY polynomial easily prohibits these knots all
from being w.s.a.p.\ as well. To obtain \em{essential}
examples, i.e., such for which all previous conditions
on $P$, $\nb$ and $\sg$ fail to rule out w.s.a.p., we
have to enter more exotic territory. 

\begin{example}\label{exam:a2=u+sig=2-criterion1}
The minimal crossing prime essential examples with
$a_2(K)=u(K)=2$ are $K=15_{128571}$ and $15_{148331}$.
Their unknotting number 2 is determined from the
tau-invariant $\tau(K)=2$ (compare Section \ref{sec:signature}), or alternatively
from Lickorish's linking form test. They also have Gordian distance $1$
to the trefoil.

There is not enough scope of prime knot tables, resp.\ information about
their entries, for us to offer a new instance for unknotting number 3.
The simplest essential examples we could build
for general unknotting number using connected sums are
\begin{equation}
3_1\# (\#^{m-1}16_{658907})\,,
\end{equation}
using $a_2(16_{658907})=1$, $\sg(16_{658907})=0$ and
$\tau(16_{658907})=u(16_{658907})=1$.
Such a composite knot obviously also has
Gordian distance $m-1=a_2-1$ to the trefoil.
 
\end{example}

\section{General signature estimate and concordance finiteness\label{sec:signature-general}}

As perhaps the most general known sharpening of $\sg>0$ for a positive
diagram $D$ of a link $L$, the stronger inequality 
\begin{equation}
\label{eqn:bdl} \sigma(L) \geq \frac{1}{24}(1-\chi(D)) \; \; \left( =
\frac{1}{24}(1-\chi(L)) \right)
\end{equation}
was proven in \cite[Theorem 1.2]{Baader-Dehornoy-Liechti}. 

Since a positive-to-negative crossing change decreases the signature by
at most $2$, (\ref{eqn:bdl}) implies 
\begin{equation}
\label{eqn:bdl2} \sigma(L) \geq \frac{1}{24}(1-\chi(D)) - 2c_-(D)
\end{equation} 
for a general link diagram $D$. 

In this section we give the following improvement of the signature estimate
(\ref{eqn:bdl2}) and discuss its applications.

\begin{theorem}
\label{thm:signature-improved}
Let $D$ be a connected reduced diagram of a non-trivial link $L$. Then 
\begin{equation}
\sigma(L) \geq \frac{1}{12}(1-\chi(D)) - \frac{4}{3}c_-(D)  + \frac{1}{2}
\geq \frac{1}{12}(1-\chi(L))- \frac{4}{3}c_-(D)+ \frac{1}{2}.
\end{equation}
\end{theorem}

\subsection{Proof of signature estimate}

In large part, our proof of Theorem \ref{thm:signature-improved} goes
along the same lines as Baader-Dehornoy-Liechti's argument \cite{Baader-Dehornoy-Liechti},
adapted so that it can be applied to general link diagrams with slight
improvements.

The proof is based on Gordon-Litherland's theorem \cite{Gordon-Litherland}.
For a (possibly non-orientable) spanning surface $F$ of a link $L$, let $\langle
\; , \; \rangle_F :H_1(F) \times H_1(F) \rightarrow \Z$ be the \emph{Gordon-Litherland
pairing}\index{Gordon-Litherland pairing} of $F$; for $a,b \in H_1(F)$,
let $\alpha,\beta$ be curves on $S$ that represent $a$, $b$, and let $p_F:
\nu(F) \rightarrow F$ be the unit normal bundle of $F$. The Gordon-Litherland
pairing of $a$ and $b$ is defined by $\langle a , b \rangle_F = lk(\alpha,
p_S^{-1}(\beta))$. 

For an oriented link $L$, the Gordon-Litherland theorem states
\begin{equation}
\label{thm:GL} -\sigma(L)=  \sigma(F) + \frac{1}{2}e(F,L).
\end{equation}
Here $\sigma(F)$ is the signature of the Gordon-Litherland pairing of
$F$, and 
\[ e(F,L)= - \frac{1}{2}\langle [L] , [L] \rangle_F - lk(L)\,, \]
where $lk(L)=\sum_{i<j}lk(L_i,L_j)$ is the total linking number.

A \em{checkerboard coloring} assigns to regions two colors, black and
white, so that at each crossing, its opposite regions receive the same
color, and its neighbored ones opposite color. We fix one of the (two) checkerboard
colorings of a diagram $D$, and let $B$ and $W$ be the black and white
checkerboard surfaces.

We say that a crossing $c$ of $D$ is \emph{of type a} (resp.\ \emph{of
type b})\index{crossing! type a}\index{crossing! type b} if, when we put
the overarc so that it is a horizontal line, the upper right-hand side
and the lower left-hand side (resp.\ the lower right-hand side and the
upper left-hand side) are colored by black (see Figure \ref{fig:type-crossings}).

Similarly, we say that a crossing $c$ is \emph{of type I} (resp.\ \emph{of
type II})\index{crossing! type I}\index{crossing! type II} if the black
region is compatible (resp.\ incompatible) with the orientation of the
diagram (i.e., its Seifert circle regions are black resp.\ white;
see Figure \ref{fig:type-crossings}). In the definition of type
a/b, the orientation of $D$ is irrelevant, whereas in the definition of type
I/II, the over-under information is irrelevant.

We say that a crossing $c$ is \emph{of type $Ia$}, for example, if $c$
is both of type I and of type a. We put $c_{Ia},c_{Ib},c_{IIa},c_{IIb}$
to be the number of crossings of type $Ia,Ib,IIa,IIb$, respectively. Note that
a positive (resp.\ negative) crossing is either of type Ib or IIa (resp.\
Ia or IIb), so
\begin{equation}
\label{eqn:ab-pm} c_+(D)= c_{Ib} + c_{IIa}, \quad c_{-}(D)=c_{Ia}+c_{IIb}.
\end{equation}

\begin{figure}[htbp]
\includegraphics*[width=90mm]{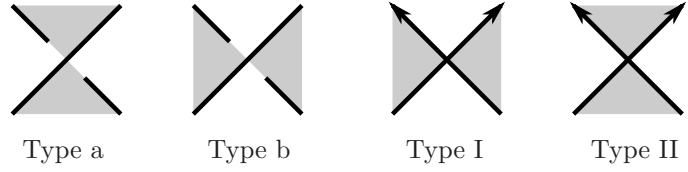}
\begin{picture}(0,0)
\put(-255,-15) {Type a}
\put(-185,-15) {Type b}
\put(-110,-15) {Type I}
\put(-40,-15) {Type II}
\end{picture}
\medskip\medskip\medskip
\caption{Types of crossings with respect to the checkerboard coloring}
\label{fig:type-crossings}
\end{figure} 

By the definition of the Gordon-Litherland pairing,
\[ \frac{1}{2}e(B,L) = c_{IIb} - c_{IIa}, \quad \frac{1}{2}e(W,L) = c_{Ia}
- c_{Ib}.\]
Thus by (\ref{thm:GL}) and (\ref{eqn:ab-pm})
\begin{equation}
\label{eqn:sgn-part-1}
-2\sigma(L) = \sigma(B) + \sigma(W) - c_+(D) + c_-(D) 
\end{equation}

Let $\mathcal{R}(W)$ and $\mathcal{R}(B)$ be the set of white and black
regions, respectively. For a white region $R \in \mathcal{R}(W)$, we associate
a simple closed curve $\gamma_R$ on $B$ which is a mild perturbation of
the boundary of $R$ (see Figure \ref{fig:gamma_R}).  

We say that the region $R$ is \emph{of type $(\alpha,\beta)$} if $R$ contains
$\alpha$ type a crossings and $\beta$ type b crossings as its corners.
By definition, 
\[ \langle [\gamma_R], [\gamma_R] \rangle_B = \alpha-\beta \]
when $R$ is of type $(\alpha,\beta)$.

\begin{figure}[htbp]
\includegraphics*[width=30mm]{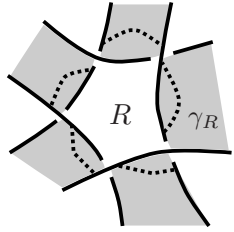}
\begin{picture}(0,0)
\put(-50,40){\large $R$}
\put(-20,40){$\gamma_R$}
\end{picture}
\caption{Curve $\gamma_R$ for a white region $R$}
\label{fig:gamma_R}
\end{figure} 

For a black region $R \in \mathcal{R}(B)$, the curve $\gamma_R$ on $W$
and the notion of type $(\alpha,\beta)$ are defined similarly, and the
Gordon-Litherland pairing is given by
\[ \langle [\gamma_R], [\gamma_R] \rangle_W = \beta-\alpha \]
when $R$ is of type $(\alpha,\beta)$.

\begin{proof}[Proof of Theorem \ref{thm:signature-improved}]

Since $D$ is reduced, there are no regions of type $(1,0)$ or $(0,1)$.
Moreover, since we assume that $L$ is non-trivial, $\# \mathcal{R}(W),\#
\mathcal{
R}(B) \geq 2$. (We use $\#$ henceforth for the cardinality of a set.)

Let $\Gamma$ be the planar graph whose vertices $V(\Gamma)$ are white
regions of type $(\alpha,\beta)$ with $\alpha-\beta \leq 0$, and two vertices
$R,R'$ are connected by an edge if they share a corner. By Appel-Haken's
Four-Color Theorem\footnote{It is interesting to find an argument that
avoids using the Four-Color theorem.} \cite{Appel-Haken} there is a 4-coloring
$col:V(\Gamma) \rightarrow \{1,2,3,4\}$; if two vertices $v,v'$ are connected
by an edge, then $col(v) \neq col(v')$. 

Let $\Gamma'$ be the subgraph of $\Gamma$ such that $V(\Gamma')=col^{-1}\{1,2\}$
and $E(\Gamma')=\{e \in E(\Gamma) \: | \: e \mbox{ connects vertices }
v,w \in V(\Gamma')\}$. With no loss of generality, we assume that 
\[ \# V(\Gamma') \geq \frac{1}{2} \#V(\Gamma). \]

Let $V_B$ be the subspace of $H_1(B)$ generated by $[\gamma_R]$ for $
R \in V(\Gamma')$. Since $\Gamma'$ is bipartite, the restriction of the
Gordon-Litherland pairing $\langle \; , \; \rangle_B$ on $V_B$ is of the
form $\begin{pmatrix}D_1 & X^{T} \\ X & D_2 \end{pmatrix}$, where $D_1,D_2$
are diagonal matrices with non-positive diagonals. Therefore the Gordon-Litherland
pairing is non-positive on $V_B$.

Let $\gamma^{W}_{>0}$ be the number of white regions $R$ such that $\langle
[\gamma_R], [\gamma_R] \rangle_B>0$.

If $V(\Gamma') \neq \mathcal{R}(W)$, $\{ [\gamma_R] \: |\: R \in V(\Gamma')\}$
is a basis of $V_B$, so $\dim V_B = \# V(\Gamma')$. Thus
\[ \dim V_B = \# V(\Gamma') \geq \frac{1}{2}\# V(\Gamma) = \frac{1}{2}(\#
\mathcal{R}(W) - \gamma^{W}_{>0}). \]

If $V(\Gamma') = \mathcal{R}(W)$, then $\dim V_{B}= \dim H_1(B)=\# \mathcal{R}(W)-1$.
Since $\mathcal{R}(W)\geq 2$, we have the same lower bound 
\[ \dim V_B = \# \mathcal{R}(W)-1 \geq \frac{1}{2} \# \mathcal{R}(W) \geq
\frac{1}{2}(\# \mathcal{R}(W) - \gamma^{W}_{>0})  \]

Therefore we get an upper bound\footnote{This is the point where the minor
improvement (the constant $\frac{1}{2}$ in the conclusion) appears.}
\begin{align*}
\sigma(B) &\leq \dim H_1(B) - \dim V_B = (\# \mathcal{R}(W)-1)-\dim V_B\\
& \leq \frac{1}{2}\# \mathcal{R}(W) -1 + \frac{1}{2}\gamma^{W}_{>0} 
\end{align*}
By a parallel argument for the white surface $W$, we get a similar estimate
\[ \sigma(W) \leq \frac{1}{2} \# \mathcal{R}(B)-1 + \frac{1}{2}\gamma^{B}_{>0}\,,
\]
where $\gamma^{B}_{>0}$ is the number of black regions $R$ such that $\langle
[\gamma_R], [\gamma_R] \rangle_W>0$.

Since $\# \mathcal{R}(W)+\# \mathcal{R}(B)-2 = c(D)$, by (\ref{eqn:sgn-part-1})
we get
\begin{equation}
\label{eqn:sigma-gamma}
-2\sigma(K) \leq -\frac{1}{2}c(D) +2c_{-}(D)+ \frac{1}{2} \gamma_{>0}^{B}
+ \frac{1}{2} \gamma_{>0}^{W}-1\,.
\end{equation}

It remains to estimate $\gamma_{>0}^{B} + \gamma_{>0}^{W}$.
Let $\gamma^{W}(\alpha,\beta)$ be the number of white regions of type
$(\alpha,\beta)$. By definition of $\gamma^{W}_{>0}$,
\[ \gamma^{W}_{>0} = \sum_{\substack{\alpha>\beta\geq 0 \\ \alpha\geq
2}} \gamma^W(\alpha,\beta). \]
By counting the number of the crossings of type $a$ that appear as a corner
of white regions, we get
\begin{align*}
2(c_{Ia}+c_{IIa}) &= \sum_{\alpha,\beta\geq 0} \alpha \gamma^{W}(\alpha,\beta)
\\
& \geq 2 \gamma^{W}(2,0) +2 \gamma^{W}(2,1) + 3\sum_{\substack{\alpha>\beta\geq
0 \\ \alpha\geq 3}} \gamma^W(\alpha,\beta) \\
& = -\gamma^{W}(2,0) - \gamma^{W}(2,1) + 3\sum_{\substack{\alpha>\beta\geq
0 \\ \alpha\geq 2}} \gamma^W(\alpha,\beta)\\
& = -\gamma^{W}(2,0) - \gamma^{W}(2,1) + 3\gamma^{W}_{>0}.
\end{align*}
Thus we conclude
\begin{equation}
\label{eqn:gammaW_>0}
\gamma^{W}_{>0} \leq \frac{2}{3}(c_{Ia}+c_{IIa})+\frac{1}{3}\gamma^{W}(2,0)
+\frac{1}{3}\gamma^{W}(2,1) \,.
\end{equation}
The same argument for the black regions shows
\begin{equation}
\label{eqn:gammaB_>0}
\gamma^{B}_{>0} \leq \frac{2}{3}(c_{Ib}+c_{IIb})+\frac{1}{3}\gamma^{B}(2,0)
+\frac{1}{3}\gamma^{B}(2,1) \,,
\end{equation}
where $\gamma^{B}(\alpha, \beta)$ is the number of black regions of type
$(\alpha,\beta)$.

Recall that a positive crossing appears as either of type Ib or of type
IIa. 
If the corner of a white region $R$ is a crossing of type IIa, then the
orientation of the link and the orientation of the white region switch;
in the language of Definition \ref{def:region}, the white region $R$ is
a non-Seifert circle region near a crossing of type IIa.

Thus if all the corners of $R$ are positive crossings, then the region
$R$ must be of type $(\alpha,2\beta')$.
In particular, if a white region $R$ is of type $(2,1)$, at least one
of its corners is a negative crossing. Similarly, if a white region $R$
is of type $(2,0)$, its corners are either both positive, or both negative.

Let $\gamma^{W}_{\pm}(2,0)$ be the number of white regions of type $(2,0)$
whose corners are positive and negative crossings, respectively.
By counting the number of negative crossings that appear as a corner of
type $(2,1)$ regions or type $(2,0)$ regions we get
\[  \gamma^{W}(2,1) + 2\gamma^{W}_{-}(2,0) \leq 2c_{-}(D). \]
Thus
\begin{equation}
\label{eqn:gamma-W} \gamma^W(2,1) + \gamma^W_{-}(2,0) \leq 2c_{-}(D). 
\end{equation}

Let $s_W(D)$ be the number of Seifert circles of $D$ which are the boundary
of a white bigon such that both corners are positive crossings. When both
corners of a white region $R$ of type $(2,0)$ are positive, then the boundary
of $R$ forms a Seifert circle of $D$, so\footnote{A substantial improvement
appears at this point; in \cite{Baader-Dehornoy-Liechti} the authors used an
upper bound of $\gamma^{W}_+(2,0)$ in terms of the crossing numbers, as
we do for $\gamma^{W}_-(2,0)$, instead of the number of Seifert circles.}
\[ \gamma^{W}_{+}(2,0) \leq s_W(D).\]

By the same argument we get the similar inequalities
\begin{equation}
\label{eqn:gamma-B}
\gamma^B(2,1) + \gamma^B_{-}(2,0) \leq 2c_{-}(D) \,,
\end{equation}
and
\[ \gamma^{B}_{+}(2,0) \leq s_B(D)\,, \]
where $\gamma^{B}_{\pm}(2,0)$ and $s_B(D)$ are defined similarly
(with white regions/bigons replaced by black ones).

Since a Seifert circle cannot be the boundary of a white bigon and black
bigon at the same time,
\[ s_W(D)+s_B(D) \leq s(D). \]
The equality happens only if all the Seifert circles are bigons with positive
crossings as their corners. Thus the equality occurs only if $D$ is the
standard torus $(2,2n)$ link diagram (with opposite orientation, so that
it bounds an annulus). In this case, the asserted inequality of the signature
is obvious, so in the following we can assume the slightly stronger inequality
\[ s_W(D)+s_B(D) \leq s(D)-1. \]
Therefore we get
\begin{equation}
\label{eqn:s-circle} \gamma^{W}_{+}(2,0) + \gamma^{B}_{+}(2,0) \leq s(D)-1\,.
\end{equation}

Hence by \eqref{eqn:gamma-W}, \eqref{eqn:gamma-B} and \eqref{eqn:s-circle},
\begin{align*}
&\gamma^{W}(2,0)+\gamma^{W}(2,1) +  \gamma^{B}(2,0)+\gamma^{B}(2,1) \\
& \quad = \gamma^{W}_{-}(2,0)+\gamma^{W}(2,1) + \gamma^{B}_{-}(2,0)+\gamma^{B}(2,1)
+ \gamma^{W}_{+}(2,0) + \gamma^{B}_{+}(2,0)\\
& \quad \leq 4c_-(D) + s(D) -1.
\end{align*}
By (\ref{eqn:sigma-gamma}), \eqref{eqn:gammaW_>0} and \eqref{eqn:gammaB_>0}
we conclude 
\begin{align*}
-2\sigma(K) & \leq -\frac{1}{2}c(D) +2c_-(D) + \frac{1}{2} \gamma_{>0}^{B}
+ \frac{1}{2} \gamma_{>0}^{W}-1  \\
& \leq -\frac{1}{2}c(D) +2c_-(D) + \frac{1}{3}(c_{Ia}+c_{IIa}+c_{Ib}+c_{IIb})\\
& \quad \qquad + \frac{1}{6}(\gamma^{W}(2,0)+\gamma^{W}(2,1) +  \gamma^{B}(2,0)+\gamma^{B}(2,1))-1
\\
& \leq -\frac{1}{2}c(D) +2c_-(D) + \frac{1}{3}c(D) +\frac{1}{6}(s(D)-1)
+ \frac{2}{3}c_{-}(D)-1 \\
 & = \frac{1}{6}(s(D)-c(D)-1) + \frac{8}{3}c_-(D)-1.
\end{align*}

\end{proof}

When $D$ is a successively $k$-almost positive diagram, some additional
arguments give the following slightly better estimate.
\begin{corollary}
\label{cor:signature-improved-succ-positive}
If $K$ has a successively $k$-almost positive diagram $D$, then
\[ \sigma(K) \geq \frac{1}{12}(1-\chi(D))- \frac{13}{12}k + \frac{1}{3}\,.
\]
\end{corollary}

\begin{proof}[Proof of Corollary \ref{cor:signature-improved-succ-positive}]
We show that when $D$ is a successively $k$-almost positive diagram,
we have improvements of inequalities \eqref{eqn:gamma-W}, \eqref{eqn:gamma-B}
and \eqref{eqn:s-circle}, which
appeared in the previous proof, leading to the better estimate as stated.

For a successively $k$-almost positive diagram, $\gamma^W_-(2,0)=0$. We
have that
$\gamma^W(2,1)$ is bounded above by the number of white regions which
have at least one negative crossing as a corner,
and the number of such white regions is at most $k+1$. Thus we get
\begin{equation}
\label{eqn:gamma-W-sap}
\tag{\ref{eqn:gamma-W}${}'$}
\gamma^W_-(2,0)+\gamma^{W}(2,1) \leq k+1 \ (=c_-(D) +1) .
\end{equation}
The same argument shows 
\begin{equation}
\label{eqn:gamma-B-sap}
\tag{\ref{eqn:gamma-B}${}'$}
\gamma^B_-(2,0)+\gamma^{B}(2,1) \leq k+1 \ (=c_-(D) +1) .
\end{equation}

There are $(k+1)$ Seifert circles which are connected to a negative crossing.
None of these Seifert circles is the boundary of a white or black bigon
with positive corners,
and therefore $s_W(D)+ s_B(D) \leq s(D)-(k+1)$. Thus we have a better
bound
\begin{equation}
\label{eqn:s-circle-sap}
\tag{\ref{eqn:s-circle}${}'$}
\gamma^W_+(2,0) +\gamma^{B}_+(2,0) \leq s(D)-(k+1).
\end{equation}

Using (\ref{eqn:gamma-W-sap}), (\ref{eqn:gamma-B-sap}), (\ref{eqn:s-circle-sap})
instead of  (\ref{eqn:gamma-W}), (\ref{eqn:gamma-B}), (\ref{eqn:s-circle})
we get 
\begin{align*}
 -2\sigma(K) & \leq -\frac{1}{2}c(D) +2c_-(D) + \frac{1}{2} \gamma_{>0}^{B}
+ \frac{1}{2} \gamma_{>0}^{W}-1  \\
& \leq -\frac{1}{2}c(D) +2c_-(D) + \frac{1}{3}c(D) +\frac{1}{6}(s(D)-1)
+ \frac{1}{6}c_{-}(D)+\frac{1}{3}-1 \\
 & = \frac{1}{6}(s(D)-c(D)-1) + \frac{13}{6}c_-(D)-\frac{2}{3}.
\end{align*}
\end{proof}

\subsection{Concordance finiteness}

In \cite[Theorem 1.1]{Baader-Dehornoy-Liechti}, as an application of the
signature estimate, the authors showed that every \emph{topological} knot
concordance class contains finitely many positive knots.  Since their
arguments are based on the (Levine-Tristram) signatures, which are invariants
of algebraic knot concordance, they actually showed the same finiteness
result for an \emph{algebraic} knot concordance class.

Theorem \ref{thm:signature-improved} and the non-negativity
of the Levine-Tristram signature (proven in Proposition \ref{prop:positivity-wp})
lead to a more general
concordance finiteness result.
\begin{theorem}
\label{cor:finite-concordance}
For any $\varepsilon > 0$ and $C \in \R$, every algebraic knot
concordance class $\mathcal{K}$ contains only finitely many
weakly successively $k_i$-almost positive knots $K_i$ such that
\begin{equation}
\label{eqn:concordance-finite}
k_i \leq \Bigl(\frac{1}{8}-\eps\Bigr)\cdot g_c(K_i)+C\,.
\end{equation}
\end{theorem}

\begin{proof}

Assume, to the contrary, that the algebraic concordance class $\mathcal{K}$
contains infinitely many weakly successively $k_i$-almost positive knots
$\{K_i\}$.

Let $D_i$ be a successively $k_i$-almost positive diagram of
$K_i$. Since $g_c(K_i) \leq g_c(D_i)$,
by the assumption (keeping $\eps \leq \frac{1}{8}$)
\[ k_i \leq \Bigl(\frac{1}{8}-\eps\Bigr) \cdot g_c(K_i)+C \leq 
\Bigl(\frac{1}{8}-\eps\Bigr) \cdot g_c(D_i)+C\,, \]
so we get
\[ g_c(D_i)-8k_i \geq 8\eps g_c(D_i) -8C. \]
By Theorem \ref{thm:signature-improved}
\begin{align*}
\sigma(\mathcal{K}) = \sigma(K_i) 
& \geq \frac{1}{6}g(D_i) - 
\frac{4}{3}k_i +\frac{1}{2} = \frac{1}{6}\left(g_c(D_i)-8k_i\right)+\frac{1}{2}\\
& \geq \frac{1}{6}\left(8\varepsilon g_c(D_i)-8C \right) +\frac{1}{2}\,.
\end{align*}
Hence we have the uniform upper bound on the canonical genus of the diagrams
$\{D_i\}$
\[ g(D_i) \leq \frac{3}{4\varepsilon}\left(\sigma(\mathcal{K})-\frac{1}{2}\right)
+\frac{1}{\varepsilon}C .\]

The boundedness of the canonical genus implies that there is a finite
set of diagrams $\mathcal{D}$ such that each $D_i$ is obtained from one
of a diagram $D_{i,{\sf base}} \in \mathcal{D}$ by $\bt$-twist
operations (see Theorem \ref{thm:generator} for details). 

Here the \emph{$\bt$-twist}\index{$\bt$-twist} is an operation that replaces
a crossing of $D$ with three successive crossings (inserting a full twist)
as 
\begin{equation}
\label{eqn:bt-twists}
\raisebox{-3mm}{
\begin{picture}(24,28)
\put(0,0){\vector(1,1){24}}
\put(10,14){\vector(-1,1){10}}
\put(24,0){\line(-1,1){10}}
\end{picture} }  \longrightarrow 
 \raisebox{-3mm}{
\begin{picture}(72,28)
\put(0,0){\line(1,1){24}}
\put(10,14){\vector(-1,1){10}}
\put(14,10){\line(1,-1){10}}
\put(24,24){\line(1,-1){10}}
\put(24,0){\line(1,1){24}}
\put(38,10){\line(1,-1){10}}
\put(48,0){\vector(1,1){24}}
\put(48,24){\line(1,-1){10}}
\put(62,10){\line(1,-1){10}}
\end{picture} }\,.
\end{equation}
Note that the $\bt$-twist preserves the genus of the diagram.

Since every $D_i$ is successively almost positive, we may assume that
the base diagram $D_{i,{\sf base}} \in \mathcal{D}$ is successively almost
positive, and that $D_i$ is obtained from $D_{i,{\sf base}}$ by \emph{positive
$\bt$-twists}, the $\bt$-twists inserting \emph{positive} crossings (as
depicted in 
(\ref{eqn:bt-twists})).

Due to the finiteness of the set $\mathcal{D}$, there are only finitely
many places to apply $\bt$-twists. Thus there is a diagram $D_{\sf base}
\in \mathcal{D}$ and a crossing $c$ of $D_{\sf base}$ having the following
property:
for any $N>0$, there is a knot $K'(N)$ in $\{K_i\}$ such that $K'(N)$
is obtained from $D_{\sf base}$ by $\bt$-twists at least $N$-times at
$c$, and by applying further positive $\bt$-twists.

Let $D_{0}$ be the link diagram obtained by smoothing the crossing $c$
of $D_{\sf base}$ and let $L$ be the link represented by $D_{0}$.
Since applying positive $\bt$-twists at a positive crossing lying on the
negative overarc yields not weakly successively almost positive diagram,
we may assume that the crossing $c$ does not lie on the negative overarc.
This implies that $D_0$ is a weakly successively almost positive diagram.

Let $D(N)$ be the weakly successively almost positive diagram obtained
from $D_{\sf base}$ by $N$ positive $\bt$-twists at $c$, and let $K(N)$
be the knot represented by $D(N)$.
Then $K'(N)$ is obtained from $K(N)$ by positive $\bt$-twists, hence
\begin{equation}
\label{eqn:K-N} \sigma(K(N)) \leq\sigma(K'(N)) = \sigma(\mathcal{K}).
\end{equation}

Since the canonical Seifert surface $S_N$ of $D(N)$ is obtained from the
canonical Seifert surface $S_0$ of $D_0$ by adding a (positively) $N$-twisted
band, the Seifert matrix $A(N)$ of $S_N$ is of the form

\begin{equation*}
 A(N)= \begin{pmatrix}N+a & w \\ v & A(S_0) \end{pmatrix} \,,
\end{equation*}
where $a$ is a constant and $A(S_0)$ is the Seifert matrix of $S_0$.

Take a non-algebraic $1\neq \omega \in \{z \in \C \: | \: |z|=1\}$ sufficiently
close to $1$, so that $\sigma_{\omega}(\mathcal{K})=0$ holds. By definition,
$\sigma_{\omega}(K(N))$ and $\sigma_{\omega}(L)$ are the signatures of

\begin{align*}
A_{\omega}(N) & = (1-\omega)A(N) + (1-\overline{\omega})A(N)^T \\
 & = \begin{pmatrix} (2-2\mathsf{Re}(\omega))(N+a) & (1-\omega)w+(1-\overline{\omega})v^{T}
\\
(1-\omega)v+(1-\overline{\omega})w^{T} & (1-\omega)A(S_{0}) + (1-\overline{\omega})A(S_{0})^T
\end{pmatrix}\\
A_{\omega}(S_0)&= (1-\omega)A(S_{0}) + (1-\overline{\omega})A(S_{0})^T.
\end{align*}
Therefore by the cofactor expansion
\[  \det A_{\omega}(N) = (2-2\mathsf{Re}(\omega))(N+a)\det A_{\omega}(S_0)+
C\,,\]
where $C$ is a constant that does not depend on $N$.
Since we have chosen $\omega$ so that it is non-algebraic, $\det A_{\omega}(S_{0})
\neq 0$. Thus when $N$ is sufficiently large,
the sign of $\det A_{\omega}(N)$ and $\det A_{\omega}(S_{0})$ are the
same, which means that the matrix $A_{\omega}(N)$ has one more positive
eigenvalue than $A_{\omega}(S_0)$.
Thus for sufficiently large $N$
\[ \sigma_{\omega}(K(N))= \sigma_{\omega}(L)+1. \]
On the other hand, since $D_0$ is weakly successively almost positive,
it is weakly positive. Therefore $\sigma_{\omega}(L) \geq 0$ by Proposition
\ref{prop:positivity-wp}.
By (\ref{eqn:K-N})
\[ \sigma_{\omega}(\mathcal{K}) =\sigma_{\omega}(K'(N)) \geq \sigma(K(N))
=\sigma_{\omega}(L)+1 \geq 1. \]
Since we have chosen $\omega$ so that $\sigma_{\omega}(\mathcal{K})=0$,
this is a contradiction.
\end{proof}

For a fixed $k \geq 0$, by taking $\varepsilon = \frac{1}{8}$ and $C=k$
in Theorem \ref{cor:finite-concordance}, we see the condition (\ref{eqn:concordance-finite})
always satisfied for \emph{all} weakly successively $k$-almost positive
knots. Therefore we get the following generalization of concordance finiteness.

\begin{corollary}\label{cor:finite-concordance3}
For a fixed $k \geq 0$, every algebraic knot concordance class $\mathcal{K}$
contains only finitely many weakly successively $k$-almost positive knots.
\end{corollary}

In particular, the $k=1$ case shows 

\begin{corollary}\label{cor:finite-concordance-almost-positive}
Every algebraic knot concordance class $\mathcal{K}$ contains only finitely
many almost positive knots.
\end{corollary}

\begin{remark}\label{rem:2-almost-positive}
If we remove `weakly successively', Corollary \ref{cor:finite-concordance3}
is already false for $2$-almost positive, since the twist knots $K_{m}$
for $4m+1=\ell^{2}$ for $\ell \in \Z$ are algebraically slice but all
of them are $2$-almost positive.

What happens in the stricter topological or smooth categories, whether
every topological or smooth concordance class contains finitely many $k$-almost
positive knots or not (for a fixed $k$), is less clear.
\end{remark}

The same proof gives a slightly better constant in Theorem \ref{cor:finite-concordance},
if we assume the s.a.p.\ property, since for s.a.p.\ knots we have a slightly
better signature estimate (Corollary \ref{cor:signature-improved-succ-positive}).

\begin{corollary}
\label{cor:finite-concordance2}
For any $\varepsilon > 0$ and $C \in \R$, every algebraic knot
concordance class $\mathcal{K}$ contains only finitely many
successively $k_i$-almost positive knots $K_i$ such that
\begin{equation}\label{eqn:concordance-finite2}
k_i \leq \Bigl(\frac{2}{13}-\eps\Bigr)\cdot g_c(K_i)+C\,.
\end{equation}
\end{corollary}

\section{Class inclusions and Comparisons}
\label{sec:comparisons}

In this section we show that various classes of links introduced in the
paper are indeed distinct, as summarized in the chart of Section \ref{sec:intro}.

\begin{theorem}
\label{thm:class-inclusions}
There exists 
\begin{enumerate}
\item[(i)] an almost positive knot of type I which is not positive,
\item[(ii)] an almost positive knot (of type II) which is not almost positive
of type I,
\item[(iii)] a loosely successively 2-almost positive knot which is not
good successively almost positive,
\item[(iv)] a good successively 2-almost positive knot which is not almost
positive of type I,
\item[(v)] a strongly quasipositive knot which is not good successively
almost positive.
\item[(vi)] a weakly 2- positive knot which not weakly successively 
almost positive.
\end{enumerate}
\end{theorem}

Among the six parts, (i), (ii), (v), (vi) are known or easily follow from
known results.

The most difficult and essential results are (iii) and (iv), distinguishing
almost positive from successively almost positive links.

Our argument so far tells us that a weakly successively almost positive
diagram is a natural and more appropriate generalization of a positive diagram
than a general $k$-almost positive diagram, and shares various properties
of positive diagrams.
However, this, in turn, makes it hard to distinguish (almost) positive
links and (weakly or good) successively almost positive links.

\begin{remark}\label{rem:inclusion-arbitrary-k}
It should be pointed out that several inclusions in the chart of Section
\ref{sec:intro} can be considered for fixed number $k$ of negative crossings.
In general, strict inclusions for fixed $k$ neither imply nor are implied
by 
strict inclusion for arbitrary $k$. These are thus a separate set of problems,
some of which remain very non-trivial, in particular for
higher values of $k$, even if we do not focus on them much below.
\end{remark}

\subsection{Positive vs.\ almost positive\label{sec:p-vs-ap}}

Since positive links and almost positive links already share many properties,
distinguishing almost positive links from positive links is not an easy
or obvious task, although there are several methods.

The simplest example of almost positive, but not positive knot is $10_{145}$;
it admits almost positive diagrams of type I and type II. The non-positivity
can be detected by Cromwell's theorem \cite{Cromwell} that $c(K) \leq
4g(K)$ if $K$ is positive and its Conway polynomial $\nabla_K(z)$ is monic
(i.e., $K$ is fibered).

We will return to this problem in Section \ref{sec:positive-list}.

\subsection{Almost positive of type I vs.\ almost positive of type II}
\label{sec:ap-I-vs-II}

Our argument so far tells that a good successively almost positive diagram
has better properties than a loosely successively almost positive diagram.

However, we note that it can happen that loosely successively almost positive
links share a property of positive links, which we cannot prove for good
successively almost positive links.

Let $D_K(a,z)$ be the Dubrovnik version of the Kauffman polynomial;
$D_K(a,z)=a^{-w(D)}\Lambda_{D}(a,z)$, where $\Lambda_D(a,z)$ is the regular
isotopy invariant defined by the skein relations
\[ \Lambda_{ \raisebox{-1mm}{\includegraphics*[width=4mm]{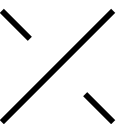}}}(a,z)
- \Lambda_{ \raisebox{-1mm}{\includegraphics*[width=4mm]{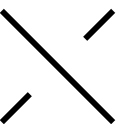}}}(a,z)
= z(\Lambda_{ \raisebox{-1mm}{\includegraphics*[width=4mm]{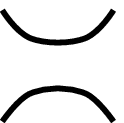}}}
-\Lambda_{ \raisebox{-1mm}{\includegraphics*[width=4mm]{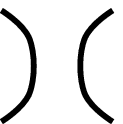}}}),
\]

\[  \Lambda_{ \raisebox{-1mm}{\includegraphics*[width=4mm]{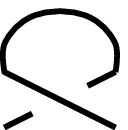}}}(a,z)
= a \Lambda_{ \raisebox{-1mm}{\includegraphics*[width=4mm]{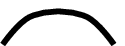}}}(a,z),
\q
\Lambda_{ \raisebox{-1mm}{\includegraphics*[width=4mm]{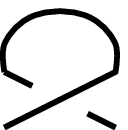}}}(a,z)
= a^{-1} \Lambda_{ \raisebox{-1mm}{\includegraphics*[width=4mm]{eps/skeinst.eps}}}(a,z),
\q
\Lambda_{ \raisebox{-1mm}{\includegraphics*[width=4mm]{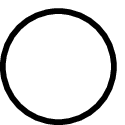}}}(a,z)
= 1. \]

We express the Dubrovnik polynomial as
\[D_L(a,z)= \sum_{i} D_L(z;i)a^{i}. \]
Similarly, we express the HOMFLY polynomial as 
\[ P_L(v,z)=\sum_{i}P_L(z;i)v^{i}. \]

Yokota showed the following property of positive links.

\begin{theorem}\cite{Yokota}
\label{thm:Yokota}
For a positive link $L$
\begin{equation}
\label{eqn:Yokota}
P_L(z;1-\chi(L))=D_L(z;1-\chi(L))
\end{equation}
holds. Moreover, this is a non-negative polynomial.
\end{theorem}

The best way to understand this coincidence and non-negativity is to use
the Legendrian link point of view. For basics on Legendrian links we refer
to \cite{Etnyre}.

\begin{definition}
\label{def:ruling}
Let $\mathcal{D}$ be the front diagram of a Legendrian link $L$. The subset
$\rho$ of the set of crossings of $\mathcal{D}$ is a \emph{ruling}\index{ruling}
if the diagram $\mathcal{D}_{\rho}$ obtained from $\mathcal{D}$ by taking
the horizontal smoothing $\LCross \to \Smooth$ at each crossing in $\rho$
satisfies the following conditions:
\begin{itemize}
\item Each component of $\mathcal{D}_{\rho}$ is the standard diagram of
the Legendrian unknot (i.e., it contains two cusps and no crossings).
\item For each $c \in \rho$, let $P$ and $Q$ be the components of $D_{\mathcal{\rho}}$
that contain the smoothed arcs at $c$. 
Then $P$ and $Q$ are different components of $\mathcal{D}_{\rho}$, and
in the vertical slice around $c$, the two components $P$ and $Q$ are not
nested -- they are aligned in one of the configurations in Figure \ref{fig:ruling}
(ii).
\end{itemize}

\begin{figure}[htbp]
\includegraphics*[width=70mm]{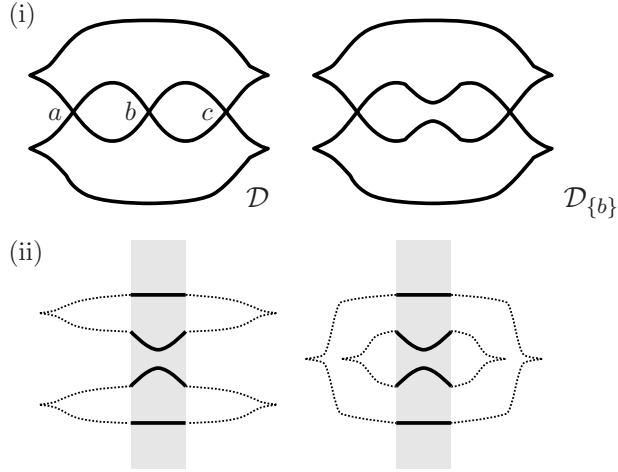}
\begin{picture}(0,0)
\put(-210,170) {(i)}
\put(-210,80) {(ii)}
\put(-195,133) {$a$}
\put(-166,133) {$b$}
\put(-137,133) {$c$}
\put(-120,100) {\large $\mathcal{D}$}
\put(0,100) {\large $\mathcal{D}_{\{b\}}$}
\end{picture}
\caption{(i) Front diagram $\mathcal{D}$ and $\mathcal{D}_{\rho}$ for
$\rho=\{b\}$; here $\rho$ is not a ruling. (ii) Normality condition for
the ruling.}
\label{fig:ruling}
\end{figure} 

A ruling is \emph{oriented}\index{ruling! oriented ruling} if every element
of $\rho$ is a positive crossing. Let $\Gamma(\mathcal{D})$ and $O\Gamma(\mathcal{D})$
be the sets of rulings and oriented rulings, respectively.
\end{definition}

For a ruling $\rho$ we define 
\[ j(\rho)=\# \rho - l\mbox{-cusp}(\mathcal{D}) +1,\]
 where $l\mbox{-cusp}(\mathcal{D})$ denotes the number of the left cusps
of $\mathcal{D}$.

Let $tb(\mathcal{D}) = w(\mathcal{D}) -  l\mbox{-cusp}(\mathcal{D})$ be
the Thurston-Benneuqin invariant. Rutherford showed the following formula
of the part of the Kauffman(Dubrovnik)/HOMFLY polynomials.

\begin{theorem}\cite[Theorem 3.1, Theorem 4.3]{Rutherford}
\begin{equation}
\label{eqn:ruling-formula}
D_L(z;tb(\mathcal{D})+1) = \sum_{\rho \in \Gamma(\mathcal{D})} z^{j(\rho)},
\quad P_L(z;tb(\mathcal{D})+1) = \sum_{\rho \in O\Gamma(\mathcal{D})}
z^{j(\rho)}. 
\end{equation}
\end{theorem}

Thus, we have the following explanation of the coincidence and non-negativity.

\begin{corollary}
\label{cor:D=P-criterion}
If a link $L$ admits a front diagram $\mathcal{D}$ such that $\Gamma(\mathcal{D})=O\Gamma(\mathcal{D})$,
then $D_L(z;tb(\mathcal{D})+1) = P_L(z;tb(\mathcal{D})+1)$ and it is a
non-negative polynomial.
\end{corollary}

This gives a proof of Theorem \ref{thm:Yokota} as follows.

\begin{proof}[Proof of Theorem \ref{thm:Yokota}]
As is discussed and proven in \cite{Tanaka}, one can view a positive diagram
$D$ as a front diagram $\mathcal{D}$; we put each positive crossing so
that it is in the form \Crossing, and put each Seifert circle to form
a front diagram having exactly two cusps so $l\mbox{-cusp}(\mathcal{D})
=s(D)$. 

Then the set of all crossings forms a (oriented) ruling. Thus $\mathcal{D}$
attains the maximum Thurston-Bennequin number. In particular,
\[ \overline{tb}(K) = tb(\mathcal{D}) = c(D)-s(D) = -\chi(D) = -\chi(K).\]
Moreover, since all the crossings of $\mathcal{D}$ are positive, $\Gamma(\mathcal{D})
= O\Gamma(\mathcal{D})$, hence Corollary
\ref{cor:D=P-criterion} gives the desired equality and non-negativity.
\end{proof}

A mild generalization shows that an almost positive link of type II (and
a suitable loosely successively positive link) shares the same property.

\begin{theorem}
\label{thm:almost-type-II}
If $K$ admits an almost positive diagram of type II, then 
\[ P_K(z;1-\chi(K))=D_K(z;1-\chi(K)) \neq 0\]
and it is a non-negative polynomial.
\end{theorem}
\begin{proof}
We view an almost positive diagram of type II as a front diagram $\mathcal{D}$
as shown in Figure \ref{fig:almost-positive-front}, where the box represents
a positive diagram part viewed as a front diagram as we mentioned earlier.
(See \cite{Tagami} for a detailed discussion on how to achieve this, where
the Lagrangian fillability of almost positive knots of type II is shown
by similar methods and arguments.) 

A ruling of $\mathcal{D}$ cannot contain the negative crossing of $\mathcal{D}$,
so $\Gamma(\mathcal{D}) = O\Gamma(\mathcal{D})$. On the other hand, since
$D$ is of type II, the Seifert circles $s$ and $s'$ connected by the negative
crossing $c_-$ are also connected by a positive crossing. We take such
a crossing $c$ so that there is no such crossing between $c$ and $c^{-}$.
Let $\rho=C(\mathcal{D}) \setminus \{c_-,c\}$. Then $\rho$ is a ruling,
so that $\Gamma(\mathcal{D}) = O\Gamma(\mathcal{D}) \neq \emptyset$, and
\[ \overline{tb}(K) = tb(\mathcal{D}) = (c(D)-2)-s(D) = -\chi(D)-2 = -\chi(K).\]
\end{proof} 

\begin{corollary}\label{cor:type-II}
If $K$ is almost positive of type II, 
\[ \min \deg_a D_K(z,a)=\max \deg_z P_K(v,z)=1-\chi(K).\]
\end{corollary}

\begin{figure}[htbp]
\includegraphics*[width=100mm]{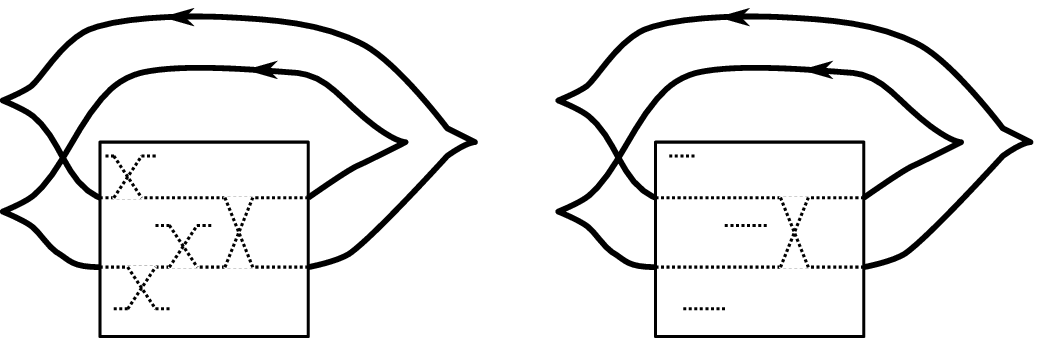}
\begin{picture}(0,0)
\put(-300,90) {(i)}
\put(-140,90) {(ii)}
\put(-285,50) {$c_-$}
\put(-215,25) {$c$}
\put(-170,0) {\large $\mathcal{D}$}
\put(-20,0) {\large $\mathcal{D}_{\rho}$}
\end{picture}
\caption{(i) Front diagram of almost positive diagram of type II; the
box represents a front diagram consisting of positive crossings. (ii)
Non-trivial ruling $\rho=C(\mathcal{D}) \setminus \{c_-,c\}$ and its resolution
$\mathcal{D}_{\rho}$ }
\label{fig:almost-positive-front}
\end{figure}

By \cite[Theorem 1.4]{st-minimum-genus}, there exists an almost positive
knot that does not admit an almost positive diagram of type II. The example
given there (the pretzel knot $P(3,3,3,3,-1)$) satisfies the equality
(\ref{eqn:Yokota}). 
The second author also showed that these exists an almost positive knot
that does not admit an almost positive diagram of type I.

We do not have an example of an almost positive knot that fails to have
property (\ref{eqn:Yokota}), so the following question remains open.
\begin{question}
Does equality (\ref{eqn:Yokota}) hold for almost positive links?
\end{question}

\subsection{Good successively almost positive vs.\ loosely successively
almost positive \label{sec:gsap-vs-loosely-sap}}

\begin{theorem}
\label{thm:loose-vs-good}
There exists a loosely successively almost positive knot which is not
a good successively almost positive knot. 
More precisely, the following three knots $15_{132907}, 15_{96757},15_{125012}$
given in Figure \ref{Fig1} are 15 crossing, genus 2 knots which are loosely
successively almost positive but are not good successively almost positive.
\end{theorem}

\begin{figure}[htb]
\[
\begin{array}{ccc}
\includegraphics*[width=40mm]{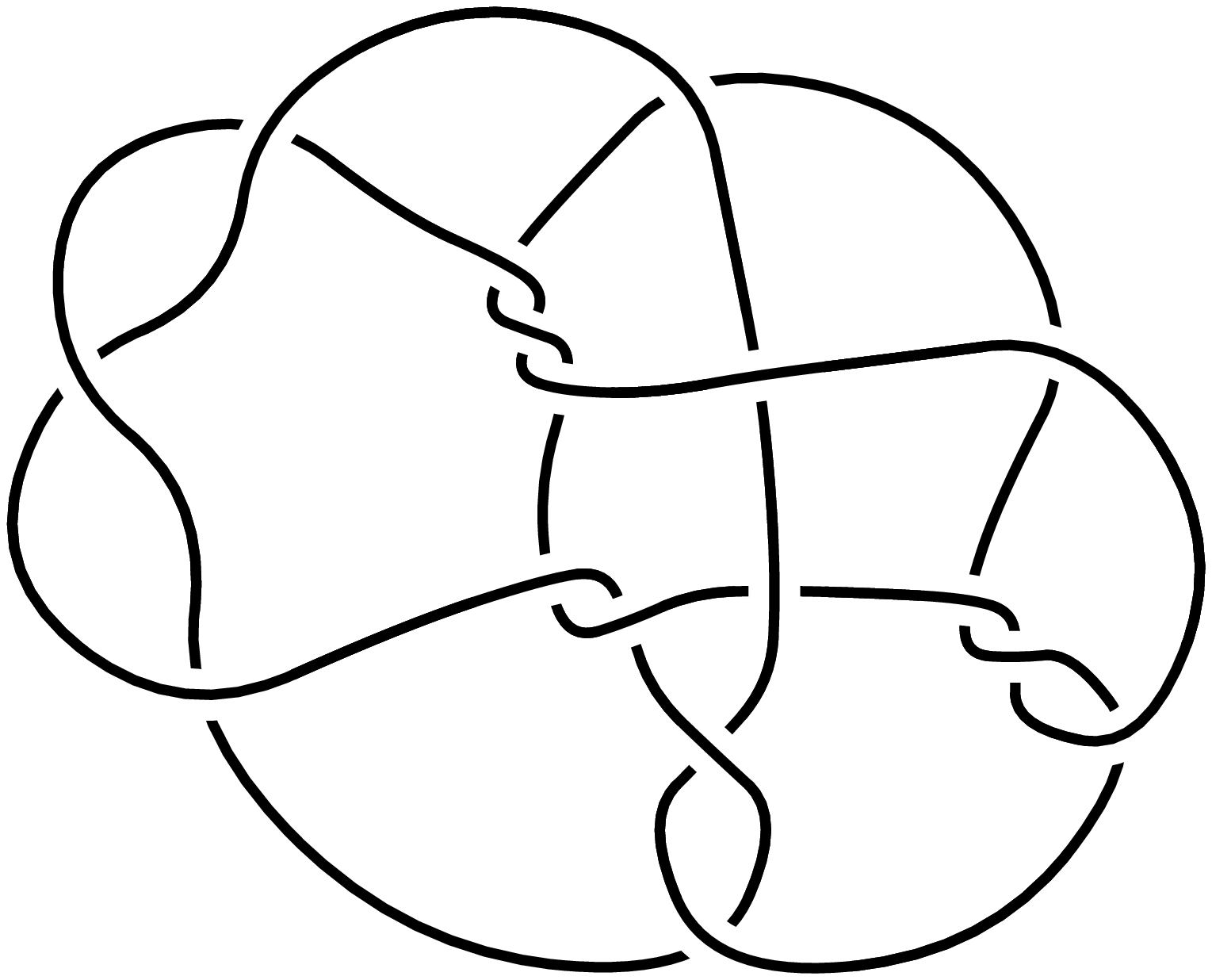} & 
\includegraphics*[width=40mm]{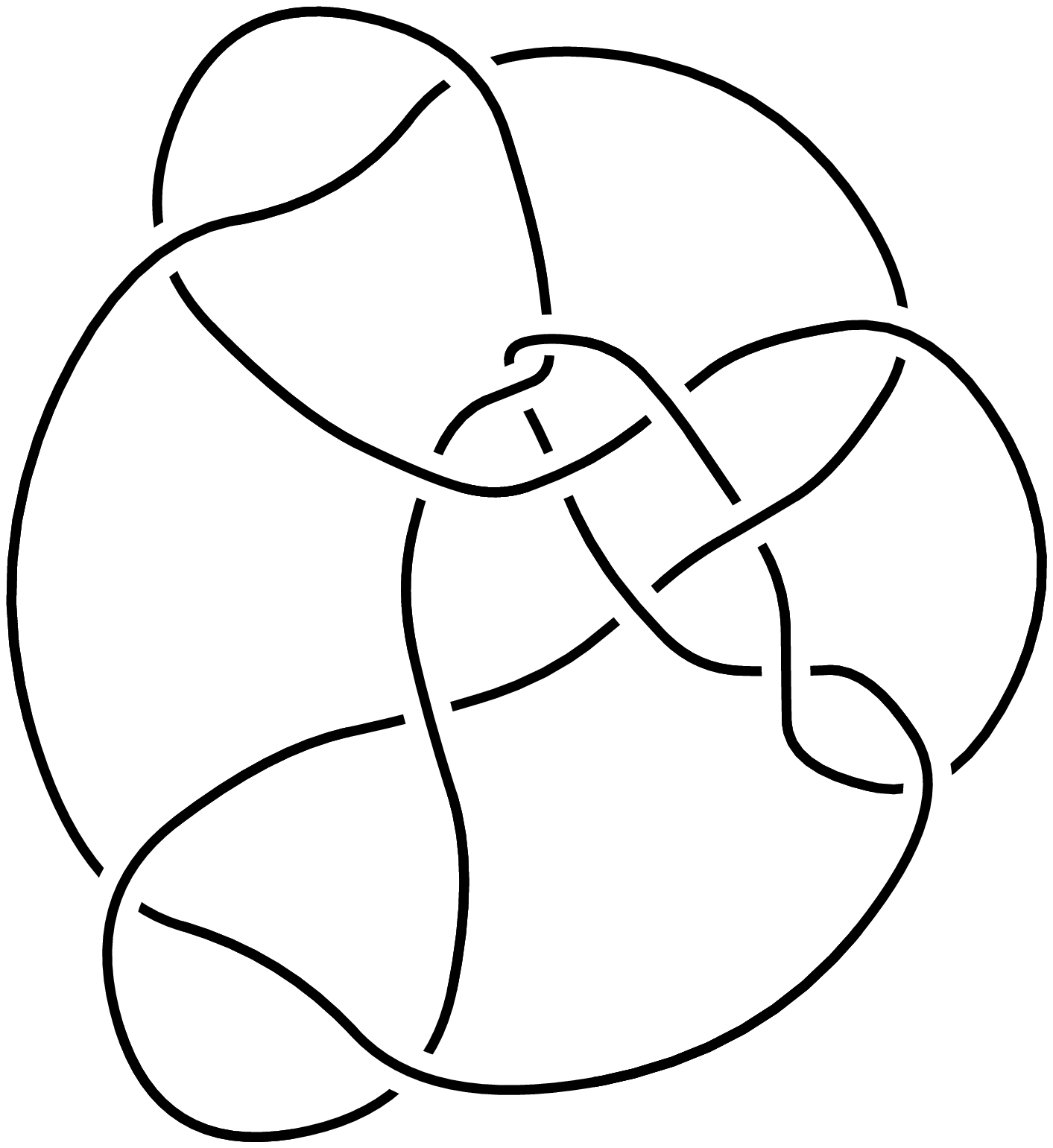} & 15_{132907} \\
\includegraphics*[width=40mm]{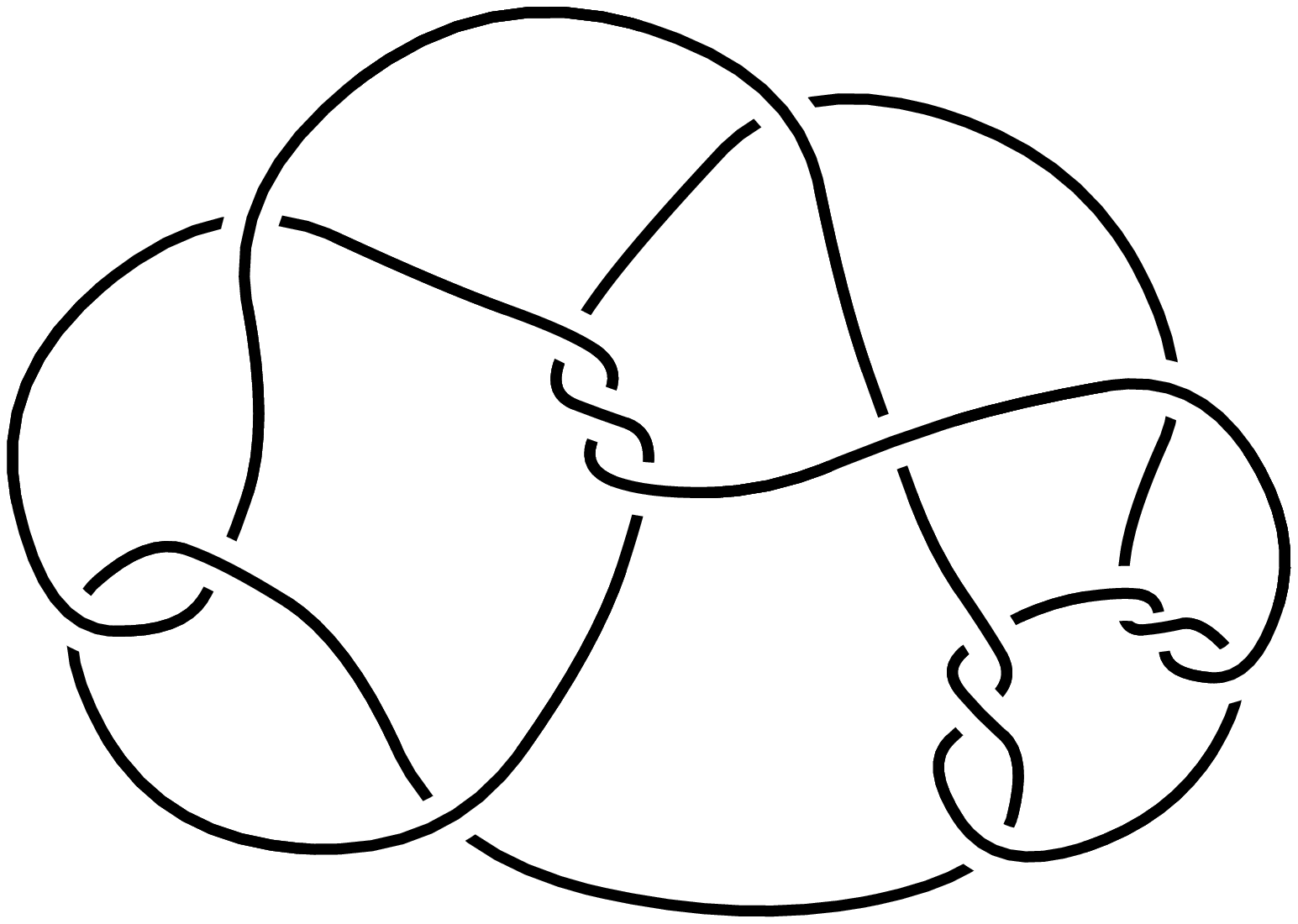} &
\includegraphics*[width=40mm]{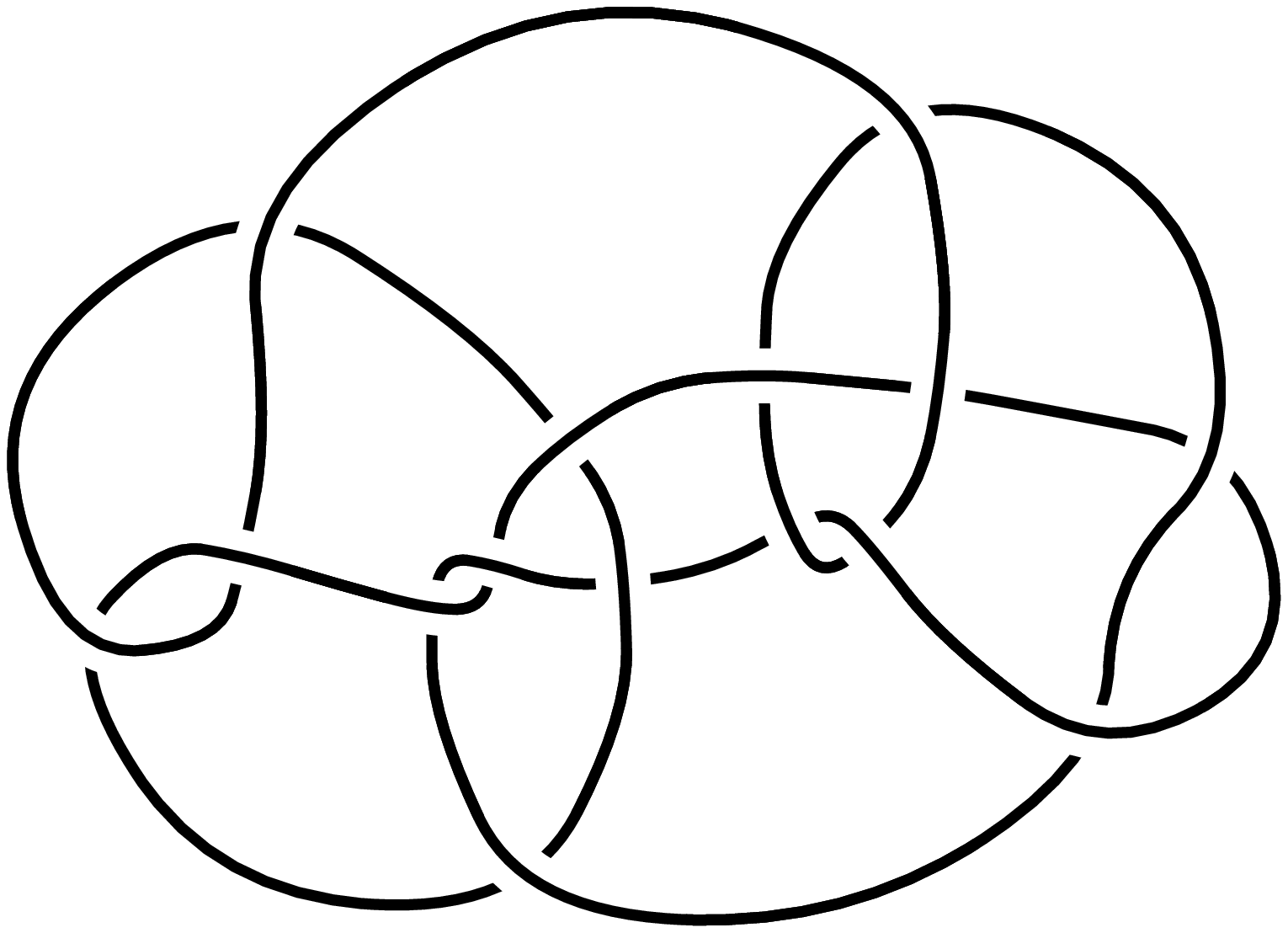} & 15_{96757} \\
\includegraphics*[width=40mm]{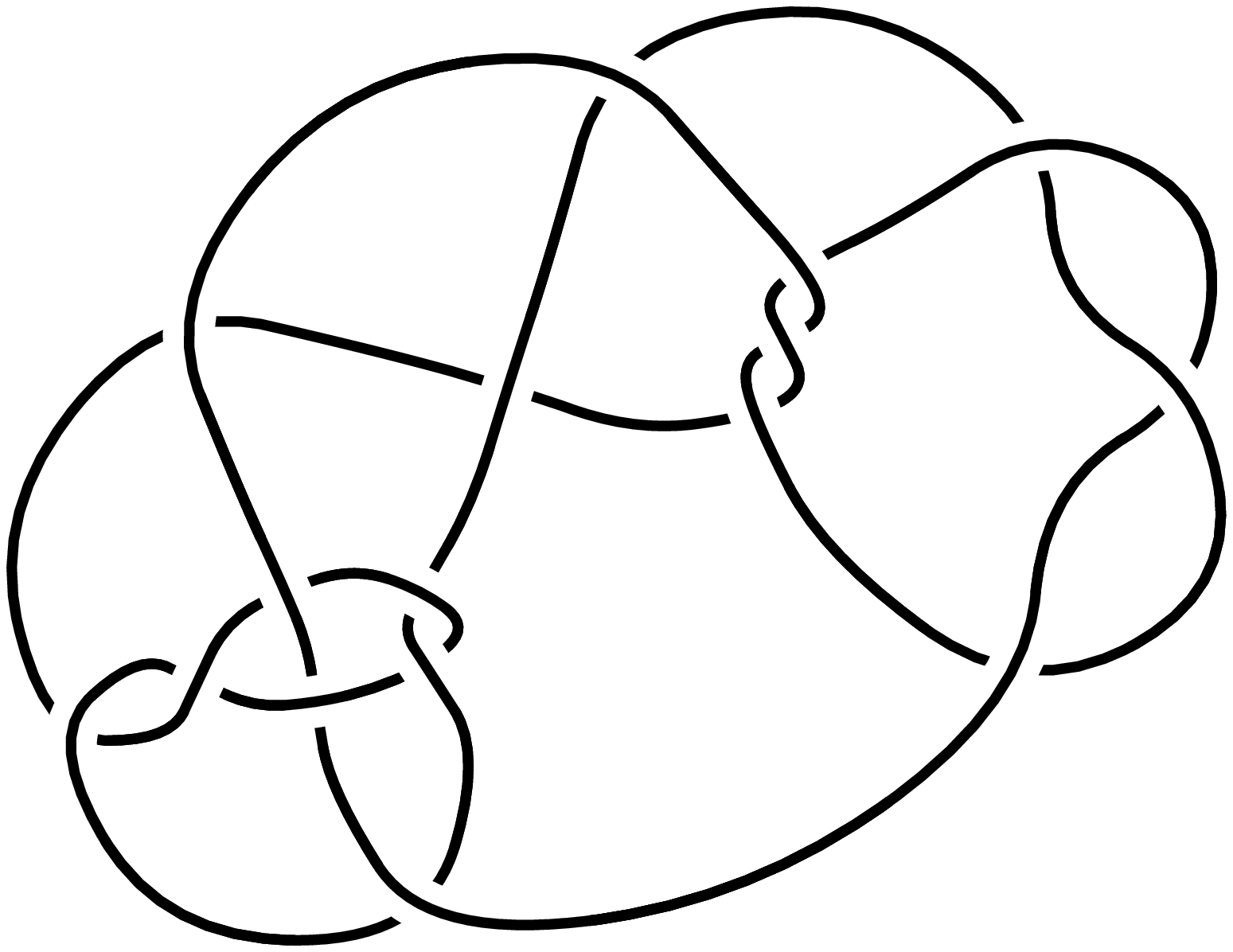} &
\includegraphics*[width=40mm]{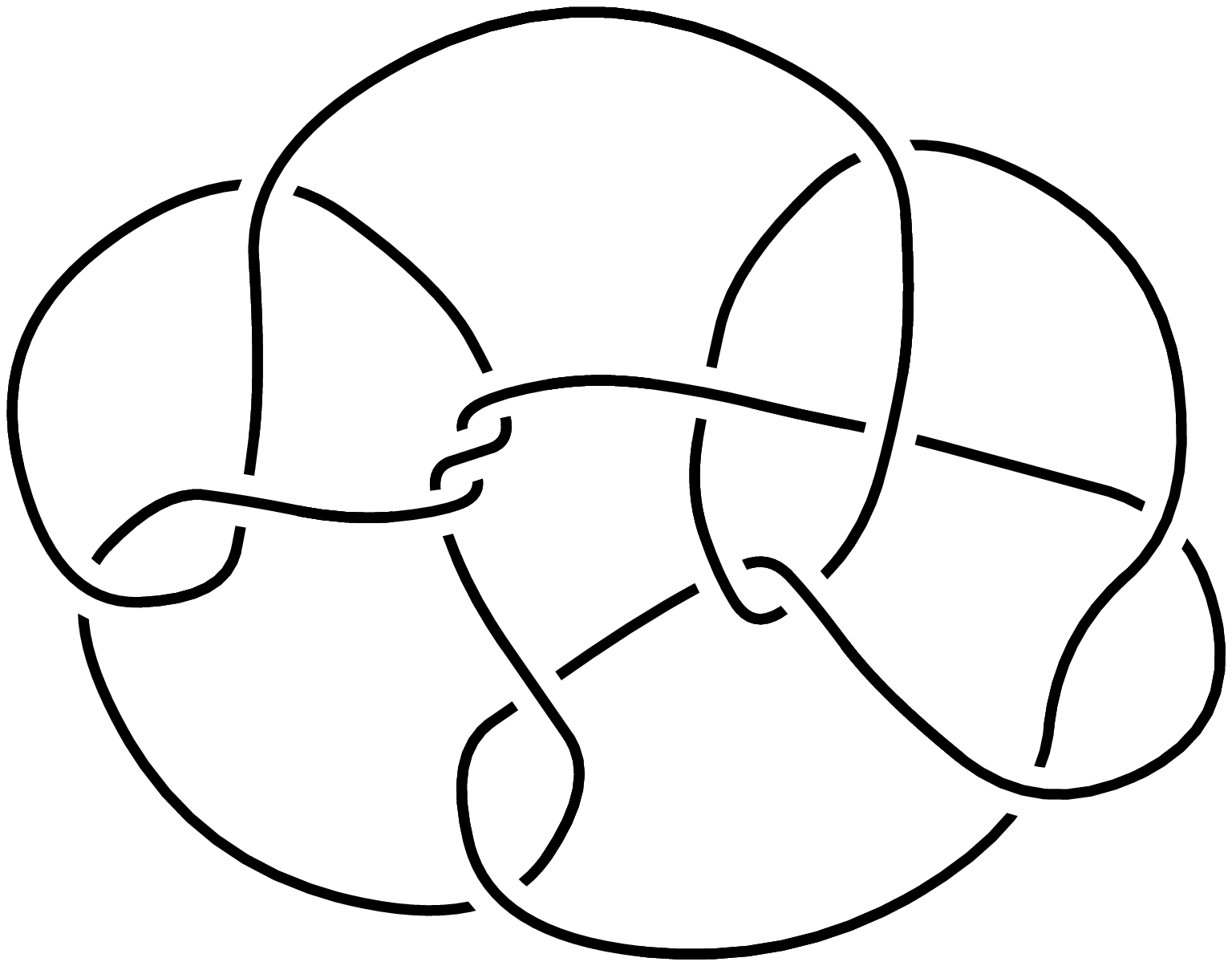} & 15_{125012} \\
\end{array}
\]
\caption{\label{Fig1}Knots for Theorem \ref{thm:loose-vs-good}. The first
column gives a (loosely) successively 2-almost positive diagram (where
the negative crossings appear on an underarc).
The second column gives their diagrams as included in the table
\cite{KnotScape}. 
}
\end{figure}

The proof of Theorem \ref{thm:loose-vs-good} is based on the generator-twisting
method of \cite{st-genus-two}, and we will frequently use the terminologies
in  Definition \ref{def:crossing-equivalence}.

\begin{definition}
\label{def:generator}
An alternating diagram $D$ (or its knot) is a \em{generator}\index{diagram!
generator} if each
of its $\sim$-equivalence classes has no more than two crossings.
\end{definition}

The terminology `generator' is justified by the following theorem, which
allows us to effectively enumerate all the diagrams of a given genus.

\begin{theorem}\cite{Brittenham,st-genus-one}
\label{thm:generator}
For a given $g>0$, there exist only finitely many generators of genus
$g$. 
All diagrams $D$ of genus $g$ are obtained from one of a generator by
$\bt$-twists, flypes, and crossing changes.
\end{theorem}

For later use we remark the following simple observations.

\begin{lemma}
\label{lem:Seifert-equiv}
Let $D$ be a non-split diagram of Euler characteristic $\chi(D)$.
Then 
\[
\sum_{\scbox{$S$:Seifert equivalence class}}
(\#S-1)\le 1-\chi(D).\]
Here $\#S$ denotes the number of crossings that belong to the the
Seifert equivalence class $S$.
\end{lemma}

\begin{proof}
Since $D$ is non-split, the number of Seifert equivalence classes
is at least $s(D)-1$. Thus
\begin{align*}
\chi(D) &= s(D)-c(D) = s(D) -\sum_{\scbox{$S$:Seifert equivalence class}}
\#S\\
&\leq 1 -\sum_{\scbox{$S$:Seifert equivalence class}}(\#S-1).
\end{align*} 
\end{proof}

\begin{lemma}
\label{lem:k-bound}
If a link $K$ admits a reduced successively $k$-almost positive positive
diagram $D$, then 
\begin{equation}\label{eqn:concordance-finite'}
k \leq 1-\chi(D)\,.
\end{equation}
\end{lemma}

\begin{proof}
It is sufficient to consider the case $D$ is a non-split diagram.
In a  successively almost positive diagram $D$, any two
successive negative crossings connect a separating
Seifert circle from opposite sides. Decompose $D=D_1*\dots*D_k$ as a
diagram Murasugi sum along these $k-1$ separating Seifert circles.

Let us write $D$ as $D_1 \ast \cdots \ast D_k$, a diagram Murasugi sum
of (good or loosely) almost positive diagrams $D_1,\ldots,D_k$. Since
$D$ is reduced, so are $D_i$. 
Thus $\chi(D_i) \leq 0$, so we conclude $\chi(D) = \sum_{i=1}^{k} \chi(D_i)
- (k-1) \leq 1-k$.
\end{proof}

\begin{proof}[Proof of Theorem \ref{thm:loose-vs-good}]

By Lemma \ref{lem:k-bound}, if any of the knots
given in Figure \ref{Fig1} admits a reduced good
successively $k$-almost positive diagram, then $k \leq 4$.

It is easy to see they are not positive; a positive genus 2 knot $K$ must
satisfy $\Md Q(K)\ge c(K)-2$ for the $Q$-polynomial $Q(K)$ (see \cite{st-generator}),
and a direct computation shows that none of three knots in Figure \ref{Fig1}
satisfies this inequality. 

One can also confirm that $15_{132907},15_{96757}, 15_{125012}$ are not
almost positive of type I (i.e., good successively $1$-almost positive),
following the method in \cite{st-minimum-genus}, or the method similar
to the following argument to rule out good successively $2$-almost positive
(though in practice, there are several differences; see Remark \ref{rem:almost-positive-case}
for details).

Here we give details on how to check $15_{132907},15_{96757}, 15_{125012}$
are not good successively 2-almost positive. A proof that they are not
good successively $k$-almost positive for $k=3,4$ is similar.

We refer to the definition of the Alexander polynomial $\Dl$ in
Section \ref{sec:polynomial-invariants}, and its properties which translate
from
corresponding properties of $\nb$ under the
conversion \eqref{eqn:Conway-Alexander}.

The knots $K \in \{15_{132907},15_{96757}, 15_{125012}\}$ in Figure \ref{Fig1}
satisfy 
\[ \max \deg_{t}\Delta_K(t) =2, \mbox{ and } \Mcf\Delta_K
 \le 9.\]
We show that they are not good successively 2-almost positive by enumerating
all the good successively 2-almost positive diagrams of knots $K$ such
that $g(K)=\Md\Dl=2$ and that $\Mcf\Dl\le 9$. 
We checked that none of these diagrams represents any of the knots $15_{132907},15_{96757},
15_{125012}$.

We determine the generating set $\mathcal{D}$ of successively 2-almost
positive knot diagrams of genus $2$. That is, we determine a set $\mathcal{D}$
of diagrams having the following properties.
\begin{itemize}
\item Every good successively 2-almost positive knot $K$ of genus two
admits a good successively 2-almost positive $D$ such that $D$ is obtained
from one of the generators $D' \in \mathcal{D}$ by applying $\bt$-twists
\eqref{eqn:bt-twists}.
\item $D'$ is a successively 2-almost positive diagram.
\end{itemize}

Here and in the following, by $\bt$-twists we always mean \emph{positive}
$\bt$-twists.

Unlike properties like alternating, positive, or almost positive, it should
be noted that the property that a diagram is successively 2-almost positive
is not flype invariant. Hence flypes must be paid attention to. See Remark
\ref{rem:almost-positive-case} for details on the treatment of flypes
for good successively 1-almost positive.

The generating set $\mathcal{D}$ is determined in the following manner.
\begin{itemize}
\item[(i)] Take a genus 2 generator diagram $D$.
\item[(ii)] List all the diagrams obtained from $D$ by a flype. Then positivize
the diagrams and switch two crossings to make $D$ successively $2$-almost
positive.

\item[(iii)] Reduce modulo type B flypes,
\item[(iv)] Discard all diagrams with $\sim$-equivalent crossings
of opposite sign.
\item[(v)] Apply (i)-(iv) for all genus 2 generator diagrams.
\item[(vi)] Sort out duplicates.
\end{itemize}

Here the step (iii) means that we choose
a (canonical, according to some deterministic criterion)
minimal representative among a type B flype equivalence class of diagrams.
(For type B flypes see, e.g., \cite[Fig.\ 5]{st-generator}.) Since type
B flypes commute with $\bt$-twists, for our purpose, one can discard the
diagrams related by type B flypes.

At the step (iv), among the successively $2$-almost positive diagrams
obtained so far, we discard all diagrams with $\sim$-equivalent crossings
of opposite sign. Even after $\bt$-twists at the
positive crossing(s), one can cancel $\sim$-equivalent positive and negative
crossings, so that the diagram $D$ becomes positive or almost positive.

Having determined the generating set $\mathcal{D}$, we proceed to enumerate
all the good successively $2$-almost positive knots having the properties
$g(K)=\max \deg_{t}\Delta_K(t)=2$ and that $\Mcf\Dl\le 9$.

To get a \emph{good} successively 2-almost positive diagram, we need a
twist at every positive crossing Seifert equivalent to a negative crossing.
By Lemma \ref{lem:Seifert-equiv}, the number of such positive crossings
is at most $1-\chi(D)=2g(D)=4$.
Hence one can make a generator diagram $D \in \mathcal{D}$ into a good
successively $2$-almost positive diagram by at most $4$ $\bt$-twists.
In particular, after at most $4$ $\bt$-twists, we may always assume that
$g(D)= \max \deg_t \Delta_K(t)=2$.

It follows from the properties $\max \deg_{t}\Delta_K(t)=g(D)$ and $\Mcf
\Delta_K(D)>0$ for every successively 2-almost positive link diagram $D$
that $\Mcf \Delta_K(D)$ will always strictly increase under $\bt$-twists.
Since $\Mcf\Dl(D)\le 9$, the number of $\bt$-twists that can be applied
is
bounded\footnote{It turns out that the maximal number of $\bt$-twists
that can be applied is indeed 12, which leads to up to 30 crossing (genus
2) diagrams.} by 
\begin{equation}\label{eqn:bound-bt-twist}
(1-\chi(D))+\Mcf\Dl-1\,=\,4+9-1 = 12\,.
\end{equation}

Thus by taking all the good successively positive diagrams obtained from
a generator in the generating set $\mathcal{D}$ by at most $12$ $\bt$-twists,
we determine all the good successively almost positive knot diagrams of
genus $2$ having the property $\Mcf\Dl\le 9$.

By computing the Jones polynomials $V_K(t)$, we checked that these
polynomials do not match with any of $15_{132907},15_{96757}, 15_{125012}$
except 9
diagrams (of 14 and 15 crossings). These remaining 9 diagrams
locate to some 14 crossing knot, so we are done.
  
\end{proof}

With the necessary tools in place,
the work can be carried out on a laptop within 45 minutes.

\begin{remark}
\label{rem:almost-positive-case}

The confirmation that $15_{132907},15_{96757}, 15_{125012}$ are not good
successively $1$-almost positive is done by a similar method, but since
the property that almost positive is invariant under flype, armed with
the following observations we can reduce the size of the generating set
$\mathcal{D}$ which makes computation faster. 

Let $D'$ be a diagram obtained from $D$ by flype, and assume that $\widetilde{D'}$
is obtained from $D$ by $\bt$-twists. Let $\widetilde{D}$ be a diagram
obtained from $D$ by applying $\bt$-twists at the corresponding crossings.
Then $\widetilde{D'}$ and $\widetilde{D}$ are related each other by mutation
of diagrams
(see, e.g., \cite{st-mutants}). 

Thus, one can postpone taking into account the effect of the flypes until
the last step, to modify the argument as follows.

\begin{itemize}
\item At the step (ii) of a construction of generating set, we just make
a generator $D$ into almost positive diagram, by suitably changing the
crossings (we do not apply flypes).

\item At the final step, for a knot $K$ represented by a candidate diagram
$D$, we check whether the Jones polynomial $V_K(t)$ of candidates are
equal or not. 
\item Since the Jones polynomial is mutation invariant, if the Jones polynomials
are different, any mutant of $D$ does not represent $15_{132907},15_{96757},
15_{125012}$. When the Jones polynomial coincides (it seems to be rare),
then we consider all the mutants of diagram $D$ and check all of them
are not equal to $15_{132907},15_{96757}, 15_{125012}$ by other methods.
\end{itemize}

In the case of good (type I) almost positive diagrams, fortunately, every
diagram can be distinguished from $15_{132907},15_{96757}, 15_{125012}$
by the Jones polynomial, so we do not need to consider mutants.
\end{remark}

\subsection{Good successively almost positive vs.\ almost positive\label{sec:good-vs-ap}}

\begin{theorem}
\label{thm:good-vs-ap}
There exist knots which are good successively 2-almost positive but are
not almost positive. More precisely, the knots $17_{*1614792}, 17_{*908691},17_{*549551}$
of Figure \ref{Fig2} are good successively 2-almost positive knots of
genus 3 which are not almost positive.
All three knots are checked to be 17 crossing knots.
\end{theorem}

\begin{figure}[htb]
\[
\begin{array}{cc}
\includegraphics*[width=45mm]{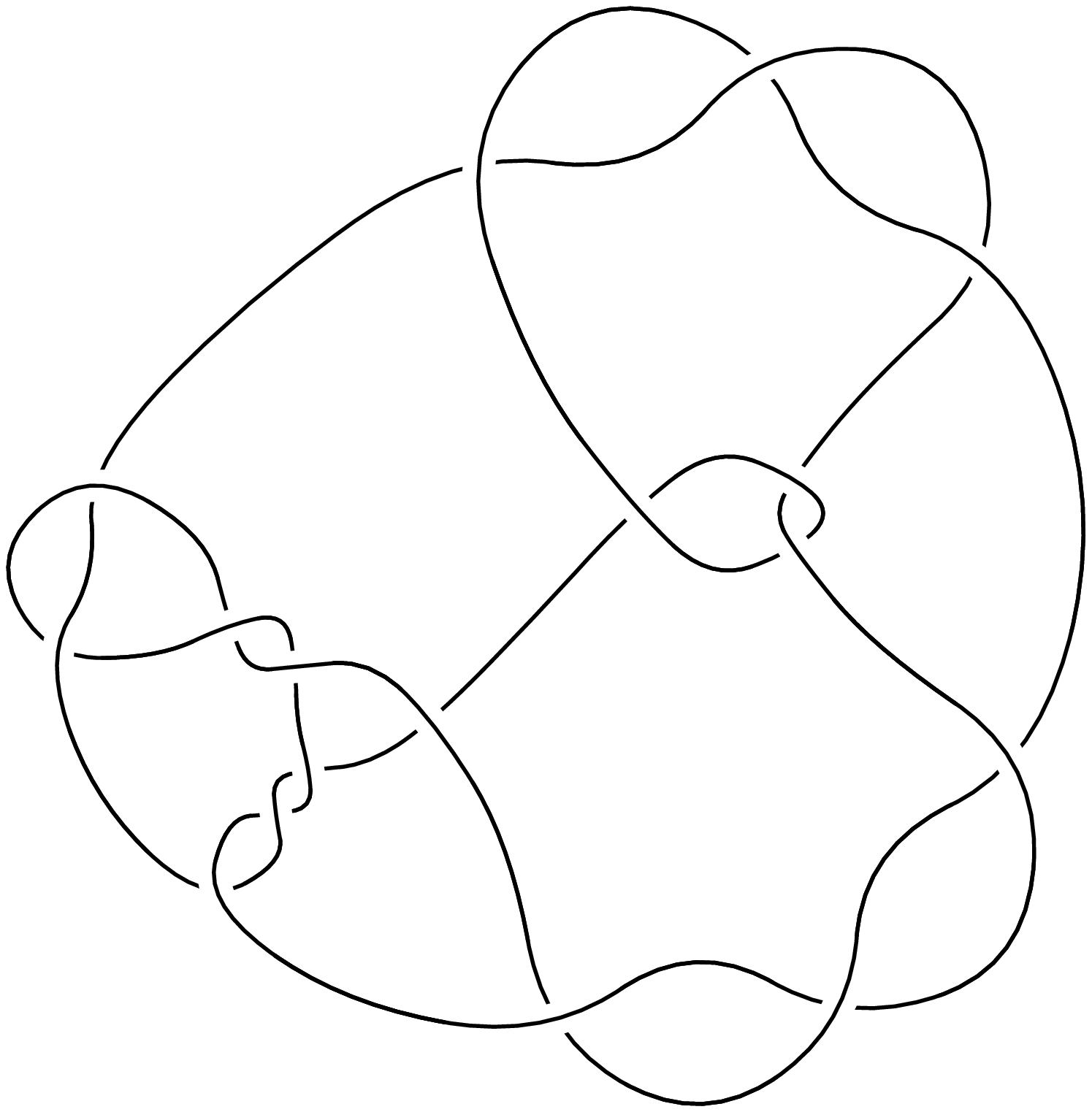} & 
\includegraphics*[width=45mm]{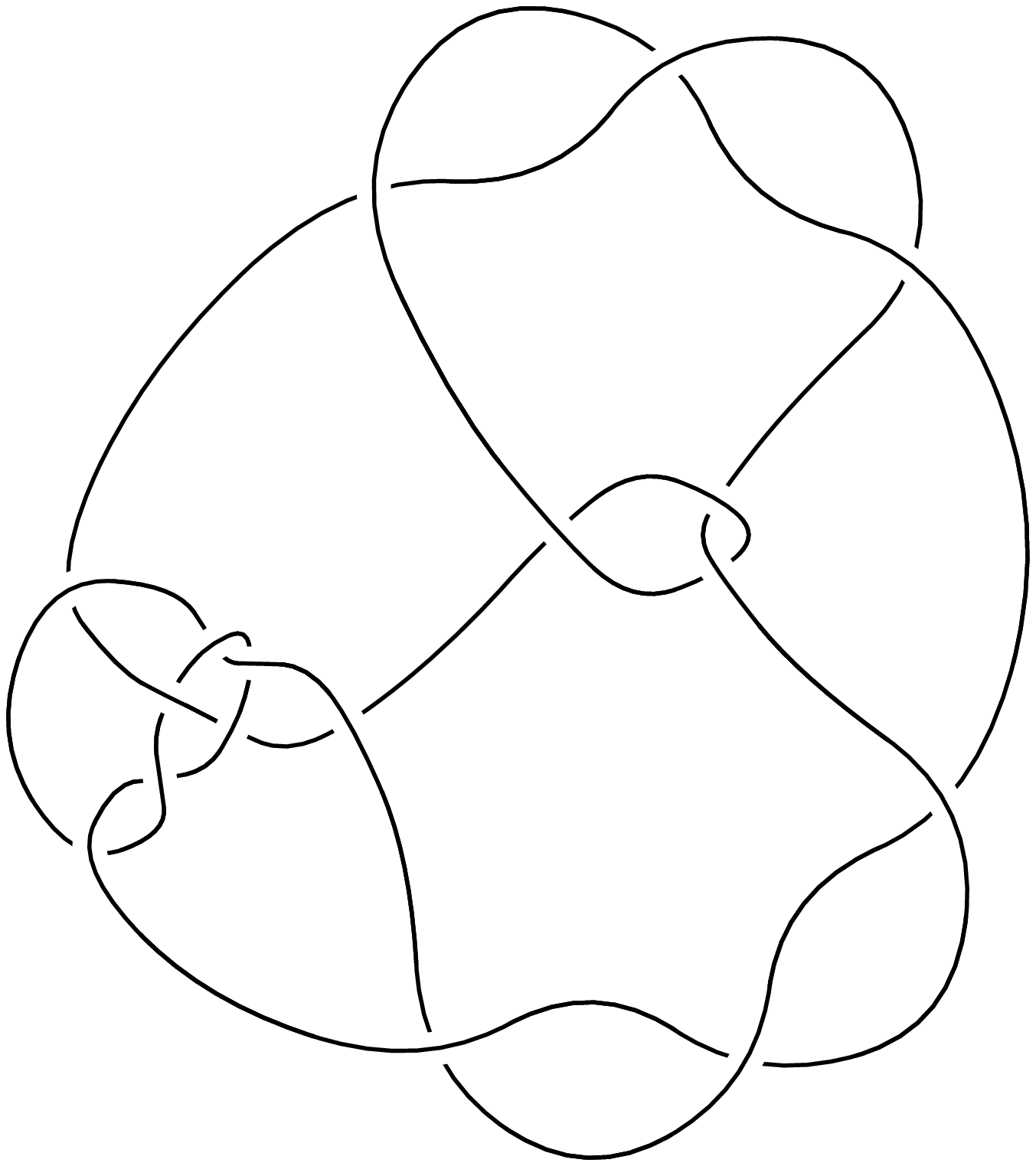} \\
17_{*1614792} & 17_{*908691} \\[8mm]
\includegraphics*[width=45mm]{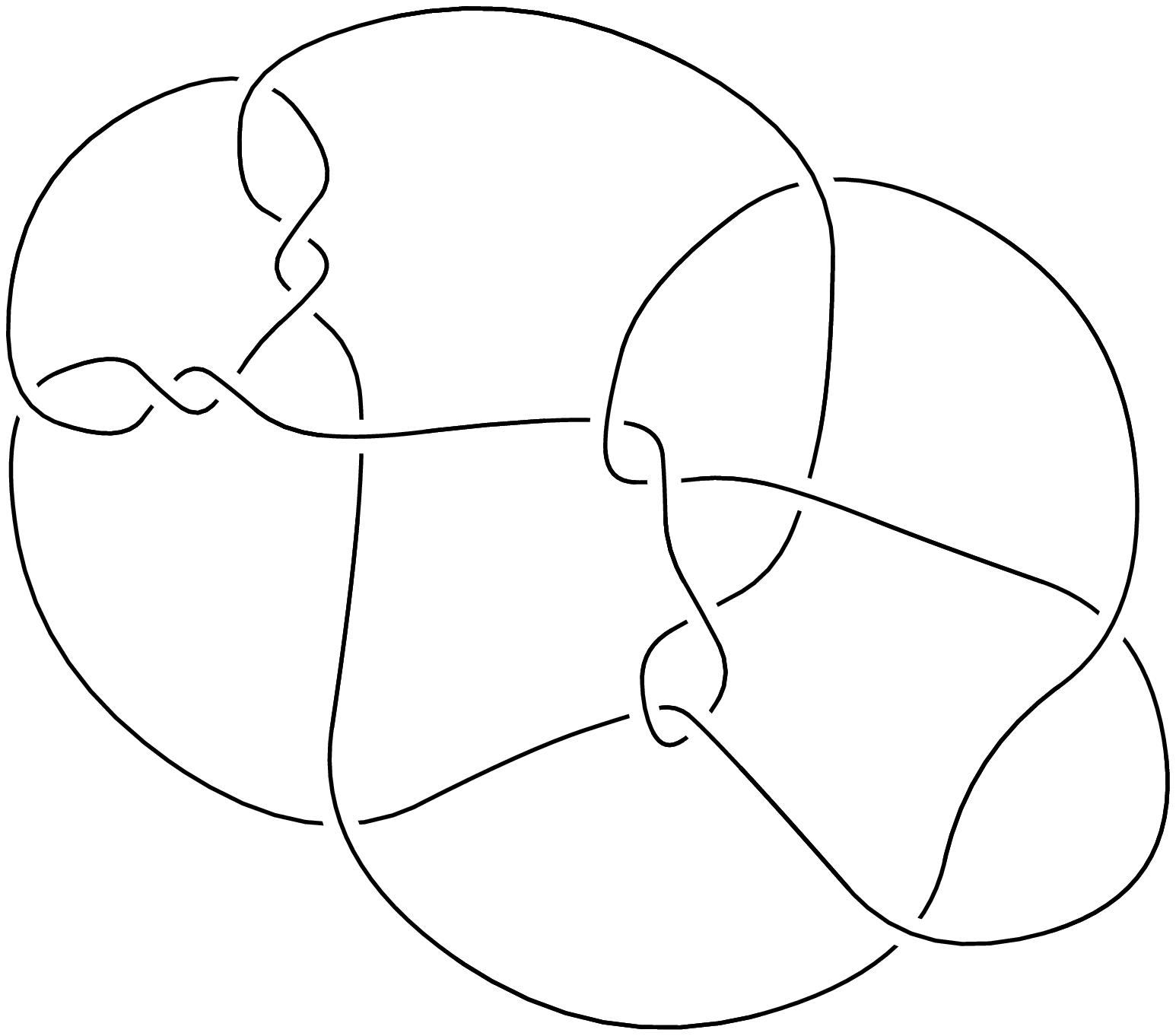} &

\includegraphics*[width=45mm]{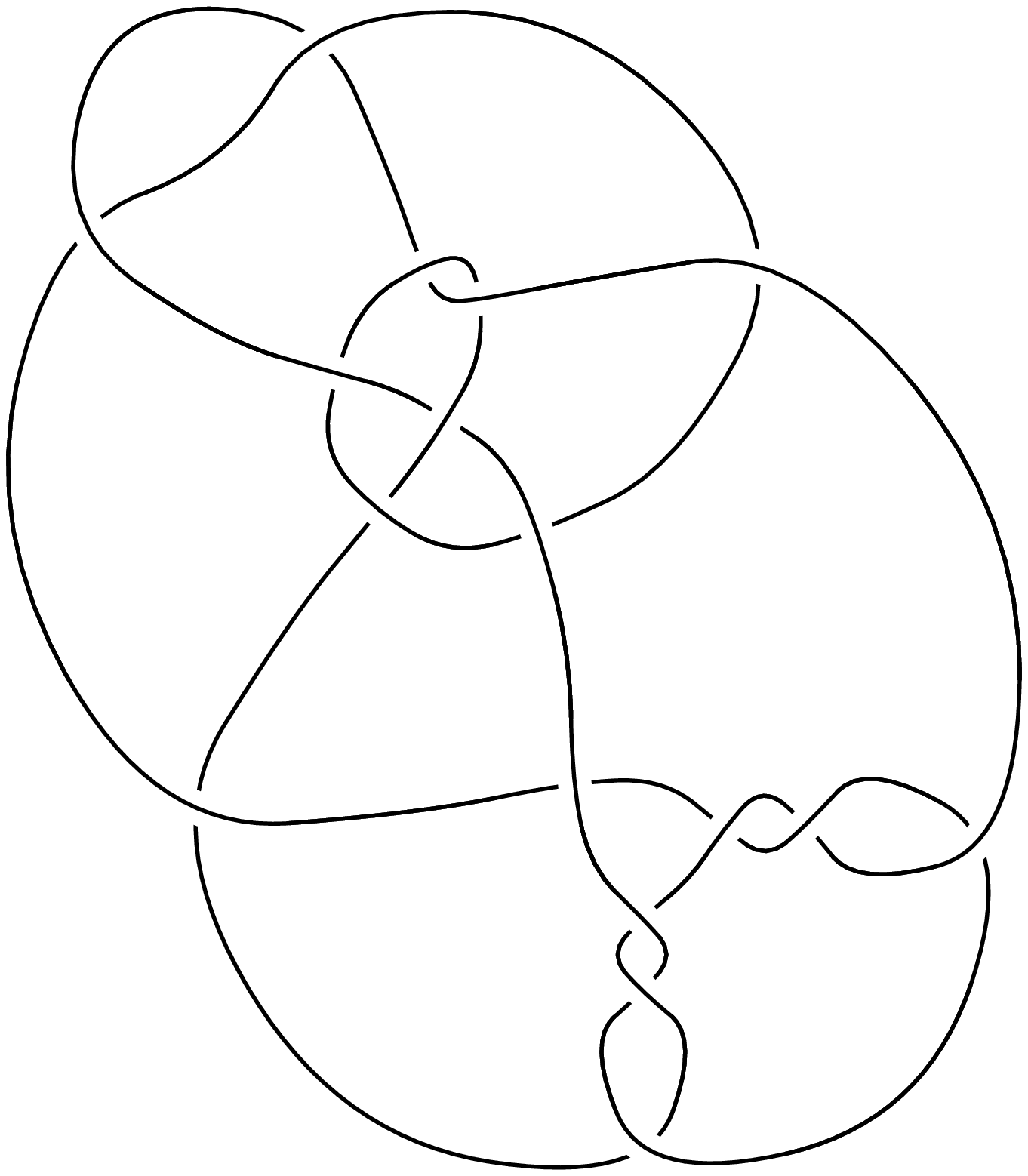} \\
\multicolumn{2}{c}{17_{*549551}}
\end{array}
\]
\caption{\label{Fig2}Knots for Theorem \ref{thm:good-vs-ap}. The first
row diagrams and the left second row diagram are all good successively
2-almost positive minimal diagrams (where the negative crossings appear
on an underarc). The third knot has some loosely successively 2-almost
positive minimal diagrams also, e.g.,
the one shown in the second row right.}
\end{figure}

\begin{proof}

A main stream of the proof is similar to Theorem \ref{thm:loose-vs-good}
or \cite{st-minimum-genus}, which works for almost positive of type II
(i.e., even if $\Md\Dl=g(D)-1=g(K)$). 
However, the three knots are fibered of genus 3, so $\Mcf\Delta_K=1$.
This simplifies the argument as we will describe below.

First of all, checking that they are not positive can be done by
Cromwell's $c(K)\le 4g(K)$ test \cite{Cromwell} as in Section \ref{sec:p-vs-ap}.

We show they are not almost positive.
Let us take an almost positive generator diagrams of
genus $3$ or $4$. These generators were compiled as in \cite{st-genus-two,st-generator}.\\

\begin{caselist}
\case $g(D)=4$ (i.e., $D$ is type II, or loosely successively 1-almost
positive).\label{C1}\\

As in the proof of Theorem \ref{thm:loose-vs-good}, $\bt$-twists will
always augment
$\Mcf\Delta_K$ even if $\Md\Delta_K=3$, or will realize
$\Md\Delta_K=4$. Thus $\Md\Delta_K=3$ and $\Mcf\Delta_K=1$ imply
that we do not need to use $\bt$-twists.
Thus only almost positive generator diagrams of genus 4
need to be tested.

To further reduce work, we use that all three knots $K$ have signature
$\sigma(K)=4$, which means that positivized (genus 4) generators $D$
must have $\sigma(D) \leq 6$.
It (drastically) reduces the number of
generators to consider. 

Again one can pre-select diagrams by checking the coincidences of the
Jones polynomial $V_K(t)$. It turns out that the Jones polynomial already
distinguishes all these diagrams from any of the three knots. 
Since $\bt$-twists are not applied, unlike Remark \ref{rem:almost-positive-case}
no mutations need to be taken care of.

The test can be performed in a few minutes.\\

\case $g(D)=3$ (i.e., $D$ is good, or type I).\\

Here twisting must be tested, and it was performed similarly
to the proof of Theorem \ref{thm:loose-vs-good}, but again
keeping eye on the signature (as in Case \ref{C1}) and additionally using
tests of the maximal value of the coefficients of $\nb$
of the three knots (which is attained by $17_{*549551}$ for both
$z$-degree $2$ and $4$).

As we argued above, a $\bt$-twist must be performed on each crossing 
Seifert equivalent to the negative one (otherwise $\Md\Dl<3$)
but at no other crossing (otherwise if $\Md\Dl=3$, then $\Mcf\Dl>1$).
By Lemma \ref{lem:Seifert-equiv}, in particular 
at most $6$ twists need to be applied. But one can easily see
that $6$ twists are needed only for (the series of) $7_1$ and $5$ twists
only
for some 8-crossing generators.

There is a number of redundancies still left, but 
avoiding further (very technical) simplifications, the test
was manageable in about 50 minutes on a laptop.
\end{caselist}
\end{proof}

\subsection{(Good) Successively almost positive vs.\ strongly quasipositive}

Strong\-ly quasipositive knots are not necessarily successively almost
positive; for a given link $L$, there exists a strongly quasipositive
link $L'$ such that $\nabla_L(z)=\nabla_{L'}(z)$ \cite[88 Corollary]{Rudolph-survey}.
See \cite[Section 5]{Silvero} for details and concrete examples. Thus
the Conway polynomial of strongly quasipositive link may not be non-negative.

\subsection{Enumeration of positive knots\label{sec:positive-list}}

A similar method to the one used for 
Theorem \ref{thm:loose-vs-good} or Theorem \ref{thm:good-vs-ap} can also
overcome the problems of \cite{Abe-Tagami} which asks whether some almost positive
knots are non-positive or not, although we need various additional simplifications
to make computation faster and reasonable.
In particular, a similar method allows us to complete the list of low-crossing
positive (prime) knots.

For positive diagrams, since the value $\Mcf\Dl$ is generally higher,
the bound on the number of $\bt$-twists in \eqref{eqn:bound-bt-twist}
is large, so the number of candidate diagrams will be much larger.

To reduce the computation, we exploit a useful fact that a positive diagram
$D$ is A-adequate ($+$-adequate), 

By \cite[Proposition 3.4]{st-adequate} (which is the reinterpretation
of Thistlethwaite's work \cite{Thistlethwaite} on A-adequacy ($+$-adequacy)),
for a positive diagram $D$ of $K$
\[ \max \deg_{z}F_K(a,z) \geq c(D)-1+\chi(D) \]
holds. Since one $\bt$-twist increases $c(D)$ by two, if $K$ is obtained
from a generator $D_0$ by $\bt$-twists $k$ times,
\[ 1-\chi(D_0)+\max \deg_{z}F_K(a,z) -c(D_0) \geq 2k \]
holds.
This bound is often far better than the bound $k\le \Mcf \Delta_K-1$,
which is the analogue of \eqref{eqn:bound-bt-twist} for a positive diagram.
(In \eqref{eqn:bound-bt-twist}, the first term  $1-\chi(D)$ on the
left hand side has appeared to make the diagram good. Thus for a positive
diagram we do not need this term.)

It took a few hours on a desktop for up to 13 crossings
and about 3 days for 14-15 crossings, although it will very likely
not be able to (easily) completely settle the list for 16 crossings.
The result (232 prime non-alternating
knots up to 13, and 3355 knots of 14-15 crossings) is available on
\cite{st-knot-table}. Some more detailed account on the computation
will take up extra space and may be given elsewhere. Here we just state
the following answer to \cite[Question 6.9]{Abe-Tagami}.

\begin{theorem}
\label{thm:AT-answer}
Up to $12$ crossings, there are $13$ almost positive, but
not positive prime knots:
\[
\begin{array}{c}
10_{145},\  
12_{1436},\ 
12_{1437},\ 
12_{1564},\ 
12_{1617},\ 
12_{1620}, 
12_{1654},\\
12_{1690},\  
12_{1692},\ 
12_{1720},\ 
12_{1816},\ 
12_{1930},\ 
12_{1948}\,.
\end{array}
\]
\end{theorem}

\begin{proof} 
These knots are shown to be almost positive in \cite{Abe-Tagami}. That they are
non-positive can be checked by the generator-twisting method as we discussed.

Corollary \ref{cor:HOMFLY-ap} and the methods of \ref{sec:Knot-table}
give alternative proofs of the non-almost positivity of the remaining
knots,
except for $12_{1811}$ and $12_{2037}$. 
The first one has zero signature so it cannot be almost positive (indeed,
by Theorem \ref{thm:signature>0} it cannot be w.s.a.p.).

The second one has $g_3(K)\neq g_4(K)$ so it is not Bennequin sharp. Since
an almost positive knot is Bennequin sharp, $K$ cannot be almost positive.

\end{proof}

It must be noted also that, while semiadequacy input is not
useful for non-positive diagrams in general, there are tools available
to use polynomial degrees to reduce twisting in non-positive diagrams
as well. Such ``regularization'' tests were extensively elaborated on
in \cite{st-hyperbolicity-canonical-genus2}. We avoided these technicalities
for Theorem \ref{thm:loose-vs-good}, since the number of diagrams left
to check (about 34,000)
was fairly manageable anyway.

\subsection{Weakly successively almost positive vs.\ weakly positive}

By means of Example \ref{exam:2-ap-is-wp}, take many 2-almost positive
knots,
like the figure-8-knot. Many of the properties of 
weakly successively almost positive knots we proved (like positivity
of $\nb$) are not satisfied.

\section{Questions and problems\label{sec:question}}

Throughout the paper we raised various questions or conjectures. In this
section we gather a set of problems arising from various discussions in the
paper.

Though to save the space we mainly concentrate on weakly successively almost
positive links, a similar question is always meaningful if we restrict it to
some more specific class like (good/loosely) successively almost positive links.

The most important open problem for weakly successively almost positive
links is their strong quasipositivity.

\begin{question}
\label{ques:SQP}
{$ $}
\begin{itemize}
\item[(a)] Is every weakly successively almost positive link strongly quasipositive
?
\item[(b)] Is every weakly successively almost positive link Bennequin-sharp
?
\item[(c)] Is every weakly successively almost positive link quasipositive ?
\end{itemize}
\end{question}

 As outlined, we will address these questions, to some extent, in \cite{Part2}.

We have seen that various classes of positivity notions are indeed different,
by showing the canonical inclusion is strict.
Nevertheless, there are still several inclusions whose strictness is not
confirmed yet. In light of Theorem \ref{thm:loose-vs-good} or Theorem
\ref{thm:good-vs-ap}, it is natural to expect the affirmative answer to
the following. 

\begin{question}
\label{ques:strictness}
Is there a weakly successively almost positive link which is not successively almost positive ?
\end{question}

See Remark \ref{rem:inclusion-arbitrary-k} for the corresponding problems
where the number $k$ of negative crossings is fixed.
In a similar direction, it is also natural to ask the following.

\begin{question}\label{ques:inclusion-k}
For every $k>0$, is there a weakly successively $k$-almost positive link
which is not weakly successively $k'$-almost positive of all $k'<k$ ?
\end{question}

The same question applies for successively almost positive. 
It should be noted that the method in Theorem \ref{thm:loose-vs-good}
or
Theorem \ref{thm:good-vs-ap}, in principle, works for higher \emph{fixed}
$k$.
But it is already difficult to practice for $k=3$, and certainly leaves
unclear at present how to simultaneously construct examples for arbitrary
$k$ (or,
in fact, infinitely many $k$ for that matter).

We have seen various properties of invariants of a w.s.a.p.\ link
that are
generalizations of properties of a positive link (under the additional
assumption that it is Bennequin-sharp, if needed).

Nevertheless, there are several properties which remain open.
\begin{question}
Let $L$ be a weakly successively almost positive link.
\begin{itemize}
\item If $L$ is fibered, is the unique
non-zero coefficient $c_{k,1-\chi(L)}$ always $c_{1-\chi(L),1-\chi(L)}=1$
? (Conjecture \ref{conj:fibered-HOFMLY-top-coefficients})
\item Does the maximum $z$-degree term of the HOMFLY polynomial of w.s.a.p.\ %
link have the gap-free property ? (see Remark \ref{rem:gap-free})
\item Does the equality (\ref{eqn:Yokota}) hold for almost positive links
of type I ? More generally, does equality (\ref{eqn:Yokota}) hold for
(good) successively almost positive links? 
\end{itemize}
\end{question}

In Theorem \ref{cor:finite-concordance} we showed the finiteness of
(weakly) successively almost positive knots in an algebraic concordance
class, for  weakly successively $k$-almost positive knots for \emph{fixed}
$k$. It is natural to ask whether fixing $k$ is necessary or not.

\begin{question}\label{ques:concordance}
Does every algebraic (or, topological, smooth) concordance class contain
at most finitely many weakly successively almost positive knots?
\end{question}

In a related direction, it is of independent interest to explore an optimal
signature estimate from (w.s.a.p.) diagram.

\begin{question}
Let $D$ be a reduced diagram of a link $L$. Find optimal coefficients
$C_1,C_2$ so that
\[ \sigma(L) \geq C_1(1-\chi(L)) - C_2 c_-(D) \]
Similarly, for a reduced w.s.a.p.\ diagram $D$ of $L$, find optimal coefficients
$C'_1,C'_2$ so that
\[ \sigma(L) \geq C'_1(1-\chi(L)) - C'_2 c_-(D) \]
\end{question}

In \cite{Ozawa} Ozawa showed the visibility of primeness for positive
diagrams: a link represented by a positive diagram $D$ is non-prime if
and only if the diagram $D$ is non-prime. As a consequence, if $K_1\#
K_2$ is positive, then both $K_1$ and $K_2$ are positive.

A straightforward generalization for (good) successively almost positive
diagrams does not hold. 

\begin{example}
\label{example:non-prime}
The knot diagram on Figure \ref{3.} is a good successively 2-almost positive
diagram of $3_1\# 10_{145}$ which is not prime ($3_1$ is positive and
$10_{145}$ is almost positive).
\end{example}

\begin{figure}[htb]
\[
\begin{array}{c}
\includegraphics*[width=45mm]{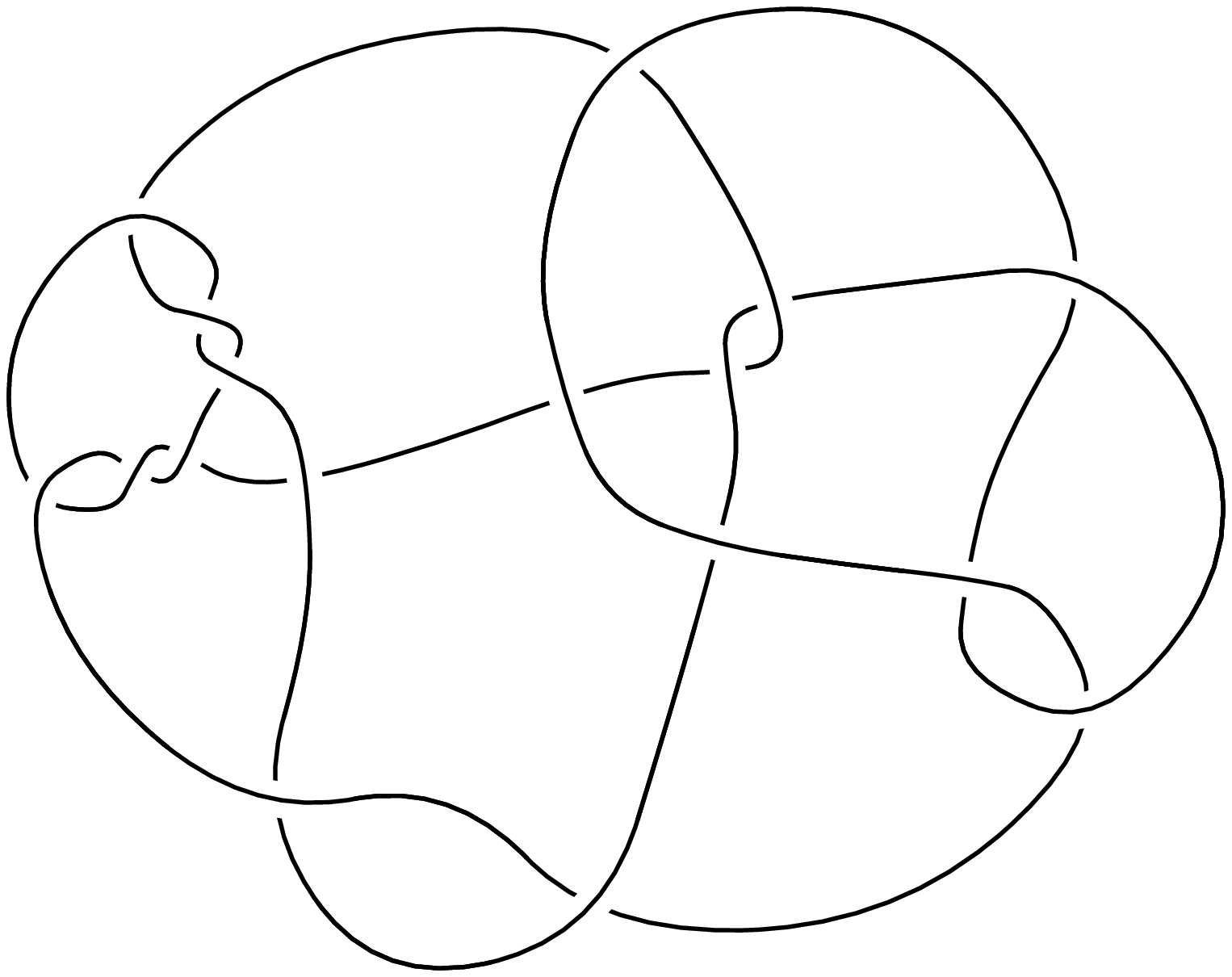} \\
3_1\# 10_{145}
\end{array}
\]
\caption{\label{Fig3}A prime good successively $2$-almost positive diagram
of a
composite knot.\label{3.}}
\end{figure}

Obviously one can construct such examples for good successively $1$-almost
positive as well.

The question and observation suggest that when we discuss (good) successively
$k$-almost positive diagrams, it is better and more natural to take the
minimum $k$.

\begin{definition}
We say that a (weakly) successively almost $k$-positive diagram $D$ of
a link $L$ is \emph{tight} if $k$ (the number of negative crossings) is
minimum among all the (weakly) successively almost positive diagram of
$L$.
\end{definition}

Let us define \[ p_{good}(K)= \min \{k \: | \: K \mbox{ is good successively
}k \mbox{-almost positive} \} \] if $K$ is good successively almost positive,
and $p(K)=-\infty$ otherwise. Similarly, we define \[ p_{succ}(K)= \min
\{k \: | \: K \mbox{ is (weakly) successively }k\mbox{-almost positive}
\}\] if $K$ is (weakly) successively almost positive, and set $p_{succ}(K)=-\infty$
otherwise.

The following way of extending Ozawa's result, modified to evade the
previous counterexample, could still be true.

\begin{question}\label{ques:prime}
Is a knot $K$ prime if $K$ is represented by a prime tight w.s.a.p.\ diagram
$D$ ?
\end{question}

Although Question \ref{ques:prime} seems not previously
encountered even for $k=1$, at least thus far it appears 
worth asking: If a prime (good) almost positive diagram 
depicts a composite knot, is the knot positive (i.e.,
the connected sum of positive factors)?

The technicality of determining $p_{succ}(K)$ or $p_{good}(K)$ also raises
the question about a property that can more naturally distinguish
(almost) positive links from (good) successively almost positive ones.

Investigating these quantities will be of independent interest. For instance,
Lemma \ref{lem:k-bound} implies that $p_{good}(K) \leq 1-\chi(K)$. Finding
other upper or lower bounds would be useful to attack Question \ref{ques:inclusion-k}.
Note also that the condition \eqref{eqn:concordance-finite} is of similar
nature,
and its weakening, together with Lemma \ref{lem:k-bound}, is a potential strategy towards
Question \ref{ques:concordance}.

We have seen the evidence for the remarks at the start of Section \ref{sec:comparisons}.
But they also underscore the difficulty of Question \ref{ques:inclusion-k}
as to understanding the filtration of successively almost positive
links with respect to $p_{succ}(K)$.

In this regard, we pose one more natural problem
related to Question \ref{ques:prime}:

\begin{question}\label{ques:prime2}
Is $K_1\# K_2$ w.s.a.p.\ (s.a.p.) if and only if both $K_1$ and $K_2$ are
so? Is $p_{succ}(K_1\# K_2)=p_{succ}(K_1)+p_{succ}(K_2)$ (similarly
for $p_{good}$)?
\end{question}

\setcounter{section}{0}

\renewcommand{\thesection}{Appendix \Alph{section}}
\section{Fibered link enhancement for Scharlemann-Thompson's theorem\label{sec:ST-theorem}}
\renewcommand{\thesection}{\Alph{section}}

In this appendix we prove Theorem \ref{thm:st-fibered}. Our proof is a
mild extension of Scharlemann-Thompson's proof using sutured manifold
theory.

\subsection{Sutured manifold}
\begin{definition}
A \emph{sutured manifold}\index{sutured manifold} $(M,\gamma)$ is a pair
of an irreducible oriented compact 3-manifold $M$ and a subset $\gamma
\subset \partial M$ consisting of pairwise disjoint annuli such that each
annulus contains an oriented core circle.
\end{definition}

The \emph{suture}\index{suture} $s(\gamma)$ is the union of the
oriented core circles of the annuli $\gamma$. 
We orient each component of $R(\gamma):= \partial M \setminus \gamma$
so that the orientation is coherent with the orientation of the sutures (i.e.,
the orientation of $\partial R$ coincides with that of the sutures). 
We denote by $R_{+}(\gamma)$ (resp.\ $R_-(\gamma)$) the union of the components
of $R(\gamma)$ such that the orientation agrees (resp.\ disagrees) with
the orientation as a subsurface of $\partial M$.

\begin{definition}
A sutured manifold $(M,\gamma)$ is \emph{taut}\index{sutured manifold!
taut} if $R(\gamma)$ is a Thurston norm-minimizing in $H_2(M,\gamma)$.
\end{definition}

\begin{example}[Trivial sutured manifold, product sutured manifold]
A sutured manifold $(M,\gamma)$ is a \emph{product sutured manifold}\index{sutured
manifold! product sutured manifold} if it is homeomorphic to $(S\times
[0,1], \partial S \times [0,1])$ for some compact oriented surface $S$.
When $S=D^{2}$, we say that $(M,\gamma)$ is the \emph{trivial}\index{sutured
manifold! trivial sutured manifold} sutured manifold.
\end{example}

\begin{example}[Complementary sutured manifold]
Let $R$ be a Seifert surface of (non-split) link $L$ in $M$. The \emph{complementary
sutured manifold}\index{complementary sutured manifold} $(M_R,\gamma_R)$
of $R$ is a sutured manifold $(M \setminus (-1,1)\times R, \partial R
\times [-1,1])$.
\end{example}

For a sutured manifold $(M,\gamma)$ and a properly embedded surface $S$
that transversely intersects with the sutures $s(\gamma)$, let $M'$ be the
3-manifold obtained from $M$ by cutting along $S$. 
Then by connecting sutures naturally, we get a sutured manifold $(M',\gamma')$.
We refer to such a situation as $S$ \emph{defines a sutured manifold decomposition}
$(M,\gamma) \stackrel{S}{\rightsquigarrow} (M',\gamma')$.

For example, the complementary sutured manifold $(M_R,\gamma_R)$ of a
Seifert surface $R$ can be seen as a sutured manifold decomposition $(M,\emptyset)
\stackrel{R}{\rightsquigarrow} (M',\gamma')$ defined by $R$.

A key feature of a sutured manifold decomposition is the following.

\begin{lemma}\cite[Lemma 3.5]{Gabai-foliationI}
\label{lem:taut}
Let $(M,\gamma) \stackrel{S}{\rightsquigarrow} (M',\gamma')$ be a sutured
manifold decomposition. If $(M',\gamma')$ is taut, then $(M,\gamma)$ is
taut.
\end{lemma}

\begin{definition}
Let $(M,\gamma)$ be a sutured manifold. A \emph{product annulus} is a
properly embedded annulus $A$ in $M$ such that $\partial A \subset R(\gamma)$
and $\partial A \cap R_{\pm}(\gamma)\neq \emptyset$. A \emph{product disk}
is a properly embedded disk $D$ in $M$ such that $\gamma \cap D$ is two
essential arcs.
\end{definition}

A \emph{product decomposition}\index{product decomposition} is a sutured
manifold decomposition $(M,\gamma)  \stackrel{\Delta}{\rightsquigarrow}
(M',\gamma')$ defined by a product disk $\Delta$.

The following lemma explains a relation between a product disk/annulus
and product sutured manifolds.

\begin{lemma}\cite[Lemma 2.2, Lemma 2.5]{Gabai-fibered-link}
\label{lem:Gabai-product}
Let $(M,\gamma) \stackrel{S}{\rightsquigarrow} (M',\gamma')$ be a sutured
manifold decomposition defined by a product disk or a product annulus
$S$. Then $(M,\gamma)$ is a product sutured manifold if and only if $(M',\gamma')$
is a product sutured manifold. 
\end{lemma}

\subsection{Norm-reducing slopes}

Let $M$ be an irreducible 3-manifold whose boundary is a non-empty union
of tori, and let $P$ be a component of $\partial M$. For a slope (unoriented
simple closed curve) $s$ of $P$, we denote by $M(s)$ the Dehn filling
on $s$. 

Let $S$ be a properly embedded surface such that $P \cap S =\emptyset$.
In this section we review a slope $\alpha$ where $S$ (can) fail to be
norm-minimizing in $M(\alpha)$.
See \cite{Baker-Taylor} for more general results on how the Thurston norm
behaves under Dehn fillings.

\begin{definition}
An \emph{$I$-cobordism}\index{$I$-cobordism} between surfaces $S_0$ and
$S_1$ is a 3-manifold $V$ whose boundary is $S_0 \cup S_1$ such that the
induced maps $(i_j)_*:H_1(S_j) \rightarrow H_1(V)$, $(j=0,1)$ are both
injective. 

Let $P$ be a torus boundary component of $\partial M$ and $S$ be a properly
embedded surface in $M$ which is disjoint from $S$. We say that $M$ is
\emph{$S_P$-atoroidal}\index{$S_P$-atoroidal} if, whenever cutting $M$ along a torus $T$ in $M \setminus S$ (and taking a suitable component) gives rise
to an $I$-cobordism between $P$ and $T$, then $T$ is parallel to $P$.
\end{definition}

To understand norm-reducing slopes, the following existence theorem of
sutured manifold hierarchy plays a fundamental role.

\begin{theorem}\cite[Step 1 of the proof of Theorem 1.7]{Gabai-foliationII}
\label{thm:hierarchy}
Let $R$ be a norm-minimizing surface which is disjoint from $P$, and let
$(M_R,\gamma_R)$ be the complementary sutured manifold of $R$. Then there
exists a taut sutured manifold hierarchy
\begin{equation}
\label{eqn:suture1} (M_R,\gamma_R) \stackrel{S_1}{\rightsquigarrow} (M_1,\gamma_1)
\stackrel{S_2}{\rightsquigarrow}  \cdots  \stackrel{S_n}{\rightsquigarrow}
 (M_n,\gamma_n), 
\end{equation}
such that
\begin{itemize}
\item[(i)] Each separating component of $S_i$ is a product disk.
\item[(ii)] Each non-separating component of $S_i$ is either a product
disk, or a norm-minimizing surface realizing some non-trivial $y \in H_2(H,\partial
H)$ which is disjoint from $P$. Here $H$ denotes the connected component
of $M_{i-1}$ that contains $P$.
\item[(iii)] $(M_n,\gamma_n)$ is a union of trivial sutured manifolds
and an  $I$-cobordism between a torus $T$ and $P$ with non-empty sutures
on $T$. 
\end{itemize}
\end{theorem}

To simplify the notation, in the following arguments we always assume
that the surfaces $S_i$ in \eqref{eqn:suture1} are connected. We will also
neglect trivial sutured manifolds components to simply regard $(M_n,\gamma_n)$
as an $I$-cobordism between a torus $T$ and $P$.

\begin{theorem}\cite[Theorem 1.8]{Gabai-foliationII}
\label{theorem:gabai}
Assume that $M$ is irreducible and $S_P$-atoroidal. Then the surface $S$
remains to be norm-minimizing in $M(s)$, for all but one slope $s$.
\end{theorem}

To give a precise description of the (possible) exceptional slope $s$,
we review the proof.

\begin{proof}
Take a sutured manifold hierarchy \eqref{eqn:suture1} given in Theorem
\ref{thm:hierarchy}. The assumption that $M$ is $S_P$-atoroidal says that
$M_n \cong P\times [0,1]$.
Let $s$ be a slope on $P$ defined by the suture on $T=P\times\{1\}$. We
call the slope $s$ the \emph{suture slope} (with respect to the sutured
manifold hierarchy \eqref{eqn:suture1}).

Let us consider the Dehn filling of $M(s')$ of a slope $s'$ on $P$, for
$s \neq s'$. 
Then by attaching solid tori along $P$ at each stage, the sutured manifold
hierarchy \eqref{eqn:suture1} yields the sutured manifold hierarchy
\[ (M_S(s'),\gamma_S) \stackrel{S_1}{\rightsquigarrow} (M_1(s'),\gamma_1)
\stackrel{S_2}{\rightsquigarrow}  \cdots  \stackrel{S_n}{\rightsquigarrow}
 (M_n(s'),\gamma'_n)\]

Since we have assumed that $s'$ is not equal to the suture slope $s$,
the slope $s'$ does not bound a disk in the solid torus in $M_n(s') \cong
D^{2} \times S^{1}$. This means that we can further apply the sutured
manifold decomposition
\[ (M_n(s'),\gamma_n) \stackrel{\Delta}{\rightsquigarrow} (M_{n+1}(s'),\gamma_{n+1})
\]
along the meridional disk $\Delta$ to eventually get the trivial sutured
manifold $(M_{n+1}(s'),\gamma_{n+1})$. By Lemma \ref{lem:taut} this implies
that the sutured manifold $(M_S(s'),\emptyset)$ is taut, hence $S \subset
M(s')$ is norm-minimizing.
\end{proof}

\begin{remark}
By some additional arguments, as is stated and proven in \cite[Corollary
2.4]{Gabai-foliationII}, the conclusion of Theorem \ref{theorem:gabai}
remains to be true without assuming the $S_P$-atoroidality assumption\footnote{The
original assertion \cite[Theorem 1.8]{Gabai-foliationII} states a stronger
assertion on the existence of a taut foliation having $S$ as its leaf, which
requires the $S_P$-atoroidality assumption.}. 
\end{remark}

\subsection{Scharlemann-Thompson theorem and its enhancement}

We proceed to prove the fibered link enhancement of Scharlemann-Thompson's
theorem.

For the skein triple $(L_+,L_0,L_-)$ formed at the crossing $c$, let $D$
be the crossing disk at $c$, and let $K=\partial D$. Let $R=R_+$ be a
Seifert surface of $L=L_+$ contained in $S^{3} \setminus K$ that attains
the maximum Euler characteristic among such Seifert surfaces (so it can
happen that $\chi(L_+) \neq \chi(R)$). We can put $R$ so that its intersection $R\cap
D$ is a single arc $\alpha$.
By splitting $R$ along $\alpha$, we get a Seifert surface $R_0$ of $L_0$.
By $(-1)$ surgery on $K$ we get a link $L_-$, and $R$ gives rise to a
Seifert surface $R_-$ of $L_-$ (see Figure \ref{fig:cross-disk}).

\begin{figure}[htbp]
\includegraphics*[width=80mm]{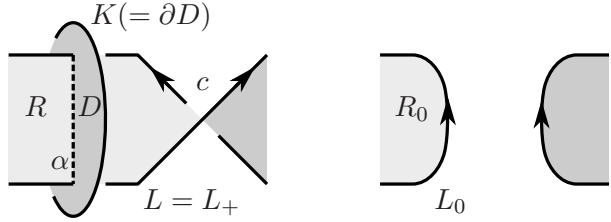}
\begin{picture}(0,0)
\put(-200,75) {\large $K (=\partial D)$}
\put(-180,5) {\large $L=L_+$}
\put(-160,50) {\large $c$}
\put(-215,20) {\large $\alpha$}
\put(-225,40) {\large $R$}
\put(-205,40) {\large $D$}
\put(-70,5) {\large $L_0$}
\put(-85,40) {\large $R_0$}
\end{picture}
\caption{Crossing disk $D$ (left) and Seifert surface $R_0$ of $L_0$ (right).}
\label{fig:cross-disk}
\end{figure}

Using these terminologies, the Scharlemann-Thompson theorem is stated in the
following manner which contains additional information.

\begin{theorem}(Scharlemann-Thompson theorem, slightly more precise version)
At least two of $(R_+,R_-,R_0)$ are norm-minimizing (i.e., attain the
maximum Euler characteristic as a Seifert surface in $S^{3}$).
\end{theorem}

To explain which can be the largest among $\{\chi(R_+),\chi(R_-),\chi(R_0)-1\}$
(i.e., which can fail to be norm-minimizing) we review the proof. 

\begin{proof}

Let $M=S^{3} \setminus N(L\cup K)$ be the complement of the link $L\cup
K$, and let $P$ be the torus boundary component of $M$ which comes from
$K$.
 By looking at the split factor of $L$ that contains the skein crossing,
in the following we always assume that $L$ is non-split and that $M$ is
irreducible.
Let $M_0,M_{-}, M_{+}(=M)$ be the $0,-1,\infty$ surgery on $K$.

When $M$ is not $R_P$-atoroidal, Scharlemann-Thompson showed that \cite[Subclaim
(a), (b)]{Scharlemann-Thompson}
\begin{equation}
\label{eqn:ST-subclaim} 
M_{\pm} \mbox{ are irreducible and } R_{\pm} \mbox{ are norm-minimizing}.
\end{equation}
Thus in the following we assume that $M$ is $R_P$-atoroidal.

In this case, the Seifert surface $R_0$ of $L_0$ is related to the surface
$R \subset M_0$ in the following manner.

\begin{claim}\cite[Claim 2]{Scharlemann-Thompson}
\label{claim-st}
Let $(M,\gamma)$ be the complementary sutured manifold for $R \subset
M_0$, and let $(M_{R_0},\gamma_{R_0})$ be the complementary sutured manifold
for $S_0 \subset S^{3}$. Then there is a product disk decomposition
\[ (M,\gamma) \stackrel{\Delta}{\rightsquigarrow} (M_{R_0},\gamma_{R_0})\]
\end{claim}

Then the assertion follows from Theorem \ref{theorem:gabai};
as we have seen, $R \subset S^{3} \setminus (L \cup \partial D)$ remains
norm-minimizing for every Dehn filling on a slope $\alpha$ of $\partial
D^{2}$ with at most one exception. The possible exception $\alpha$ is
equal to the suture slope $s$ (with respect to some sutured manifold hierarchy
\eqref{eqn:suture1}).
\end{proof}

The fibered link enhancement can be proven in a similar manner, by adding
some additional arguments on the sutured manifold hierarchy.
 
\begin{proof}[Proof of Theorem \ref{thm:st-fibered}]
(i) Assume that $\chi(L_{+})=\chi(L_0)+1 < \chi(L_{-})$ and that $L_0$
is fibered, hence $R_0$ is a fiber surface for $L_0$. 

Then $L_0$ is not split so $M_0$ is irreducible.
If $M$ is not $R_P$-atoroidal, then by \eqref{eqn:ST-subclaim} $R_{\pm}
\subset M_{\pm}$ are norm-minimizing. This contradicts the assumption
$\chi(L_{+}) < \chi(L_{-})$, so $M$ is $R_P$-atoroidal.

Let $(M_R,\gamma_R)$ be the complementary sutured manifold of $R \subset
S^{3}$. Since $\chi(L_+) < \chi(L_-)$, $R_-$ is not norm-minimizing in
$M_{-}$. Thus by Theorem \ref{thm:hierarchy} there is a sutured manifold
hierarchy
\begin{equation}
\label{eqn:s-hierarchy} (M_R,\gamma_R)\stackrel{S_1}{\rightsquigarrow}
(M_1,\gamma_1) \stackrel{S_2}{\rightsquigarrow} \cdots  \stackrel{S_n}{\rightsquigarrow}
(M_n,\gamma_n).
\end{equation}
such that its suture slope is $-1$.

On the other hand, since $R_0$ is a fiber surface of $L_0$, by Lemma \ref{lem:Gabai-product},
the complementary sutured manifold $(M_{R_0},\gamma_{R_0})$ of $R_0 \subset
S^{3}$ is a product sutured manifold. By Claim \ref{claim-st} and Lemma
\ref{lem:Gabai-product}, this implies that the complementary sutured manifold
$(M,\gamma)$ of $R \subset M_0$ is a product sutured manifold.

Let us take the taut sutured manifold hierarchy obtained from \eqref{eqn:s-hierarchy}
by $0$-surgery on each torus $P$:
\[ (M,\gamma) = (M_R(0),\gamma_R)\stackrel{S_1}{\rightsquigarrow} (M_1(0),\gamma_1)
\stackrel{S_2}{\rightsquigarrow} \cdots  \stackrel{S_n}{\rightsquigarrow}
(M_n(0),\gamma_n).\]
By Theorem \ref{thm:hierarchy} (ii), if $S_1$ is not a product disk, $S_1$
is non-separating and $-S_1$ also defines a taut sutured manifold decomposition.
Therefore by \cite[Corollary 2.7]{Gabai-fibered-link}, we have that $S_1$ is a
product annulus. Then $S_1$ is either a product disk or a product
annulus, so by Lemma \ref{lem:Gabai-product} $(M_1(0),\gamma_1)$ is a
product sutured manifold. Repeating the same argument, we conclude that
$S_i$ is either a product disk or a product annulus, and $(M_i(0),\gamma_i)$
is a product sutured manifold for all $i$.

Therefore the suture of $\gamma_n$ of $M_n=P\times[0,1]$ consists of two
curves of slope $-1$. This shows that $(M_{n}(\infty),\gamma_n)$ is also
a product sutured manifold. Let us take the taut sutured manifold hierarchy
obtained from \eqref{eqn:s-hierarchy} by $\infty$-surgery on each torus
$P$
\[ (M_R(\infty),\gamma_R)\stackrel{S_1}{\rightsquigarrow} (M_1(\infty),\gamma_1)
\stackrel{S_2}{\rightsquigarrow} \cdots  \stackrel{S_n}{\rightsquigarrow}
(M_n(\infty),\gamma_n).\]
Since we have seen that each $S_i$ is a product disk or a product annulus,
$(M_R(\infty),\gamma_R)$, the complementary sutured manifold of $R \subset
S^{3}$, is a product sutured manifold by Lemma \ref{lem:Gabai-product}.
Therefore $R$ is a fiber surface of $L=L_+$.\\

(ii)
We assume that $\chi(L_{+})=\chi(L_0)+1 < \chi(L_{-})$ and that $L_-$
is fibered with fiber surface $R_-$.
If $M$ is not $R_P$-atoroidal, then $L_-$ is a satellite link with winding
number zero. However, this contradicts the assumption that $L_-$ is fibered
because winding number zero satellite implies that the commutator subgroup
of $\pi_1(S^{3}\setminus L_-)$ contains the fundamental group of the companion
torus $\pi_1(T) \cong \Z \oplus \Z$, so it cannot be the free group. 

Thus in this case $M$ is $R_P$-atoroidal. The rest of the argument is
the same; $R_0$ being not norm-minimizing implies that we can find a sutured
manifold hierarchy of the complementary sutured manifold $(M_R,\gamma_R)$
so that the suture slope is $0$. That $R_-$ is a fiber surface implies that
each $S_i$ in the sutured manifold hierarchy is a product disk or a product
annulus, and the suture of $(M_n,\gamma_n)$ consists of two curves of
slope $0$. Hence $(M_n(\infty),\gamma_n)$ is also a product sutured manifold.
By Lemma \ref{lem:Gabai-product} this implies that $(M_R(\infty),\gamma_R)$
is a product sutured manifold, hence $R$ is a fiber surface of $L=L_+$.
\end{proof}

\renewcommand{\thesection}{Appendix \Alph{section}}
\section{Knot table\label{sec:Knot-table}}
\renewcommand{\thesection}{\Alph{section}}

\subsection{The Rolfsen knots (prime knots up to 10 crossings)}

We discuss here briefly which of the knots in Rolfsen's knot table
are s.a.p.\ or w.s.a.p. The answer is: only the obvious ones.

\begin{theorem}
Up 10 crossings, a prime knot $K$ is weakly successively almost positive
if and only if it is either positive or almost positive.
\end{theorem}
\begin{proof}
By Corollary \ref{cor:homogeneous-w.s.a.p} it is sufficient to treat non-alternating
knots.

For non-alternating knots which are not almost positive (or positive),
we find that Theorem \ref{thm:Conway-polynomial} (ii) and Theorem \ref{thm:HOMFLY-polynomial}
(i-a) shows that they are not w.s.a.p., with the following exceptions.
\begin{itemize}
\item The knot $10_{140}$. This knot is not w.s.a.p.\ since its HOMFLY polynomial
does not satisfy Theorem \ref{thm:HOMFLY-polynomial} (i-b). 
\item The knot $10_{132}$. This knot has the same HOMFLY polynomial as
the positive knot $5_1$. This knot is not w.s.a.p.\ because its signature
is $0$. (By Theorem \ref{thm:signature>0}, the signature of w.s.a.p.\ knot
is always positive.)
\end{itemize}
\end{proof}

\subsection{Prime 11 and 12 crossing knots}

As for knots up to 12 crossings, we have a result as follows.

\begin{theorem}
Up to 12 crossings, a prime knot $K$ is weakly successively almost positive
if and only if it is either positive or almost positive, possibly except
$12_{1408},\ 12_{1487},\ 12_{1837},\ 12_{2037}$.
\end{theorem}
\begin{proof}
As in the Rolfsen's knot table case, Theorem \ref{thm:Conway-polynomial}
(ii) and Theorem \ref{thm:HOMFLY-polynomial} (i-a) also deals with the
non-alternating prime 11 crossing knots whose table diagram in \cite{KnotScape}
is not positive or almost positive. (Only for $11_{379}$ we need Theorem
\ref{thm:HOMFLY-polynomial} (i-b).)

Applying the same test on non-alternating 12 crossing knots,
we find that the theorems also settle most, except 12. 

\begin{itemize}
\item Four of them, $12_{1581}$, $12_{1609}$, $12_{2038}$ and
$12_{2118}$, fall to Theorem \ref{thm:HOMFLY-polynomial} (i-d).
\item Another four, $12_{1409}$, $12_{1488}$, $12_{1811}$, $12_{1870}$,
fall to
Theorem \ref{thm:signature>0}.
\item The remaining 4 knots $12_{1408},\ 12_{1487},\ 12_{1837},\ 12_{2037}$
are the simplest prime knots for which we cannot decide yet on the w.s.a.p.
(and s.a.p.) property. 

\end{itemize}
\end{proof}

As a final note, we do not wish to imply that the decision for
composite knots easily follows from that on their factors.
Some connected sums like those in Examples \ref{exam:a2=u+sig=2-criterion}
and
\ref{exam:a2=u+sig=2-criterion1} clearly appeal to this circumstance.
There is the corresponding
Question \ref{ques:prime2}, which is unresolved.

\section*{Acknowledgement}
The first author would like to thank Kimihiko Motegi and Masakazu Teragaito
for stimulating conversations and discussions. He also thanks Keiji Tagami
for various comments, and informing of several related results and errors
in an earlier version of the paper. He was
partially supported by 
KAKENHI Grant Number 19K03490
and 21H04428.
The second author was partially funded by the
National Research Foundation
of Korea (grant NRF-2017\mb R1E1A1A0\mb 3071032)
and the International Research \& Development Program of
the Korean Ministry of Science and ICT (grant NRF-2016K1A3A7A03950702).

\printindex
\end{document}